\documentclass{amsart}
\usepackage[utf8]{inputenc}

\usepackage{geometry}
\usepackage{stmaryrd}
\usepackage{epsdice}

\usepackage[colorlinks=true,hyperindex, linkcolor=magenta, pagebackref=false, citecolor=cyan,pdfpagelabels, linktocpage=true]{hyperref}

\usepackage{enumitem}

\usepackage{color}

\usepackage[OT2,T1]{fontenc}

\usepackage[final]{microtype}

\usepackage{tikz}
\usepackage{tikz-cd}

\usepackage{colonequals}
\usepackage{mathrsfs}

\usepackage{relsize}
\usepackage[bbgreekl]{mathbbol}
\usepackage{amsfonts}
\DeclareSymbolFontAlphabet{\mathbb}{AMSb} 
\DeclareSymbolFontAlphabet{\mathbbl}{bbold}

\newcommand{\AGV}{\operatorname{AGV}}

\newcommand{\Z}{\mathbf{Z}}
\newcommand{\N}{\mathbf{N}}

\newcommand{\PP}{\mathbf{P}}

\newcommand{\rR}{\mathrm{R}}

\newcommand{\n}{\mathfrak{n}}

\newcommand{\p}{\mathfrak{p}}
\newcommand{\q}{\mathfrak{q}}

\renewcommand{\O}{\mathcal{O}}
\newcommand{\F}{\mathcal{F}}

\newcommand{\X}{\mathcal{X}}

\newcommand{\m}{\mathfrak{m}}

\newcommand{\colim}{\mathrm{colim}}

\newcommand{\Spec}{\mathrm{Spec}\,}
\newcommand{\Spa}{\mathrm{Spa}\,}

\newcommand{\h}{\mathrm{h}}

\newcommand{\Om}{\wdh{\Omega}^1}
\newcommand{\nil}{\mathrm{nil}}
\newcommand{\Sym}{\mathrm{Sym}}
\newcommand{\red}{\mathrm{red}}

\renewcommand{\o}{\otimes}

\newcommand{\et}{\operatorname{\'et}}

\renewcommand{\rm}{\mathrm}

\newcommand{\xr}{\xrightarrow}

\renewcommand{\et}{\textnormal{\'et}}
\newcommand{\an}{\mathrm{an}}
\newcommand{\cal}{\mathcal}

\renewcommand{\rm}{\mathrm}
\renewcommand{\bf}{\mathbf}
\newcommand{\ov}{\overline}
\newcommand{\ud}{\underline}
\newcommand{\wdh}{\widehat}
\renewcommand{\h}{\mathrm{h}}

\DeclareMathOperator{\charac}{char}
\DeclareMathOperator{\Frac}{Frac}

\DeclareMathOperator{\Spf}{Spf}
\DeclareMathOperator{\Min}{Min}
\DeclareMathOperator{\Rings}{Rings}
\DeclareMathOperator{\Sets}{Sets}
\DeclareMathOperator{\Idem}{Idem}

\newcommand{\abs}[1]{\lvert#1\rvert}
\newcommand{\suchthat}{\;\ifnum\currentgrouptype=16 \middle\fi\vert\;}
\newcommand\restr[2]{{\left.\kern-\nulldelimiterspace#1\vphantom{\big|}\right|_{#2}}}

\usepackage{mathtools}
\usepackage{amsmath}

\usepackage{amsthm,amssymb}
\numberwithin{equation}{section}

\theoremstyle{plain}

\newtheorem{thm}[equation]{Theorem}
\newtheorem{conjecture}[equation]{Conjecture}
\newtheorem{proposition}[equation]{Proposition}
\newtheorem{lemma}[equation]{Lemma}

\newtheorem{corollary}[equation]{Corollary}
\newtheorem{cor}[equation]{Corollary}

\newtheorem{convention}[equation]{Convention}

\theoremstyle{definition}

\newtheorem{defn}[equation]{Definition}
\newtheorem{notation}[equation]{Notation}
\newtheorem{construction}[equation]{Construction}
\newtheorem{question}[equation]{Question}

\newtheorem{rmk}[equation]{Remark}

\title{Algebraization Techniques and Rigid-Analytic Artin--Grothendieck Vanishing}

\author{Ofer Gabber}
\address{}
\email{}

\author{Bogdan Zavyalov}
\address{}
\email{}

\begin{document}

\bibliographystyle{halpha-abbrv}

\begin{abstract} First, we prove an algebraization result for rig-smooth algebras over a general noetherian ring. The result says that we can always algebraize a geometrically reduced affinoid rigid-analytic space in ``one direction'' in an appropriate sense. As an application of this result, we show the remaining cases of the Artin--Grothendieck vanishing for affinoid algebras. This allows us to deduce a stronger version of the rigid-analytic Artin--Grothendieck vanishing conjecture over a field of characteristic $0$. Using a completely different set of ideas, we also obtain a weaker version of this conjecture over a field of characteristic $p>0$. 
\end{abstract}

\maketitle
\tableofcontents

\section{Introduction}

\subsection{Partial algebraization}

In \cite{Artin-approximation}, M.~Artin pioneered the study of algebraization questions in algebraic geometry. Although algebraization problems may take different forms depending on the exact problem under consideration, usually they can be summarized by the following (imprecise) question: 

\begin{question}\label{question:algebraization} Let $A$ be a ring with a finitely generated ideal $I$, and let $\ov{S}$ be an ``algebraic structure'' over $A^{\wedge}_{I}$. Can we ``algebraize'' (or at least ``approximate'') it to an ``algebraic structure'' $S$ over $A$?
\end{question}

One positive answer to Question~\ref{question:algebraization} is provided by \cite[Th.~1.10]{Artin-approximation} and its later generalization due to Popescu (see \cite[\href{https://stacks.math.columbia.edu/tag/0AH5}{Tag 0AH5}]{stacks-project}). It says that, for a noetherian $G$-ring $A$ which is henselian along an ideal $I$, any solution $\wdh{y}\in (A^{\wedge}_I)^m$ of a (finite) system of polynomial equations $F_i\in A[X_1, \dots, X_m]$ can be approximated by a solution $y\in A^m$. \smallskip

This result has become a standard (but very important) technical tool in many areas of algebraic geometry: various algebraization questions (see \cite[Cor.~2.6 and \textsection 3]{Artin-approximation}), structure theory of algebraic stacks (see \cite{Luna}), homological and commutative algebra (see \cite{tight-closure}), construction of moduli spaces (see \cite{Artin-moduli}, \cite{moduli}), and \'etale cohomology (see \cite{deGabber}). \smallskip

Using the above approximation (and its version due to Elkik), we prove the following general algebraization result: 

\begin{thm}[Noetherian rig-smooth algebraization; Theorem~\ref{thm:noetherian-approximation}]\label{thm:intro-artin-algebraization} Let $A$ be a noetherian ring, and let $I\subset A$ be an ideal. Let $B$ be an $I$-adically complete $A$-algebra such that $B$ is rig-smooth over $(A, I)$ (in the sense of \cite[\href{https://stacks.math.columbia.edu/tag/0GAI}{Tag 0GAI}]{stacks-project}) and $B/IB$ is a finite type $A/I$-algebra. Then there is a finite type $A$-algebra $C$ such that $C$ is smooth outside $\rm{V}(I)$ and there is an isomorphism $C^{\wedge}_I \simeq B$ of $A$-algebras.
\end{thm}

Theorem~\ref{thm:intro-artin-algebraization} positively answers the question raised in \cite[\href{https://stacks.math.columbia.edu/tag/0GAX}{Tag 0GAX}]{stacks-project}. Of course, many special cases of Theorem~\ref{thm:intro-artin-algebraization} were known before. For instance, the case when $A$ is a noetherian $G$-ring was settled in \cite[\href{https://stacks.math.columbia.edu/tag/0GAT}{Tag 0GAT}]{stacks-project}, and the case when $I$ is a principal ideal was settled in \cite[Th.~7 on p.~582]{Elkik}. \smallskip

Theorem~\ref{thm:intro-artin-algebraization} provides a general tool for studying rig-smooth formal schemes. For instance, a version of Theorem~\ref{thm:intro-artin-algebraization} for rig-\'etale morphisms is used in the generalization of Artin's Theorem on dilatations (see \cite[\href{https://stacks.math.columbia.edu/tag/0ARB}{Tag 0ARB}]{stacks-project} and \cite[\href{https://stacks.math.columbia.edu/tag/0GDU}{Tag 0GDU}]{stacks-project}, see also \cite[Th.~3.2]{Artin-2}  for the original result), while Artin's Theorem on dilatations is used in the proof of Artin's Theorem on contractions (see \cite[\href{https://stacks.math.columbia.edu/tag/0GH7}{Tag 0GH7}]{stacks-project} and \cite[Th.~3.1]{Artin-2}). \smallskip

However, the noetherian assumption in Theorem~\ref{thm:intro-artin-algebraization} is too limiting for the purposes of rigid-analytic geometry. Namely, one would like to apply this result to $A=\O_C$ for a rank-$1$ valuation ring in an algebraically closed non-archimedean field $C$. This ring is never noetherian, so Theorem~\ref{thm:intro-artin-algebraization} does not apply in this situation. Nevertheless, this issue has been overcome by R.~Elkik, who studied an analogue of Artin approximation in certain non-noetherian situations. In particular, she proved the following remarkable fact:

\begin{thm}[Elkik's algebraization; special case of {\cite[Th.~7 on p.~582 and Rmq.~2(c) on p.588]{Elkik}}]\label{thm:intro-elkik-algebraization} Let $K$ be a non-archimedean field with a pseudo-uniformizer $\varpi \in \O_K$, and let $A$ be a rig-smooth\footnote{See \cite[Prop.~3.3.2]{Tem08} for the fact that a topologically finitely presented map $\O_K \to A$ is formally smooth ``outside $\rm{V}(\varpi)$'' in the sense of \cite[p.581]{Elkik} if and only if $\O_K \to A$ is rig-smooth in the sense of \cite[Def.~3.1]{BLR4}. Strictly speaking, \cite[Prop.~3.3.2]{Tem08} is written under the additional assumption that $K$ is discretely valued, but the proof goes through for a general non-archimedean field $K$.}, flat, topologically finite type $\O_K$-algebra. Then there is a flat, finite type $\O_K$-algebra $B$ such that $B_K$ is $K$-smooth and there exists an isomorphism $B^{\wedge}_{(\varpi)} \simeq A$ of $\O_K$-algebras. 
\end{thm}

Theorem~\ref{thm:intro-elkik-algebraization} provides a very general machinery to reduce certain (local) questions about smooth rigid-analytic spaces to analogous questions about smooth algebraic varieties. For example, this approach has been used in \cite{berkovich-finiteness} to study finiteness of \'etale cohomology of rigid-analytic spaces, in \cite{uniformization} to study questions related to semi-stable reduction, and in \cite[\textsection 2.5]{thesis} to study properties of dualizing modules on rig-smooth admissible formal schemes. \smallskip 

However, Theorem~\ref{thm:intro-elkik-algebraization} is still not strong enough to reduce questions about {\it singular} rigid-analytic spaces to analogous questions about singular algebraic varieties. Furthermore, the naive analogue of Theorem~\ref{thm:intro-elkik-algebraization} is false if one does not impose any smoothness assumptions on $A$ (see Appendix~\ref{appendix:counter-examples}). One of the key results of this paper is that it is nevertheless possible to algebraize a geometrically reduced affinoid rigid-analytic space in ``one direction'' under a very mild assumption. \smallskip

Before we formulate the precise result, we recall that $\Spa(A, A^\circ)$ stands for the adic spectrum in the sense of Huber, see \cite[\textsection 3]{H0}. Furthermore, for a $K$-affinoid algebra $A$, the adic space $X=\Spa(A, A^\circ)$ is called \emph{geometrically reduced} if $X \times_{\Spa(K, \O_K)} \Spa(L, \O_L)$ is reduced for any extension of non-archimedean fields $K\subset L$; see \cite[\textsection 3.3]{Conrad99} for details. We warn the reader that \cite[p.~509]{Conrad99} provides an example of a non-archimedean field $K$ and a geometrically reduced (in the usual algebraic sense, see \cite[\href{https://stacks.math.columbia.edu/tag/030S}{Tag 030S}]{stacks-project}) $K$-affinoid algebra $A$ such that $\Spa(A, A^\circ)$ is \emph{not} geometrically reduced.

\begin{thm}[Partial algebraization; Corollary~\ref{cor:partial-algebraization-1}]\label{thm:intro-partial-algebraization}  Let $K$ be a non-archimedean field, let $\varpi \in \O_K$ be a pseudo-uniformizer, and let $A_0$ be a flat, topologically finite type $\O_K$-algebra such that $A\coloneqq A_0[\frac{1}{\varpi}]$ is a $K$-affinoid algebra of Krull dimension $d>0$ and $\Spa(A, A^\circ)$ is geometrically reduced. Put $R=\O_K\langle X_1, \dots, X_{d-1}\rangle[X_d]$. Then there is a finite, finitely presented $R^{\rm{h}}_{(\varpi)}$-algebra $B$ with an isomorphism $B\otimes_{R^{\rm{h}}_{(\varpi)}} R^{\wedge}_{(\varpi)} \simeq A_0$ of $\O_K$-algebras.
\end{thm}

\begin{cor}[Partial algebraization II; Corollary~\ref{cor:partial-algebraization-2}]\label{cor:intro-partial-algebraization-2} Let $K$, $\varpi$, and $A_0$ be as in Theorem~\ref{thm:intro-partial-algebraization}. Then there is a finitely presented, quasi-finite morphism $\O_K\langle X_1, \dots, X_{d-1}\rangle[X_d] \to B$ and an isomorphism $B^{\wedge}_{(\varpi)}  \simeq A_0$ of $\O_K$-algebras.
\end{cor}

We note that Theorem~\ref{thm:intro-partial-algebraization} and Corollary~\ref{cor:intro-partial-algebraization-2} are optimal. Namely, the assumption that $X=\Spa(A, A^\circ)$ is geometrically reduced in the sense of \cite[p.~509]{Conrad99} cannot be dropped in Theorem~\ref{thm:intro-partial-algebraization} and Corollary~\ref{cor:intro-partial-algebraization-2} (see also Remark~\ref{rmk:optimal}). We refer to Proposition~\ref{prop:counterexample-1-appendex} and Proposition~\ref{prop:counterexample-3-appendex} for explicit counterexamples. \smallskip

Unlike Theorem~\ref{thm:intro-elkik-algebraization}, Theorem~\ref{thm:intro-partial-algebraization} does not allow to directly reduce questions about geometrically reduced (affinoid) rigid-analytic spaces to similar questions about general (affine) algebraic varieties. However, it does give a quite robust tool of reducing problems about geometrically reduced (affinoid) rigid-analytic spaces of dimension $d$ to algebro-geometric problems for (algebraic) curves over $\Spec K\langle T_1,\dots, T_{d-1}\rangle$. We implement this strategy to prove the remaining cases of Artin--Grothendieck vanishing for affinoid algebras, we discuss this proof in more detail in the next section. We also expect this strategy to be useful for other problems in rigid-analytic geometry. For example, the first author knows how to use Corollary~\ref{cor:intro-partial-algebraization-2} to prove \cite[Theorem 1.3.5]{Zav-thesis} without using any input from \cite{BS3} or \cite{Sch2}.

\subsection{Artin--Grothendieck vanishing}

In this paper, we apply the partial algebraization techniques to study Artin--Grothendieck vanishing for affinoid algebras. To put things into context, we recall that the theory of \'etale cohomology of schemes has been extensively studied in \cite{SGA4}, \cite{SGA41/2}, and \cite{SGA5}. Among the many important results obtained there, Artin and Grothendieck have proven the following celebrated vanishing result:

\begin{thm}[Artin--Grothendieck vanishing; {\label{thm:intro-AGV} \cite[Exp.\,XIV, Cor.\,3.2]{SGA4}}] Let $k$ be a separably closed field, and let $A$ be a finite type $k$-algebra. Then, for any torsion \'etale sheaf $\F$ on $\Spec A$, we have 
\[
\rm{H}^i_\et\big(\Spec A, \F\big)=0
\]
for $i> \dim A$. 
\end{thm}


In this paper, we discuss two possible analogues of Theorem~\ref{thm:intro-AGV} in the rigid-analytic world. The first analogue is a version of Artin--Grothendieck vanishing for affinoid algebras, which was conjectured in \cite[Remark 7.4]{BM21}. We use the partial algebraization techniques to prove this conjecture in full generality: 

\begin{thm}[Artin--Grothendieck vanishing for affinoid algebras; Theorem~\ref{thm:main-theorem}]\label{thm:intro-schematic-AGV} Let $C$ be an algebraically closed non-archimedean field, and let $A$ be a $C$-affinoid algebra. Then, for any torsion \'etale sheaf $\F$ on $\Spec A$, we have 
\[
\rm{H}^i_\et\big(\Spec A, \F\big)=0
\]
for $i> \dim A$. 
\end{thm}

\begin{rmk} We recall that \cite[Prop.\,3.4.1/6]{BGR} and our convention that non-archimedean fields are complete and non-trivally valued imply that any separably closed non-archimedean field is algebraically closed. Therefore, there is no extra generality in considering separably closed non-archimedean fields in the formulation of Theorem~\ref{thm:intro-schematic-AGV}. 
\end{rmk}

Theorem~\ref{thm:intro-schematic-AGV} has been previously established under certain additional assumptions on $\F$:  

\begin{enumerate}
    \item For a non-archimedean field $C$ of characteristic $p>0$ and a sheaf $\F$ of $\Z/p\Z$-modules, the conclusion of Theorem~\ref{thm:intro-schematic-AGV} follows from \cite[Exp.\,X, Th.\,5.1]{SGA4};
        \item Using the arc-topology, Bhatt and Mathew proved Theorem~\ref{thm:intro-schematic-AGV} for sheaves of $\Z/n\Z$-modules when $n$ is invertible in $\O_C$ (see \cite[Th.\,7.3]{BM21});
    \item In general, \cite[Th.\,7.3]{BM21} shows that $\rm{H}^i_\et\big(\Spec A, \F\big)=0$ for any torsion \'etale sheaf $\F$ and $i>\dim A +1$. 
\end{enumerate}

Therefore, the main new contribution of our work is the case when $C$ is a non-archimedean field of mixed characteristic $(0, p)$ and $\F$ is a sheaf of $\Z/p\Z$-modules. Nevertheless, our proof of Theorem~\ref{thm:intro-schematic-AGV} is independent of \cite{BM21} and works uniformly for any ground field $C$ and any sheaf of $\Z/n\Z$-modules $\F$ (as long as $n$ is invertible in $C$).

The other possible rigid-analytic analogue of Theorem~\ref{thm:intro-AGV} was conjectured by Hansen: 

\begin{conjecture}[{\cite[Conj.~1.2]{Hansen-vanishing}}]\label{conjecture:RAAGV} Let $C$ be an algebraically closed non-archimedean field, let $A$ be an $C$-affinoid algebra, and let $n$ be an integer invertible in $\O_C$. Then, for any Zariski-constructible sheaf $\F$ (see Definition~\ref{defn:zariski-constructible}) of $\Z/n\Z$-modules on $\Spa(A, A^\circ)_\et$, we have 
\[
\rm{H}^i_\et\big(\Spa(A, A^\circ), \F\big)=0
\]
for $i> \dim A$.   
\end{conjecture}

Conjecture~\ref{conjecture:RAAGV} is known when $\charac C=0$ due to \cite[Th.~7.3]{BM21} and \cite[Th.~1.10]{Hansen-vanishing}. In this paper, we prove a {\it stronger} version of this conjecture, allowing for more general $n$, when $\charac C = 0$ or $A$ is of dimension $1$. We also establish some {\it weaker} statements when $\charac C=p>0$. We formulate the results of our paper more precisely below.  

\begin{thm}[Rigid-analytic Artin--Grothendieck vanishing; Theorem~\ref{thm:main-theorem-char-0}]\label{thm:intro-main-theorem-char-0} Let $C$ be an algebraically closed non-archimedean field, let $A$ be an affinoid $C$-algebra, and let $\F$ be an algebraic torsion \'etale sheaf on $\Spa(A, A^\circ)$ (in the sense of Definition~\ref{defn:algebraizable}). Then 
\begin{equation}\label{eqn:intro-AGV}
\rm{H}^i_{\et}\big(\Spa(A, A^\circ), \F\big) = 0
\end{equation}
for $i>\dim A$. In particular, (\ref{eqn:intro-AGV}) holds in either of the following situations:
\begin{enumerate}
    \item $\F$ is a lisse sheaf of $\Z/n\Z$-modules for some integer $n>0$;
    \item $\charac C =0$ and $\F$ is a Zariski-constructible sheaf of $\Z/n\Z$-modules for some integer $n>0$.
\end{enumerate}
\end{thm}

We emphasize that Theorem~\ref{thm:intro-main-theorem-char-0} has no assumptions on the integer $n$. In particular, it applies to Zariski-constructible $\bf{F}_p$-sheaves on an affinoid space over a field $C$ of mixed characteristic $(0, p)$. Unfortunately, Theorem~\ref{thm:intro-main-theorem-char-0} does not imply the full version of Conjecture~\ref{conjecture:RAAGV} over a field of characteristic $p>0$ due to the existence of non-algebraic constructible sheaves. Nevertheless, it recovers all previously-known versions of Artin--Grothendieck vanishing in rigid-analytic geometry. We mention them below:  

\begin{enumerate}
    \item In \cite[Th.\,6.1]{Berkovich-vanishing}, Berkovich treats the case when $\F$ is a sheaf of $\Z/n\Z$-modules, $n$ is invertible in $\O_C$, and both $A$ and $\F$ are ``algebraizable'' in a certain sense;
    \item In \cite[Th.\,1.4]{Hansen-vanishing}, Hansen treats the case when $n$ is invertible in $\O_C$, $\F=\ud{\Z/n\Z}$ is the constant sheaf, and $\Spa(A, A^\circ)$ is defined over a discretely valued non-archimedean field $K\subset C$;
    \item In \cite[Th.\,1.3]{Hansen-vanishing}, Hansen also treats the case when $\charac C=0$, $n$ is invertible in $\O_C$, $\F$ is a Zariski-constructible sheaf of $\Z/n\Z$-modules, and both $\Spa(A, A^\circ)$ and $\F$ are defined over a discretely valued non-archimedean field $K\subset C$;
    \item Using \cite[Th.\,7.3]{BM21}, one can easily deduce Theorem~\ref{thm:intro-main-theorem-char-0} under the additional hypothesis that $n$ is invertible in $\O_C$ and $\F$ is a sheaf of $\Z/n\Z$-modules.  
\end{enumerate}

We also establish a weaker version of Conjecture~\ref{conjecture:RAAGV} that applies to all Zariski-constructible sheaves in all characteristics. Unlike Theorem~\ref{thm:intro-main-theorem-char-0}, this result requires a completely new set of ideas and our proof does not use any partial algebraization techniques.

\begin{thm}[Rigid-analytic Artin--Grothendieck vanishing in top degree; Theorem~\ref{thm:main-theorem-top-degree}]\label{thm:intro-main-theorem-top-degree} Let $C$ be an algebraically closed non-archimedean field, let $A$ be a $C$-affinoid algebra of dimension $d\geq 1$, let $n$ be an integer invertible in $\O_C$, and let $\F$ be a Zariski-constructible sheaf of $\Z/n\Z$-modules on $\Spa(A, A^\circ)_\et$. Then 
\[
\rm{H}^{2d}_\et\big(\Spa(A, A^\circ), \F\big) = 0.
\]
\end{thm}

As far as we are aware, Theorem~\ref{thm:intro-main-theorem-top-degree} is the first progress towards Conjecture~\ref{conjecture:RAAGV} for non-algebraic Zariski-constructible sheaves. 

We also note that Theorem~\ref{thm:intro-main-theorem-top-degree} and Theorem~\ref{thm:intro-main-theorem-char-0} prove a stronger version of Conjecture~\ref{conjecture:RAAGV}, allowing for any $n$ invertible in $C$, when $\Spa(A, A^\circ)$ is a curve; see Corollary~\ref{cor:curves} for more details.

\subsection{Acknowledgements} BZ is grateful to Bhargav Bhatt, Brian Conrad, Peter Haine, and Peter Scholze for useful discussions and for their valuable comments on the first draft of this paper. A major part of this work has been written when BZ was visiting Institut des Hautes \'Etudes Scientifiques in June, 2023; BZ would like to thank IHES for their hospitality. OG was supported by the IHES and CNRS. Finally, the authors would like to thank the Institute for Advanced Study for funding and excellent working conditions.

\subsection{Conventions}

Throughout this paper, we follow \cite{H0} and use the multiplicative convention for valuations. A {\it non-archimedean field} is a complete topological field $K$ whose topology is defined by a nontrivial rank-$1$ valuation $\abs{.}\colon K \to \bf{R}_{\geq 0}$.  We say that a $K$-algebra $A$ is {\it $K$-affinoid} if there is surjection $K\langle X_1, \dots, X_d\rangle \twoheadrightarrow A$ for some $d\geq 0$.  

For a commutative ring $A$ and an ideal $I \subset A$ (resp.\,a finitely generated ideal $I$), we denote by $A^\h_{I}$ (resp.\,$A^{\wedge}_I$) the {\it $I$-adic henselization} (resp.\,{\it $I$-adic completion}) of $A$. We refer to \cite[\href{https://stacks.math.columbia.edu/tag/0EM7}{Tag 0EM7}]{stacks-project} and \cite[\href{https://stacks.math.columbia.edu/tag/00M9}{Tag 00M9}]{stacks-project} for a detailed discussion of these notions.  

We say that an $A$-module $M$ is {\it $I$-adically complete} if the natural map $M \to \lim_n M/I^nM$ is an isomorphism, i.e., $M$ is complete and separated in the $I$-adic topology. 

In this paper, we will crucially use the notion of henselian pairs, we refer the reader to \cite[\href{https://stacks.math.columbia.edu/tag/09XD}{Tag 09XD}]{stacks-project} for the definition and detailed discussion of this notion. We also use the notion of rig-smoothness (over a pair $(A, I)$ of a noetherian ring $A$ and an ideal $I\subset A$) as defined in \cite[\href{https://stacks.math.columbia.edu/tag/0GAI}{Tag 0GAI}]{stacks-project}; this definition is equivalent to the one from \cite[\textsection 3]{BLR3}.

\section{Partial algebraization}

\subsection{Weierstrass polynomials}

Throughout this section, we fix a non-archimedean field $K$ with ring of integers $\O_K$ and a pseudo-uniformizer $\varpi \in \O_K$. We denote the unique maximal ideal of $\O_K$ by $\m_K$.  

The main goal of this section is to recall the notion of a Weierstrass polynomial and study its properties. The main result of this section is Corollary~\ref{cor:henselizations-agree} which will be crucial for our proof of the rigid-analytic Artin--Grothendieck vanishing Theorem.  

\begin{notation} Throughout this paper, we denote the ring $\O_K\langle X_1, \dots, X_{d-1}\rangle[X_d]$ by $R$. 
\end{notation}


\begin{defn} An element $f\in \O_K\langle X_1, \dots, X_{d-1}\rangle[X_d]\subset \O_K\langle X_1, \dots, X_d\rangle$ is a {\it Weierstrass polynomial of degree $s$} if $f=X_d^s+ \sum_{i=0}^{s-1} g_i(X_1, \dots, X_{d-1})X_d^{i}$ for some $g_i\in \O_K\langle X_1, \dots, X_{d-1}\rangle$.
\end{defn}

For the next construction, we fix an element $f\in R$. We also note that \cite[\href{https://stacks.math.columbia.edu/tag/0DYD}{Tag 0DYD}]{stacks-project} ensures that the ring $R^{\rm{h}}_{(\varpi)}$ is $(\varpi f)$-adically henselian. Furthermore, \cite[\href{https://stacks.math.columbia.edu/tag/05GG}{Tag 05GG}]{stacks-project} and \cite[\href{https://stacks.math.columbia.edu/tag/090T}{Tag 090T}]{stacks-project} imply that the rings $R^{\wedge}_{(\varpi f)}$ and $R^{\wedge}_{(\varpi)}$ are $(\varpi f)$-adically complete. In particular, they are also $(\varpi f)$-adically henselian by \cite[\href{https://stacks.math.columbia.edu/tag/0ALJ}{Tag 0ALJ}]{stacks-project}. 

\begin{construction} In the notation as above, the universal property of henselization (see \cite[\href{https://stacks.math.columbia.edu/tag/0A02}{Tag 0A02}]{stacks-project}) and the universal property of completion imply that there are unique $R$-algebra homomorphisms
\[
\alpha_f\colon R^{\rm{h}}_{(\varpi f)} \to R^{\rm{h}}_{(\varpi)}, 
\]
\[
\beta_f\colon R^{\wedge}_{(\varpi f)} \to R^{\wedge}_{(\varpi)} 
\]
such that the diagram 
\[
\begin{tikzcd}
R \arrow{r} \arrow{rd} & R^{\rm{h}}_{(\varpi f)}  \arrow{d}{\alpha_f} \arrow{r} & R^{\wedge}_{(\varpi f)}\arrow{d}{\beta_f} \\
& R^{\rm{h}}_{(\varpi)}  \arrow{r} & R^{\wedge}_{(\varpi)}
\end{tikzcd}
\]
commutes, where the vertical arrows are the natural ones. 
\end{construction}

The main goal of this section is to show that the morphisms $\alpha_f$ and $\beta_f$ are isomorphisms when $f$ is a {\it Weierstrass polynomial}. For this, we will need a number of preliminary lemmas:

\begin{lemma}\label{lemma:mod-p-invert-p} Let $M\in D(\O_K)$ be an object of the derived category. Then $M=0$ if and only if $M_K=0$ and $M\otimes^L_{\O_K} \O_K/(\varpi)=0$.
\end{lemma}
\begin{proof}
    The ``only if'' direction is clear, so we only need to show the ``if'' direction. We assume that $M_K=0$ and $M\otimes^L_{\O_K} \O_K/(\varpi)=0$. The latter assumption implies that the natural morphism $M \xrightarrow{\cdot \varpi} M$ is an isomorphism. Therefore, we conclude that $\rm{H}^i(M) \simeq \rm{H}^i(M_K) =0$ for any integer $i$. Thus, $M=0$.
\end{proof}

\begin{lemma}\label{lemma:uncompletion-Weierstrass-poly} Let $f\in R$ be a Weierstrass polynomial of degree $s$. Then the natural morphism
\[
\gamma\colon R/(\varpi f)\to R^{\wedge}_{(\varpi)}/(\varpi f) 
\]
is an isomorphism.
\end{lemma}
\begin{proof}
    First, we note that $R^{\wedge}_{(\varpi)}$ is simply the integral Tate algebra $\O_K\langle X_1, \dots, X_d\rangle$. Therefore, we need to show that the natural morphism
    \[
    \gamma\colon \O_K\langle X_1, \dots, X_{d-1}\rangle[X_d]/(\varpi f)\to \O_K\langle X_1, \dots, X_{d}\rangle/(\varpi f) 
    \]
    is an isomorphism. Lemma~\ref{lemma:mod-p-invert-p} ensures that it suffices to show that $\gamma_K$ and $\gamma\otimes^L_{\O_K} \O_K/(\varpi)$ are isomorphisms.  
    
    {\it Step~$1$. $\gamma_K$ is an isomorphism.} Since $\varpi$ is invertible in $K\langle X_1, \dots, X_{d-1}\rangle[X_d]$, it suffices to show that the natural morphism
    \[
    K\langle X_1, \dots, X_{d-1}\rangle[X_d]/(f) \to K\langle X_1, \dots, X_{d}\rangle/(f)
    \]
    is an isomorphism. This directly follows from \cite[Prop.~5.2.3/3(ii)]{BGR}.  

    {\it Step~$2$. $\gamma\otimes^L_{\O_K} \O_K/(\varpi)$ is an isomorphism.} We note that both $R$ and $R^{\wedge}_{(\varpi)}$ are domains (because they can be realized as subrings of $\O_K\llbracket X_1,\dots, X_d\rrbracket$), so 
    \[
    R/(\varpi) = R\otimes^L_{\O_K} \O_K/(\varpi) \quad \text{and} \quad R^{\wedge}_{(\varpi)}/(\varpi f)=R^{\wedge}_{(\varpi)}\otimes^L_{R} R/(\varpi f).
    \]
    Therefore, $\gamma\otimes^L_{\O_K} \O_K/(\varpi)$ can be identified with the morphism
    \[
    R/(\varpi f)\otimes^L_R R/(\varpi) \to R^{\wedge}_{(\varpi)}\otimes_R^L R/(\varpi f)\otimes^L_R R/(\varpi).
    \]
    But then it suffices to show that $R/(\varpi)\to R^{\wedge}_{(\varpi)}\otimes^L_R R/(\varpi)$ is an equivalence. Since $R^{\wedge}_{(\varpi)}$ is a domain (so $\varpi \in R^{\wedge}_{(\varpi)}$ is a regular element), the question reduces to showing that the natural morphism
    \[
    R/(\varpi) \to R^{\wedge}_{(\varpi)}/(\varpi)
    \]
    is an isomorphism. This follows directly from \cite[\href{https://stacks.math.columbia.edu/tag/05GG}{Tag 05GG}]{stacks-project}.
\end{proof}

\begin{cor}\label{cor:p-henselian} Let $f\in R$ be a Weierstrass polynomial. Then $R^{\rm{h}}_{(\varpi f)}$ is $(\varpi)$-adically henselian.
\end{cor}
\begin{proof}
    First, we note that \cite[\href{https://stacks.math.columbia.edu/tag/0DYD}{Tag 0DYD}]{stacks-project} and \cite[\href{https://stacks.math.columbia.edu/tag/0AGU}{Tag 0AGU}]{stacks-project} imply that it suffices to show that 
    \[
    R^{\rm{h}}_{(\varpi f)}/(\varpi f)= R/(\varpi f)
    \]
    is $(\varpi)$-adically henselian. Now Lemma~\ref{lemma:uncompletion-Weierstrass-poly} implies that
    \[
    R/(\varpi f) \simeq R^{\wedge}_{(\varpi)}/(\varpi f).
    \]
    Now note that $R^{\wedge}_{(\varpi)}$ is $(\varpi)$-adically henselian because it is $(\varpi)$-adically complete. Thus, its quotient  $R^{\wedge}_{(\varpi)}/(\varpi f)$ is also $(\varpi)$-adically henselian due to \cite[\href{https://stacks.math.columbia.edu/tag/09XK}{Tag 09XK}]{stacks-project}.
\end{proof}

\begin{cor}\label{cor:henselizations-agree} Let $f\in R$ be a Weierstrass polynomial. Then the natural morphisms 
\[
\alpha_f\colon R^{\rm{h}}_{(\varpi f)} \to R^{\rm{h}}_{(\varpi)},
\]
\[
\beta_f\colon R^{\wedge}_{(\varpi f)} \to R^{\wedge}_{(\varpi)}
\]
are isomorphisms.
\end{cor}
\begin{proof}
    We first show that $\alpha_f$ is an isomorphism. The universal property of  $R^{\rm{h}}_{(\varpi)}$ and the fact that $R^{\rm{h}}_{(\varpi f)}$ is $(\varpi)$-adically henselian (see Corollary~\ref{cor:p-henselian}) allow us to define a canonical $R$-algebra homorphism $\alpha'_f\colon R^{\rm{h}}_{(\varpi)}\to R^{\rm{h}}_{(\varpi f)}$. Therefore, it suffices to show that $\alpha_f\circ \alpha'_f=\rm{id}$  and $\alpha'_f\circ \alpha_f = \rm{id}$, but this also easily follows from the universal properties of $R^{\rm{h}}_{(\varpi)}$ and $R^{\rm{h}}_{(\varpi f)}$ respectively.   

    Now we show that $\beta_f$ is an isomorphism. Since both $R^{\wedge}_{(\varpi f)}$ and $R^{\wedge}_{(\varpi)}$ are $(\varpi f)$-adically complete, it suffices to show that the natural morphism
    \[
    R^{\wedge}_{(\varpi f)}/(\varpi f)^n = R/(\varpi^nf^n) \to R^{\wedge}_{(\varpi)}/(\varpi^nf^n)
    \]
    is an isomorphism for any $n>0$. This follows directly from Lemma~\ref{lemma:uncompletion-Weierstrass-poly} and the observation that $f^n$ is a Weierstrass polynomial for any $n>0$.
\end{proof}

\subsection{Algebraization techniques}

Throughout this section, we fix the notation of the previous section. In particular, $K$ is a non-archimedean field with ring of integers $\O_K$, and a pseudo-uniformizer $\varpi \in \O_K$. Recall that $R$ denotes the ring $\O_K\langle X_1, \dots, X_{d-1}\rangle [X_d]$.  

The key result of this section is Theorem~\ref{thm:partial-algebraization}. Combined with a generically \'etale version of Noetherian normalization (see Theorem~\ref{thm:generically-etale-normalization}), it will be the crucial ingredient in our proof of the rigid-analytic Artin--Grothendieck vanishing.  

Besides this, we also recall the results of Elkik and Fujiwara--Gabber (see \cite{Elkik} and \cite{Fujiwara95}) in the non-noetherian context. These results are certainly well-known to the experts, but we prefer to spell them out explicitly since some care is nedeed in this non-noetherian situation.  

We first discuss a version of Elkik's algebraization that will be relevant for our purposes.

\begin{lemma}\label{lemma:completion-faithfully-flat} 
\begin{enumerate}
\item\label{lemma:completion-faithfully-flat-0} The ring $R^\h_{(\varpi)}$ is $(\varpi)$-adically universally adhesive (in the sense of \cite[Def.\,7.1.3]{FGK});
\item\label{lemma:completion-faithfully-flat-1} The natural morphism $R^{\rm{h}}_{(\varpi)}\to R^{\wedge}_{(\varpi)}$ is faithfully flat;
\item\label{lemma:completion-faithfully-flat-2} If $f\in R$ is a Weierstrass polynomial, then the natural morphism $R^{\rm{h}}_{(\varpi f)} \to R^{\wedge}_{(\varpi f)}$ is faithfully flat.
\end{enumerate}
\end{lemma}
\begin{proof}
    (\ref{lemma:completion-faithfully-flat-0}) and (\ref{lemma:completion-faithfully-flat-1}) follow directly from Lemma~\ref{lemma:completion-faithfully-flat-appendix}.  Then (\ref{lemma:completion-faithfully-flat-2}) follows from (\ref{lemma:completion-faithfully-flat-1}) and Corollary~\ref{cor:henselizations-agree}.
\end{proof}

We recall that an $A$-algebra $A'$ is called finite, finitely presented if it is finitely generated as an $A$-module and finitely presented as an $A$-algebra.

\begin{lemma}[Elkik's algebraization]\label{lemma:elkik-approximation} Let $f\in R$ be a Weierstrass polynomial and $B$ a finite, finitely presented $R^{\wedge}_{(\varpi f)}$-algebra such that $B\big[\frac{1}{\varpi f}\big]$ is \'etale over $R^{\wedge}_{(\varpi f)}\big[\frac{1}{\varpi f}\big]$. Then there is a finite, finitely presented $R^{\rm{h}}_{(\varpi f)}$-algebra $\widetilde{B}$ such that $\widetilde{B}\otimes_{R^{\rm{h}}_{(\varpi f)}}  R^{\wedge}_{(\varpi f)} \simeq B$. 
\end{lemma}
If $R$ were a noetherian algebra, this would automatically follow from \cite[Th.\,5]{Elkik}.
\begin{proof}
    First, we note that \cite[Prop.\,5.4.53]{GR} implies that there is a finite \'etale $R^{\rm{h}}_{(\varpi f)}\big[\frac{1}{\varpi f}\big]$-algebra $\widetilde{B}'$ such that 
    \[
    \widetilde{B}'\otimes_{R^{\rm{h}}_{(\varpi f)}[\frac{1}{\varpi f}]} R^{\wedge}_{(\varpi f)}\Big[\frac{1}{\varpi f}\Big] \simeq B\Big[\frac{1}{\varpi f}\Big].
    \]
    Now Lemma~\ref{lemma:completion-faithfully-flat}(\ref{lemma:completion-faithfully-flat-2}) implies that $R^{\rm{h}}_{(\varpi f)}\to R^{\wedge}_{(\varpi f)}$ is faithfully flat. Thus formal gluing (see \cite[\href{https://stacks.math.columbia.edu/tag/05ES}{Tag 05ES}]{stacks-project} and \cite[\href{https://stacks.math.columbia.edu/tag/05EU}{Tag 05EU}]{stacks-project}) states that we can find an $R^{\rm{h}}_{(\varpi f)}$-algebra $\widetilde{B}$ such that 
    \[
    \widetilde{B}\otimes_{R^{\rm{h}}_{(\varpi f)}} R^{\wedge}_{(\varpi f)} \simeq B, \, \widetilde{B}\Big[\frac{1}{\varpi f}\Big] \simeq \widetilde{B}'. 
    \]
    Since $R^{\rm{h}}_{(\varpi f)}\to R^{\wedge}_{(\varpi f)}$ is faithfully flat, faithfully flat descent (see \cite[\href{https://stacks.math.columbia.edu/tag/03C4}{Tag 03C4}]{stacks-project} and \cite[\href{https://stacks.math.columbia.edu/tag/00QQ}{Tag 00QQ}]{stacks-project}) implies that $\widetilde{B}$ is a finite, finitely presented $R^{\rm{h}}_{(\varpi f)}$-algebra. 
\end{proof}

\begin{thm}\label{thm:partial-algebraization} Let $f\in R$ be a Weierstrass polynomial and $B$ a finite, finitely presented $R^{\wedge}_{(\varpi)}=\O_K\langle T_1, \dots, T_d\rangle$-algebra. If $B[\frac{1}{\varpi f}]$ is \'etale over $R^{\wedge}_{(\varpi)}[\frac{1}{\varpi f}]=K\langle T_1, \dots, T_d\rangle[\frac{1}{f}]$, then there is a finite, finitely presented $R^{\rm{h}}_{(\varpi)}$-algebra $\widetilde{B}$ such that
\[
\widetilde{B}\otimes_{R^{\rm{h}}_{(\varpi)}} R^{\wedge}_{(\varpi)} \simeq B.
\]

\end{thm}
\begin{proof}
    Corollary~\ref{cor:henselizations-agree} implies that $R^{\wedge}_{(\varpi)}$ (resp. $R^{\rm{h}}_{(\varpi)}$) is canonically isomorphic to $R^{\wedge}_{(\varpi f)}$ (resp. $R^{\rm{h}}_{(\varpi f)}$) as an $R$-algebra. Therefore, Lemma~\ref{lemma:elkik-approximation} provides us with a finite, finitely presented $R^{\rm{h}}_{(\varpi f)}\simeq R^{\rm{h}}_{(\varpi)}$-algebra $\widetilde{B}$ such that 
    \[
    \widetilde{B}\otimes_{R^{\rm{h}}_{(\varpi)}} R^{\wedge}_{(\varpi)} \simeq \widetilde{B}\otimes_{R^{\rm{h}}_{(\varpi f)}} R^{\wedge}_{(\varpi f)} \simeq B.\qedhere
    \]
\end{proof}

Now we discuss a version of the Fujiwara--Gabber comparison theorem that will be relevant for our purposes.

\begin{lemma}\label{lemma:completion-tensor-product} Let $M$ be a finite $R^{\rm{h}}_{(\varpi)}$-module, and put $M'\coloneqq M\otimes_{R^{\rm{h}}_{(\varpi)}} R^{\wedge}_{(\varpi)}$. Then there is an integer $N$ such that $M[\varpi^\infty] = M[\varpi^N]$ and $M'[\varpi^\infty] = M'[\varpi^N]$.
\end{lemma}
\begin{proof}
    First, Lemma~\ref{lemma:completion-faithfully-flat}(\ref{lemma:completion-faithfully-flat-0}) and the definition of adhesive rings imply that $M[\varpi^\infty]=M[\varpi^N]$ for some integer $N$. Then Lemma~\ref{lemma:completion-faithfully-flat}(\ref{lemma:completion-faithfully-flat-1}) and a standard flatness argument imply that $M[\varpi^m]\otimes_{R^{\rm{h}}_{(\varpi)}} R^{\wedge}_{(\varpi)} = M'[\varpi^m]$ for every integer $m\geq 0$. Therefore, $M'[\varpi^m] = M'[\varpi^N]$ for any $m\geq N$. By taking the union over all $m$, we get $M'[\varpi^\infty]=M'[\varpi^N]$.
\end{proof}

\begin{lemma}[Fujiwara--Gabber]\label{lemma:Fujiwara-Gabber} Let $B$ be a finite $R^{\rm{h}}_{(\varpi)}$-algebra, $g\colon \Spec \Big(B\otimes_{R^{\rm{h}}_{(\varpi)}} R^{\wedge}_{(\varpi)}\big[\frac{1}{\varpi}\big]\Big) \to \Spec B\big[\frac{1}{\varpi}\big]$ the natural morphism, and $\F$ a torsion \'etale sheaf on $\Spec B\big[\frac{1}{\varpi}\big]$. Then the natural morphism
\[
\rm{R}\Gamma_\et\Big(\Spec B\Big[\frac{1}{\varpi}\Big], \F\Big) \to \rm{R}\Gamma_\et\Big(\Spec \Big(B\otimes_{R^{\rm{h}}_{(\varpi)}} R^{\wedge}_{(\varpi)}\Big[\frac{1}{\varpi}\Big]\Big), g^*\F\Big)
\]
is an isomorphism.
\end{lemma}
If $R^{\rm{h}}_{(\varpi)}$ were a noetherian ring, this would almost immediately follow from \cite[Cor.\,6.6.4]{Fujiwara95}.
\begin{proof}
    First, we note that \cite[\href{https://stacks.math.columbia.edu/tag/09XK}{Tag 09XK}]{stacks-project} ensures that $B$ and $B\otimes_{R^{\rm{h}}_{(\varpi)}} R^{\wedge}_{(\varpi)}$ are $(\varpi)$-adically henselian. Furthermore, Lemma~\ref{lemma:completion-tensor-product} ensures that both $B$ and $B\otimes_{R^{\rm{h}}_{(\varpi)}} R^{\wedge}_{(\varpi)}$ have bounded $\varpi^\infty$-torsion. Therefore, the result follows from \cite[Th.\,2.3.4]{Bouthier-Cesnavicius} with $A=B$, $A'=B\otimes_{R^{\rm{h}}_{(\varpi)}} R^{\wedge}_{(\varpi)}$, $t=\varpi$, $I=B$, and $U=\Spec B\big[\frac{1}{\varpi}\big]$ (alternatively, one can use \cite[Exp.\,XX, \textsection 4.4]{deGabber} or \cite[Th.\,6.11]{BM21}).
\end{proof}

\subsection{Noether normalization and partial algebraization}

In this section, we show a generically \'etale version of the Noether normalization theorem. Then we combine it with the techniques of the previous section to get a ``partial algebraization'' result. This will be the key technical ingredient in our proof of Artin--Grothendieck vanishing for affinoid algebras.  

Throughout the section, we fix a non-archimedean field $K$ with ring of integers $\O_K$, residue field $k$, and a pseudo-uniformizer $\varpi\in \O_K$. We recall that $R$ denotes the ring $\O_K\langle X_1, \dots, X_{d-1}\rangle [X_d]$. 

\begin{thm}[Generically \'etale Noether normalization]\label{thm:generically-etale-normalization} Let $A_0$ be a flat, topologically finite type $\O_K$-algebra such that $A\coloneqq A_0[\frac{1}{\varpi}]$ is a $K$-affinoid algebra of Krull dimension $d>0$ and $\Spa(A, A^\circ)$ is geometrically reduced in the sense of \cite[\textsection 3.3]{Conrad99}. Then there is a finite, finitely presented morphism $h\colon \O_K\langle X_1,\dots, X_d\rangle \to A_0$ and a Weierstrass polynomial $f\in \O_K\langle X_1,\dots, X_{d-1}\rangle[X_d]$ such that $h$ is \'etale away from $\rm{V}(\varpi f)$, i.e., the induced map
\[
h_K\Big[\frac{1}{f}\Big] \colon K\langle X_1,\dots, X_d\rangle\Big[\frac{1}{f}\Big] \to  A\Big[\frac{1}{f}\Big] 
\]
is finite \'etale. 
\end{thm}

Before we begin the proof of Theorem~\ref{thm:generically-etale-normalization}, we want to mention some immediate corollaries of this theorem and make some remarks about the needed generality of these results.

\begin{cor}[Partial algebraization I]\label{cor:partial-algebraization-1} Keep the notation of Theorem~\ref{thm:generically-etale-normalization}. Then there is a finite, finitely presented $R^{\rm{h}}_{(\varpi)}$-algebra $B$ with an isomorphism $B\otimes_{R^{\rm{h}}_{(\varpi)}} R^{\wedge}_{(\varpi)} \simeq A_0$ of $\O_K$-algebras.
\end{cor}
\begin{proof}
    This follows directly from Theorem~\ref{thm:partial-algebraization} and Theorem~\ref{thm:generically-etale-normalization}. 
\end{proof}

\begin{cor}[Partial algebraization II]\label{cor:partial-algebraization-2} Keep the notation of Theorem~\ref{thm:generically-etale-normalization}. Then there is a finitely presented, quasi-finite morphism $R=\O_K\langle X_1, \dots, X_{d-1}\rangle[X_d] \to B$ and an isomorphism $B^{\wedge}_{(\varpi)}  \simeq A_0$ of $\O_K$-algebras.
\end{cor}
\begin{proof}
    We write $R^{\rm{h}}_{(\varpi)} = \colim_I R_i$ as a filtered colimit of \'etale $R$-algebras $R_i$ such that the natural map $R/\varpi R \to R_i/\varpi R_i$ is an isomorphism. Then Corollary~\ref{cor:partial-algebraization-1} and a standard approximation argument implies that we can find a finite, finitely presented morphism $R_i \to B_i$ such that $R^{\wedge}_{(\varpi)} \otimes_{R_i} B_i \simeq A_0$. Since $R \to R_i$ induces an isomorphism on $\varpi$-adic completions, we conclude that $A_0\simeq (R_i)^{\wedge}_{(\varpi)} \otimes_{R_i} B_i$. Finally, \cite[Th.~7.3.2 and Prop.~4.3.4]{FGK} imply that $(B_i)^{\wedge}_{(\varpi)} \simeq A_0$. Since \'etale morphisms are quasi-finite, we conclude that $B\coloneqq B_i$ does the job.
\end{proof}

\begin{rmk}\label{rmk:optimal} The assumption that $\Spa(A, A^\circ)$ is geometrically reduced cannot be dropped in Corollary~\ref{cor:partial-algebraization-2} (and, therefore, in Corollary~\ref{cor:partial-algebraization-1}) even when $\Spa(A, A^\circ)$ is of pure dimension $1$. We also note that it is not even enough to assume that $A$ is a geometrically reduced $K$-algebra in the usual algebraic sense \cite[\href{https://stacks.math.columbia.edu/tag/030S}{Tag 030S}]{stacks-project}. We refer to Proposition~\ref{prop:counterexample-1-appendex} and Proposition~\ref{prop:counterexample-3-appendex} for explicit counterexamples. 
\end{rmk}

\begin{corollary}[Partial algebraization III]\label{cor:partial-algebraization} Let $C$ be an algebraically closed non-archimedean field, and let $A$ be a reduced $C$-affinoid algebra of dimension $d>0$. Then there is a finite, finitely presented $R^{\rm{h}}_{(\varpi)}$-algebra $B$ with an isomorphism $B\otimes_{R^{\rm{h}}_{(\varpi)}} R^{\wedge}_{(\varpi)} \simeq A^\circ$ of $\O_C$-algebras.
\end{corollary}
\begin{proof}
    This follows directly from Corollary~\ref{cor:partial-algebraization-1} and \cite[Th.~1.2]{BLR4}.
\end{proof}

We note that if $\charac K=0$, then Theorem~\ref{thm:generically-etale-normalization} is an easy consequence of the usual Noether normalization \cite[Th.\,0.9.2.10]{FujKato} and Lemma~\ref{lemma:weierstrass-polynomial-exists} below. Furthermore, one can use (the proof of) \cite[Prop.\,A.2, Step 4]{ALY21} and Lemma~\ref{lemma:weierstrass-polynomial-exists} to prove Theorem~\ref{thm:generically-etale-normalization} under the additional assumptions that $K$ is algebraically closed, $A$ is a domain, and $A_0=A^\circ$. This generality suffices to conclude Corollary~\ref{cor:partial-algebraization} when $A$ is a domain. This is, in turn, the only result of this section that is used in the rest of the paper.  

However, the strategy used in \cite{ALY21} is inadequate for the purpose of proving Theorem~\ref{thm:generically-etale-normalization} in full generality. The essential difficulty comes from the fact that the ``special fiber'' $A_0\otimes_{\O_K} k$ does not need to be reduced in general. We believe Theorem~\ref{thm:generically-etale-normalization} and Corollary~\ref{cor:partial-algebraization-1} are itself of independent interest, so we decided to provide a complete proof of Theorem~\ref{thm:generically-etale-normalization} in full generality.  

The rest of this section is devoted to the proof of Theorem~\ref{thm:generically-etale-normalization}. 

\begin{lemma}\label{lemma:weierstrass-polynomial-exists} Let $d>0$ be a positive integer, and let $f\in \O_K\langle X_1,\dots,X_d\rangle$ be an element of norm $1$. Then there is a continuous $\O_K$-linear automorphism 
\[
\sigma\colon \O_K\langle X_1, \dots, X_d\rangle \to  \O_K\langle X_1, \dots, X_d\rangle,
\]
an element $u\in \O_K\langle X_1, \dots, X_d\rangle^\times$, and a Weierstrass polynomial $\omega$ such that $\sigma(f)=u\omega$.
\end{lemma}
\begin{proof}
    The proof of \cite[Lemma 2.2/7]{B} implies that we can find an $\O_K$-linear continuous automorphism $\sigma$ of $\O_K\langle X_1, \dots, X_d\rangle$ such that $\sigma(f)$ (considered as an object of $K\langle X_1, \dots, X_d\rangle$) is $X_d$-distinguished (in the sense of \cite[Def.\,2.2/6]{B}) and $|\sigma(f)|=|f|=1$. Therefore, \cite[Cor.\,2.2/9]{B} implies that there is a unit $u\in K\langle T_1,\dots,X_d\rangle^\times$ and a Weierstrass polynomial $\omega$ such that $\sigma(f)=u\omega$. In particular, $|\omega|=1$. The only thing we are left to show is that $u\in \O_K\langle X_1, \dots, X_d\rangle$ and $u^{-1} \in \O_K\langle X_1, \dots, X_d\rangle$.  

    Since the Gauss-norm is multiplicative on $K\langle X_1, \dots, X_d\rangle$ (see \cite[p.\,12]{B}), it suffices to show that $|u|=1$. But this follows from the equation $\sigma(f)=u\omega$ and the observation that $|\sigma(f)|=|f|=1$ and $|\omega|=1$.  
\end{proof}

We recall that, for every $K$-affinoid algebra $A$, there is a functorial continuous map
\[
r_A \colon \abs{\Spa(A, A^\circ)} \to \abs{\Spec A}
\]
that sends a valuation $v\in \Spa(A, A^\circ)$ to a prime ideal $\rm{supp}(v)\coloneqq v^{-1}(\{0\}) \subset A$.

\begin{lemma}\label{lemma:a-lot-of-valuations} Let $A$ be an affinoid $K$-algebra. Then the map $r_A \colon \abs{\Spa(A, A^\circ)} \to \abs{\Spec A}$ is surjective.
\end{lemma}
\begin{proof}
    For brevity, we denote by $T_d$ the Tate algebra $K\langle X_1, \dots, X_d\rangle$. We pick a prime ideal $\p\subset A$ and wish to find a valuation $v\in \Spa(A, A^\circ)$ such that $\rm{supp}\,(v)=\p$. Since the construction of $r$ is functorial in $A$, we may replace $A$ by $A/\p$ to achieve that $A$ is a domain and $\p=(0)$. 
    
    In this case, we use \cite[Prop.~3.1/3]{B} to obtain a finite monomorphism $f^\#\colon T_d \to A$. This induces morphisms $f^{\rm{alg}}\colon \Spec A \to \Spec T_d$ and $f^{\rm{ad}}\colon \Spa(A, A^\circ) \to \bf{D}^d$. One easily checks that $f^{\rm{alg}, -1}\big(\{\eta_{T_d}\}\big)=\{\eta_{A}\}$, where $\eta_{T_d}$ and $\eta_A$ are the generic points of $\Spec T_d$ and $\Spec A$ respectively. This implies that $f^{\rm{alg}}$ is dominant, so \cite[\href{https://stacks.math.columbia.edu/tag/01WM}{Tag 01WM}]{stacks-project} concludes that $f^{\rm{alg}}$ is surjective. 

    Now we claim that $f^{\rm{ad}}$ is surjective as well. We note that \cite[Lem.\,1.1.10(ii) and Lem.\,1.4.5(ii)]{Huber-etale} and \cite[Prop.\,3.6(i)]{H0} imply that it suffices to show that $A \wdh{\otimes}_{T_d} \wdh{k(x)} \neq 0$ for any rank-$1$ point $x\in \Spa(T_d, T_d^\circ)$. Therefore, \cite[Lem.\,B.3.5]{Z-quotients} ensures that it is enough to show that $A \wdh{\otimes}_{T_d} \wdh{k(x)} \simeq A \otimes_{T_d} \wdh{k(x)} \neq 0$. This follows directly from the fact that $f^{\rm{alg}}$ is surjective. 
    
    Combining the fact that $f^{\rm{alg}, -1}\big(\{\eta_{T_d}\}\big)=\{\eta_{A}\}$ with surjectivity of $f^{\rm{ad}}$ and functoriality of $r$, we reduce the question to the case $A=T_d$ and $\p=(0)$. In this case, we note that the supremum norm $v_\eta \colon T_d\to \bf{R}_{\geq 0}$ (see \cite[Prop.\,2.2/3]{B}) defines a point of $\Spa(T_d, T_d^\circ)$ with $\rm{supp}(v_\eta)=(0)$. This finishes the proof.
\end{proof}

\begin{cor}\label{cor:fpqc-covers} Let $A$ be an affinoid $K$-algebra, and let $f$ be a non-zero element of $A$. Then $\big\{\Spec A\langle \frac{\varpi^n}{f}\rangle \big\}_{n\in \N} \to \Spec A\big[\frac{1}{f}\big]$ is a jointly surjective family of flat morphisms. 
\end{cor}
\begin{proof}
    First, we note that \cite[(II.1) (iv) on page 530]{Huber-2} implies that $A \to A\langle \frac{\varpi^n}{f}\rangle$ is flat for any $n\in \N$. This formally implies that $A\big[\frac{1}{f}\big] \to A\langle \frac{\varpi^n}{f}\rangle$ is flat as well. Now, in order to see that the family $\big\{\Spec A\langle \frac{\varpi^n}{f}\rangle \big\}_{n\in \N} \to \Spec A\big[\frac{1}{f}\big]$ is jointly surjective, we invoke Lemma~\ref{lemma:a-lot-of-valuations} to reduce the question to showing that the family $\big\{\Spa(A\langle \frac{\varpi^n}{f}\rangle,A\langle \frac{\varpi^n}{f}\rangle^\circ)\big\}_{n\in \N} \to \Spa(A, A^\circ)\big(f\neq 0\big)$ is jointly surjective. This follows from the standard observation that $\Spa(A, A^\circ)\big(f\neq 0\big) = \bigcup_{n\in \N} \Spa(A\langle \frac{\varpi^n}{f}\rangle,A\langle \frac{\varpi^n}{f}\rangle^\circ) \subset \Spa(A, A^\circ)$. 
\end{proof}

We also recall the notion of completed differential forms. We fix a flat, topologically finite type $\O_K$-algebra $A_0$ and put $A=A_0\big[\frac{1}{\varpi}\big]$. We denote by $\Om_{A_0/\O_K}\coloneqq (\Omega^1_{A_0/\O_K})^{\wedge}_{(\varpi)}$ the $(\varpi)$-adic completion of the usual algebraic differentials $\Omega^1_{A_0/\O_K}$. Likewise, we denote by $\Om_{A/K}\coloneqq \Om_{A_0/\O_K}\big[\frac{1}{\varpi}\big]$ the ``generic fiber'' of $\Om_{A_0/\O_K}$; we note that \cite[Prop.~1.5]{BLR3} implies that $\Om_{A/K}$ is independent of the choice of $A_0\subset A$ and, therefore, is well-defined for any $K$-affinoid algebra $A$. We refer to \cite[\textsection I.5.1]{FujKato}, \cite[\textsection 1-3]{BLR3}, and \cite[\textsection 1.6]{Huber-etale} for an extensive discussion of these notions.

\begin{lemma}\label{lemma:basis} Let $A$ be a $K$-affinoid algebra of Krull dimension $d$. Let $\p_1, \dots, \p_r$ be a subset of minimal primes of $A$. Then, for each $j=1, \dots r$, the $k(\p_j)$-vector space $\wdh{\Omega}^1_{A/K} \otimes_A k(\p_j)$ is generated by the residue classes of the form $\ov{df}$ with $f\in \bigcap_{i\neq j} \p_i$.
\end{lemma}
\begin{proof}
    First, we note that \cite[Cor.~I.5.1.11]{FujKato} implies that the $A$-module $\wdh{\Omega}^1_{A/K}$ is generated by the elements of the form $df$ for $f\in A$. Now we note that \cite[\href{https://stacks.math.columbia.edu/tag/00DS}{Tag 00DS}]{stacks-project} ensures that, for each $j=1, \dots, r$, we can pick an element $g\in \big(\cap_{i\neq j} \p_i\big) \smallsetminus \p_j$. Then, for any $f\in A$, we have an equality $d(fg)=fdg+gdf$ in $\wdh{\Omega}^1_{A/K}$. Now since $g$ becomes invertible in $k(\p_j)$, we conclude that we have an equation $\ov{df}=\frac{\ov{d(fg)}}{\ov{g}} - \frac{\ov{f}\ov{dg}}{\ov{g}}$ in $\Om_{A/K} \otimes_A k(\p_j)$. Now the result follows since $fg$ and $g$ lie in $\cap_{i\neq j} \p_i$.
\end{proof}

\begin{lemma}\label{lemma:smooth} Let $A$ be a non-zero $K$-affinoid algebra of Krull dimension $d$ such that $\Spa(A, A^\circ)$ is geometrically reduced in the sense of \cite[\textsection 3.3]{Conrad99}, and let $\p_1, \dots, \p_r \subset A$ be the minimal prime ideals corresponding to the $d$-dimensional irreducible components of $\Spec A$. Then 
    \begin{enumerate}
        \item\label{lemma:smooth-1} the $k(\p_j)$-vector space $\Om_{A/K} \o_A k(\p_j)$ has dimension $d$;
        \item\label{lemma:smooth-2} for any finite morphism $h\colon \Spa(A, A^\circ) \to \bf{D}^d$, there is a non-zero element $f\in K\langle X_1, \dots, X_d\rangle$ such that the Zariski-open subspace $\Spa(A, A^\circ)\big(f\neq 0\big)$ is $K$-smooth of pure dimension $d$.
    \end{enumerate}
\end{lemma}
\begin{proof}
    First, we note that \cite[p.~512]{Conrad99} implies that $\Spa(A, A^\circ)$ is $K$-smooth away from a nowhere dense Zariski-closed subset $Z'\subset \Spa(A, A^\circ)$. In particular, $\dim Z'<d$. Put $Z$ to be the union of $Z'$ and all irreducible components of dimension less than $d$, and also put $U$ to be the open complement of $Z$. By construction, $\dim Z<d$ and $U$ is a smooth rigid-analytic space over $K$ of pure dimension $d$.  

    Now we note that the map $r_A\colon \abs{\Spa(A, A^\circ)} \to \abs{\Spec A}$ induces a bijection between Zariski-open subspaces (see \cite[Cor.\,B.6.8]{Z-quotients}). Thus, the fact that $\dim Z<d$ and Lemma~\ref{lemma:a-lot-of-valuations} imply that there are points $v_j\in U$ such that $r_{A}(v_j)=\p_j$. Now since $\Omega^1_{\Spa(A, A^\circ)/K}$ is a coherent sheaf associated to a finite $A$-module $\wdh{\Omega}^1_{A/K}$ and since $U$ is smooth of pure dimension $d$, we see that $\Om_{A/K}\o_A k(v_j)$ is of dimension $d$. This formally implies that $\dim \Om_{A/K} \o_A k(\p_j)=d$. 

    (\ref{lemma:smooth-2}): For this, we note that (the easier special case of) \cite[Prop.~9.6.3/3]{BGR} implies that $h(Z)\subset \bf{D}^d$ is a Zariski-closed subspace of dimension less than $d$. In particular, $h(Z)\neq \bf{D}^d$, so there is a non-zero element $f\in K\langle X_1, \dots, X_d\rangle$ such that $h(Z) \subset \rm{V}(f)$. Then $\Spa(A, A^\circ)\big(f\neq 0\big)$ is $K$-smooth of pure dimension $d$.
\end{proof}

\begin{lemma}\label{lemma:modify-morphism} Keep the notation of Lemma~\ref{lemma:smooth}, let $f_1, \dots, f_d\in A$ be a $d$-tuple of elements in $A$, and let $U\subset A$ be an open neighborhood of $0$. Then there is a $d$-tuple of elements $g_1, \dots, g_d \in U$ such that the residue classes $\ov{d(f_1 + g_1)}, \dots, \ov{d(f_d + g_d)} \in \wdh{\Omega}^1_{A/K} \otimes_A k(\p_j)$ form a basis for every $j=1, \dots, r$.
\end{lemma}
\begin{proof}
    First, we note that there is a topologically finite type $\O_C$-subalgebra $A_0\subset A$ and an integer $n\in \N$ such that $\varpi^n A_0 \subset U$. Therefore, we may and do assume that $U=\varpi^n A_0$ for $A_0$ and $n$ as above.  
    
    Now Lemma~\ref{lemma:smooth}(\ref{lemma:smooth-1}) implies that $\wdh{\Omega}^1_{A/K}\otimes_A k(\p_j)$ is a $k(\p_j)$-vector space of dimension $d$ for each $j=1, \dots, r$. Now we fix $j\in [1, \dots, r]$ and choose a subset $I_j\subset [1, \dots, d]$ such that the set $\big\{\ov{d f_i}\big\}_{i\in I_j} \subset \wdh{\Omega}^1_{A/K} \otimes_A k(\p_j)$ is a maximal linearly independent subset of $\{\ov{df_i}\}_{i= 1}^d \subset \wdh{\Omega}^1_{A/K} \otimes_A k(\p_j)$ (of size $\# I_j$). We put $g_{i, j}=0$ for each $i\in I_j$, and then Lemma~\ref{lemma:basis} implies that we can find a $\#([1, \dots, d] \smallsetminus I_j)$-tuple of elements $g_{i, j} \in \bigcap_{k\neq j} \p_k$ for $i\notin I_j$ such that the the residue classes $\ov{d (f_1 + g_{1, j})}, \dots, \ov{d (f_d + g_{d, j})}$ form a basis in $\wdh{\Omega}^1_{A/K} \otimes_A k(\p_j)$. Then we can replace each $g_{i,j}$ by $\varpi^N g_{i,j}$ for $N\gg 0$ to achieve that $g_{i, j}\in U$.  
    
    Now, for $i=1, \dots, d$, we put $g_i=\sum_{j=1}^r g_{i, j} \in U$. By construction, $g_{i, j}$ lies in $\p_{j'}$ for every $i$ and every $j\neq j'$. Therefore, $\ov{g_i}=\ov{g_{i,j}}\in k(\p_j)$ for every $i$ and $j$. Thus, we conclude that, for each $j=1, \dots, r$, the set 
    \[
    \big\{\ov{d(f_i + g_i)}\big\}_{i=1}^d = \big\{\ov{d(f_i+g_{i, j})}\big\}_{i=1}^d \subset \wdh{\Omega}^1_{A/K} \otimes_A k(\p_j)
    \]
    is a basis.
\end{proof}

\begin{lemma}\label{lemma:generically-etale} Keep the notation of Lemma~\ref{lemma:smooth}, and let $h\colon T_d=K\langle X_1, \dots, X_d\rangle \to A$ be a finite morphism of $K$-algebras such that $\big\{ \ov{d\big( h(X_i)\big)} \big\} \subset \Om_{A/K} \o_A k(\p_j)$ is a basis for each $j=1, \dots, r$. Then there is a non-zero element $f\in K\langle X_1, \dots, X_d\rangle$ such that 
\[
h\Big[\frac{1}{f}\Big] \colon K\langle X_1, \dots, X_d\rangle\Big[\frac{1}{f}\Big] \to A\Big[\frac{1}{f}\Big] 
\]
is finite \'etale.
\end{lemma}
\begin{proof}
    First, we note that \cite[\href{https://stacks.math.columbia.edu/tag/00OH}{Tag 00OH}]{stacks-project} and reducedness of $A$ imply that the natural morphism $A \otimes_{T_d} \rm{Frac}(T_d) \to \prod_{j=1}^r k(\p_j)$ is an isomorphism. Therefore, \cite[Prop.~I.5.1.10]{FujKato} and our assumption on $h$ imply that the natural morphism
    \[
    \Om_{T_d/K} \otimes_{T_d} A \otimes_{T_d} \rm{Frac}(T_d) \to \Om_{A/K} \otimes_{T_d} \rm{Frac}(T_d)
    \]
    is an isomorphism. Then a standard approximation argument implies that we can choose a  non-zero element $f\in T_d$ such that $\big(\Om_{T_d/K} \otimes_{T_d} A\big)\big[\frac{1}{f}\big] \to \Om_{A/K}\big[\frac{1}{f}\big]$ is an isomorphism. In particular, the natural morphism
    \begin{equation}\label{eqn:iso-away-from-f}
    \Om_{T_d/K} \otimes_{T_d} A\langle \frac{\varpi^n}{f}\rangle \to \Om_{A/K} \otimes_A A\langle \frac{\varpi^n}{f}\rangle
    \end{equation}
    is an isomorphism for any $n\in \N$. By Lemma~\ref{lemma:smooth}(\ref{lemma:smooth-2}) we can change $f$ to ensure that $\Spa(A, A^\circ)\big(f\neq 0\big)$ is smooth over $K$. Then \cite[Prop.~2.6]{BLR3} and (\ref{eqn:iso-away-from-f}) imply that the morphism $\Spa(A, A^\circ)\big(\frac{\varpi^n}{f}\big) \to \bf{D}^d(\frac{\varpi^n}{f})$ is finite \'etale for every $n\in \N$. In particular, $T_d\langle \frac{\varpi^n}{f} \rangle \to A\langle \frac{\varpi^n}{f} \rangle$ is finite \'etale for any $n\in \N$. Finally, we use Corollary~\ref{cor:fpqc-covers} and faithfully flat descent \cite[\href{https://stacks.math.columbia.edu/tag/02VN}{Tag 02VN}]{stacks-project} (applied to the local rings of $T_d[\frac{1}{f}]$) to conclude that $T_d[\frac{1}{f}] \to A[\frac{1}{f}]$ is finite \'etale as well. 
\end{proof}

Finally, we are ready to prove Theorem~\ref{thm:generically-etale-normalization}:

\begin{proof}[Proof of Theorem~\ref{thm:generically-etale-normalization}]
    First, we note that \cite[Th.~0.9.2.10]{FujKato} and \cite[\href{https://stacks.math.columbia.edu/tag/00OK}{Tag 00OK}]{stacks-project} imply that we can find a finite morphism 
    \[
    h\colon \O_K \langle X_1, \dots, X_d\rangle \to A_0.
    \]
    
    Now we explain how to modify $h$ to assume that $h_K\big[\frac{1}{f}\big]$ is \'etale for some non-zero element $f\in K\langle X_1, \dots, X_d\rangle$. For this, we choose $g_1, \dots, g_d \in \varpi A_0$ as in Lemma~\ref{lemma:modify-morphism} and consider the unique $\O_K$-linear continuous morphism
    \[
    h' \colon \O_K\langle X_1, \dots, X_d\rangle \to A_0
    \]
    which sends $X_i$ to $h(X_i)+ g_i$. Clearly, $h$ and $h'$ coincide modulo $\varpi$, so \cite[Prop.~I.4.2.1]{FujKato} implies that $h'$ is finite. Furhermore, Lemma~\ref{lemma:generically-etale} implies that there is a non-zero element $f\in K\langle X_1, \dots, X_d\rangle$ such that $h'_K\big[\frac{1}{f}\big]$ is finite \'etale. Therefore, we can replace $h$ by $h'$ to achieve that $h_K[\frac{1}{f}]$ is finite \'etale for some non-zero $f \in K\langle X_1, \dots, X_d\rangle$.  

    Now we explain how to modify $h$ to make $f$ to be a Weierstrass polynomial. For this, we choose an element $c\in K^\times$ such that $\abs{f}=\abs{c}$ and replace $f$ by $f/c$ to achieve that $\abs{f}=1$. Therefore, Lemma~\ref{lemma:weierstrass-polynomial-exists} guarantees that, after a continuous $\O_K$-linear automorphism of $\O_K\langle X_1, \dots, X_d\rangle$, we can assume that $f=u\omega$ for a unit $u\in \O_K\langle X_1, \dots, X_d\rangle^\times$ and a Weierstrass polynomial $\omega$. In this case, we can replace $f$ by $\omega$ to achieve that the morphism 
    \[
    h\colon \O_K\langle X_1, \dots, X_d\rangle \to A_0
    \]
    is \'etale away from $\rm{V}(\varpi f)$. Finally, we note that \cite[Th.\,7.3/4]{B} and \cite[Prop.\,1.4.7]{EGA41} imply that $h$ is finitely presented.
\end{proof}

\section{Artin--Grothendieck vanishing for affinoid algebras}

Throughout this section, we fix an algebraically closed non-archimedean field $C$ with ring of integers $\O_C$ and a pseudo-uniformizer $\varpi \in \O_C$. The main goal of this section is to prove the following theorem:

\begin{thm}[Artin--Grothendieck vanishing for affinoid algebras]\label{thm:main-theorem} Let $A$ be a $C$-affinoid algebra and let $\F$ be a torsion \'etale sheaf on $\Spec A$. Then 
\[
\rm{R}\Gamma_\et\left(\Spec A, \F\right) \in D^{\leq \dim A}(\Z).
\]
\end{thm}

For our notational convenience later on, we introduce the following notation:

\begin{notation} For a $C$-affinoid algebra $A$ and a torsion \'etale sheaf $\F$ on $\Spec A$, we say that $\AGV_{A, \F}$ holds if $\rm{R}\Gamma_\et\left(\Spec A, \F\right) \in D^{\leq \dim A}(\Z)$.
\end{notation}

\begin{rmk} Even if one is only interested in cohomological bounds on $\rm{R}\Gamma\left(\Spec A, \Z/n\Z\right)$, it is crucial for the proof to consider more general $n$-torsion sheaves $\F$. 
\end{rmk}

The last thing we discuss in this section is the easiest part of Theorem~\ref{thm:main-theorem}.

\begin{lemma}\label{lemma:main-theorem-easy} Let $\charac C=p>0$, let $A$ be a $C$-affinoid algebra, and let $\F$ be an \'etale sheaf of $\Z/p$-modules on $\Spec A$. Then 
\[
\rm{R}\Gamma_\et\left(\Spec A, \F\right) \in D^{\leq \dim A}(\Z).
\]
\end{lemma}
\begin{proof}
If $\dim A=0$, then the result is obvious. So we may assume that $\dim A\geq 1$. We put $X\coloneqq \Spec A$, then \cite[Exp.\,X, Th.\,5.1]{SGA4} says that $\rm{cd}_p(X_\et) \leq 1+ \rm{cdqc}(X)$. Since $X$ is affine, its quasi-coherent cohomological dimension is $0$. So we conclude that $\rm{cd}_p(X_\et) \leq 1$. 
\end{proof}

\subsection{Preliminary reductions}

The main goal of this section is to show that, in order to prove Theorem~\ref{thm:main-theorem}, it suffices to show that 
\[
\rm{R}\Gamma_\et(\Spec A, \Z/p\Z) \in D^{\leq \dim A}(\Z)
\]
for any normal $C$-affinoid domain $A$ and any prime number $p$.  

In this section, we will freely use that any $C$-affinoid algebra $A$ is excellent (in particular, noetherian). We refer to \cite[Th.\,3]{Kiehl-excellence} (and \cite[\textsection 1.1]{Conrad99}) for a proof of this fact. 

\begin{lemma}\label{lemma:constructible-enough} Let $A$ be an affinoid $C$-algebra. Suppose that $\AGV_{A, \F}$ holds for any prime $p$ and any constructible torsion \'etale sheaf $\F$ of $\Z/p\Z$-modules. Then $\AGV_{A, \F}$ holds for any torsion \'etale sheaf $\F$. 
\end{lemma}
\begin{proof}
    Pick  a torsion \'etale sheaf $\F$ on $\Spec A$. Then (the proof of) \cite[\href{https://stacks.math.columbia.edu/tag/0F0N}{Tag 0F0N}]{stacks-project} implies that $\F=\colim_I \F_i$ is a filtered colimit of torsion {\it constructible} sheaves $\F_i$. Therefore, \cite[\href{https://stacks.math.columbia.edu/tag/03Q5}{Tag 03Q5}]{stacks-project} implies that it suffices to prove $\AGV_{A, \F}$ for torsion constructible $\F$.  

    The natural morphism 
    \[
    \bigoplus_{p \text{ prime}} \F[p^\infty] \to \F
    \]
    is an isomorphism for any torsion sheaf $\F$. As $\Spec A$ is qcqs, it suffices to prove the result for $p^\infty$-torsion constructible sheaves $\F$ for some prime number $p$. Furthermore, \cite[\href{https://stacks.math.columbia.edu/tag/09YV}{Tag 09YV}]{stacks-project} implies that any such $\F$ is actually $p^N$-torsion for some integer $N$. Therefore, it suffices to show vanishing for each $p^i\F/p^{i+1}\F$ for $i\leq N-1$. So it suffices to verify $\AGV_{A, \F}$ for constructible sheaves of $\Z/p\Z$-modules for all primes $p$.
\end{proof}

In what follows, we are going to freely use the following fact:

\begin{lemma}\label{lemma:finite-over-affinoid}\cite[Prop.\,6.1/6]{BGR} Let $A$ be a $C$-affinoid algebra, and $B$ a finite $A$-algebra. Then $B$ is also a $C$-affinoid algebra.
\end{lemma}

\begin{lemma}\label{lemma:constant-normal-enough} Let $d$ be an integer. Suppose that $\AGV_{A, \F}$ holds for
\begin{enumerate}
    \item any $C$-affinoid algebra $A$ with $\dim A<d$ and any torsion \'etale sheaf $\F$ on $\Spec A$;
    \item any reduced $C$-affinoid algebra $A$ with $\dim A=d$, and $\F=\ud{\Z/p\Z}$ for any prime number $p$.
\end{enumerate} 
Then $\AGV_{A, \F}$ holds for any $C$-affinoid algebra $A$ of $\dim A\leq d$ and any torsion \'etale sheaf $\F$. 
\end{lemma}
\begin{proof}
Lemma~\ref{lemma:constructible-enough} implies that it suffices to prove $\AGV_{A, \F}$ a constructible \'etale sheaf $\F$ of $\Z/p\Z$-modules for any prime number $p$. So we fix a prime number $p$, a $C$-affinoid algebra $A$ of Krull dimension $d$, and a constructible sheaf of $\Z/p\Z$-modules $\F$. We wish to show that $\AGV_{A, \F}$ holds under the assumptions of this lemma.  

{\it Step~$1$. We reduce to the case of a $C$-affinoid domain $A$.} We use \cite[Prop.\,3.1/3]{B} and \cite[\href{https://stacks.math.columbia.edu/tag/00OK}{Tag 00OK}]{stacks-project} to find a finite morphism $f\colon \Spec A \to \Spec C\langle X_1, \dots, X_d\rangle$. Then \cite[\href{https://stacks.math.columbia.edu/tag/095R}{Tag 095R}]{stacks-project} ensures that $f_*\F$ is constructible and \cite[\href{https://stacks.math.columbia.edu/tag/03QP}{Tag 03QP}]{stacks-project} implies that $f_*$ is exact. Thus, $\rm{R}\Gamma_\et\left(\Spec A, \F\right) = \rm{R}\Gamma_\et\left(\Spec C\langle X_1, \dots, X_d\rangle, f_*\F\right)$, and so we may replace the pair $(A, \F)$ by $(C\langle X_1, \dots, X_d\rangle, f_*\F)$ to achieve that $A$ is a $C$-affinoid domain (see \cite[Prop.\,2.2/15]{B}).

{\it Step~$2$. We reduce to the case of a reduced $C$-affinoid $A$ and $\F=\ud{\Z/p\Z}$.} We note that constructibility of $\F$ implies that there is a qcqs dense open $j\colon V\hookrightarrow \Spec A$ such that $\F|_{V}=\cal{L}$ is a locally constant constructible sheaf of $\Z/p\Z$-modules. We denote its closed complement (with reduced scheme structure) by $i\colon Z=\Spec A'\hookrightarrow \Spec A$ and also note that $\dim A'<\dim A=d$. Therefore, the short exact sequence
\[
0 \to j_! \cal{L} \to \F\to i_*\left(\F|_{Z}\right) \to 0
\]
and the induction hypothesis imply that it suffices to show that $\AGV_{A, j_!\cal{L}}$ holds.  

Now we note that $V$ is irreducible because it is a dense open subscheme of an irreducible scheme $\Spec A$, so \cite[\href{https://stacks.math.columbia.edu/tag/0A3R}{Tag 0A3R}]{stacks-project} implies that there is a finite \'etale covering
\[
g\colon W\to V
\]
of degree prime to $p$ such that $g^*\cal{L}$ is a successive extension of the constant sheaves $\ud{\Z/p\Z}$. Using the $(g_!=\rm{R}g_*, g^*=\rm{R}g^!)$-adjunction, we get the trace map $g_*g^*\cal{L} \to \cal{L}$ such that the composition
\[
\cal{L} \to g_*g^*\cal{L} \to \cal{L}
\]
is equal to the multiplication by $\deg g$. Since $\deg g$ is coprime to $p$, we conclude that $\cal{L}$ is a direct summand of $g_*g^*\cal{L}$. Therefore, it suffices to prove that $\AGV_{A, j_!g_*g^*\cal{L}}$ holds. Since $g^*\cal{L}$ is a successive extension of constant sheaves, it actually suffices to show that $\AGV_{A, j_!g_*\ud{\Z/p\Z}}$ holds.  

Now we define $g'\colon \Spec B\to \Spec A$ to be the relative normalization of $\Spec A$ in $W$ (see \cite[\href{https://stacks.math.columbia.edu/tag/035H}{Tag 035H}]{stacks-project}). We denote the closed complement (with its reduced scheme structure) of $W$ in $\Spec B$ by $i'\colon Z'=\Spec B' \hookrightarrow \Spec B$. By construction, we have the following commutative diagram
\[
\begin{tikzcd}
W \arrow[r, hook, "j'"]\arrow{d}{g} & \Spec B \arrow{d}{g'}  & \arrow[l, swap, hook', "i'"] \Spec B' \arrow{d}\\
V \arrow[r, hook, "j"] & \Spec A & \arrow[l, swap, hook', "i"] \Spec B
\end{tikzcd}
\]
with the left square being cartesian. Now note that $g'$ is finite due to \cite[\href{https://stacks.math.columbia.edu/tag/07QV}{Tag 07QV}]{stacks-project} and \cite[\href{https://stacks.math.columbia.edu/tag/03GH}{Tag 03GH}]{stacks-project} (thus, $B$ is a $C$-affinoid algebra by Lemma~\ref{lemma:finite-over-affinoid}). Thus, $g'_!= g'_*$, and so we conclude that
\[
j_!g_*\ud{\Z/p\Z} = j_!g_!\ud{\Z/p\Z} = g'_!j'_!\ud{\Z/p\Z} = g'_* j'_!\ud{\Z/p\Z}.
\]
Therefore, $\AGV_{A, j_!g_*\ud{\Z/p\Z}}$ is tautologically equivalent to $\AGV_{A, g'_* j'_!\ud{\Z/p\Z}}$. Using that higher pushforwards along finite morphisms vanish, we also conclude that $\AGV_{A, g'_* j'_!\ud{\Z/p\Z}}$ is equivalent to $\AGV_{B, j'_!\ud{\Z/p\Z}}$. In other words, we reduced the situation to showing that $\AGV_{B, j'_!\ud{\Z/p\Z}}$ holds.  

Now we note that \cite[\href{https://stacks.math.columbia.edu/tag/035I}{Tag 035I}]{stacks-project} implies that $W$ is dense in $\Spec B$, and so $\dim B'<\dim B=\dim A=d$. Therefore, the short exact sequence
\[
0 \to j'_! \big(\ud{\Z/p\Z}\big) \to \ud{\Z/p\Z} \to i'_*\big(\ud{\Z/p\Z}\big) \to 0
\]
and the induction hypothesis imply that it suffices to show that $\AGV_{B,\ud{\Z/p\Z}}$ holds. The topological invariance of the \'etale site (see \cite[\href{https://stacks.math.columbia.edu/tag/03SI}{Tag 03SI}]{stacks-project}) implies that it suffices to show that $\AGV_{B_{\rm{red}}, \ud{\Z/p\Z}}$ holds. This finishes the proof. 
\end{proof}

\subsection{Crux of the argument}

In this section, we finally prove Theorem~\ref{thm:main-theorem}.

\begin{proof}[Proof of Theorem~\ref{thm:main-theorem}]
    We argue by induction on $d=\dim A$. The base case of $d=0$ is evident. Thus we assume that $\AGV_{A, \F}$ holds for any $C$-affinoid algebra $A$ of dimension $<d$ (and any $\F$) and prove the result for all $A$ of dimension $d$.  

    First, we note Lemma~\ref{lemma:constructible-enough} and Lemma~\ref{lemma:constant-normal-enough} imply that it suffices to show that
    \[
    \rm{R}\Gamma_\et(\Spec A, \Z/p\Z)\in D^{\leq d}(\Z)
    \]
    for any reduced $C$-affinoid algebra $A$ of dimension $d$ and any prime number $p$.  
    
    If  $p=\charac C$, then the result follows from Lemma~\ref{lemma:main-theorem-easy}. Therefore, we may and do assume that $p$ is different from the characteristic of $C$.  
    
    Now we recall that $R$ denotes the ring $\O_C\langle X_1, \dots, X_{d-1}\rangle[X_d]$, so Corollary~\ref{cor:partial-algebraization} implies that there is a finite, finitely presented $R^{\rm{h}}_{(\varpi)}$-algebra $B$ such that 
    \[
        B\otimes_{R^{\rm{h}}_{(\varpi)}} R^{\wedge}_{(\varpi)} \simeq A^\circ.
    \]
    Therefore, Lemma~\ref{lemma:Fujiwara-Gabber} implies that the natural morphism
    \[
    \rm{R}\Gamma_\et\Big(\Spec B\Big[\frac{1}{\varpi}\Big], \Z/p\Z\Big) \to \rm{R}\Gamma_\et\left(\Spec A, \Z/p\Z\right)
    \]
    is an isomorphism. Therefore, it suffices to show that $\rm{R}\Gamma_\et\left(\Spec B\big[\frac{1}{\varpi}\big], \Z/p\Z\right)\in D^{\leq d}(\Z)$.  

    We fix an integer $i>d$ and a cohomology class $c\in \rm{H}^i_\et(\Spec B\big[\frac{1}{\varpi}\big], \Z/p\Z)$. We wish to show that $c=0$. For this, we note that a standard spreading out argument (see \cite[\href{https://stacks.math.columbia.edu/tag/09YQ}{Tag 09YQ}]{stacks-project})  shows that we can find a commutative diagram
    \[
    \begin{tikzcd}
        \Spec B \arrow{d}{\psi} \arrow{r} & \Spec B' \arrow{d}{\psi'} \\
        \Spec R^{\rm{h}}_{(\varpi)} \arrow{r} \arrow{d} & \Spec R' \arrow{dl} \\
        \Spec R = \Spec \O_C\langle X_1, \dots, X_{d-1}\rangle [X_d] & 
        \end{tikzcd}
    \]
    such that the top square is cartesian, $\Spec R' \to \Spec R$ is an \'etale morphism (and an isomorphism on the mod-$\varpi$ fibers), $B'$ is a finite, finitely presented $R'$-algebra, and the class $c\in \rm{H}^i_\et(\Spec B\big[\frac{1}{\varpi}\big], \Z/p\Z)$ lies in the image of the natural morphism
    \[
    \rm{H}^i_\et\Big(\Spec B'\Big[\frac{1}{\varpi}\Big], \Z/p\Z\Big) \to \rm{H}^i_\et\Big(\Spec B\Big[\frac{1}{\varpi}\Big], \Z/p\Z\Big).
    \]
    Therefore, it suffices to show $\rm{R}\Gamma_\et\left(\Spec B'\big[\frac{1}{\varpi}\big], \Z/p\Z\right)\in D^{\leq d}(\Z)$. We consider the composite morphism
    \[
    \begin{tikzcd}
    \Spec B'[\frac{1}{\varpi}] \arrow{r} \arrow[rr, bend left, "h"] & \Spec C\langle X_1, \dots, X_{d-1}\rangle [X_d] \arrow{r} & \Spec C\langle X_1,\dots, X_{d-1}\rangle.
    \end{tikzcd}
    \]
    This is an affine relative curve, i.e., all non-empty fibers of $h$ are of dimension $\leq 1$. Therefore, constructibility of higher pushforwards along $h$ and the affine Lefschetz Theorem (see\footnote{If $\charac C=0$, one can instead use \cite[Exp.\,XIX, Th.\,5.1 and Exp.\,XIX Th.\,6.1]{SGA4} respectively.} \cite[Exp.\,XIII, Th.\,1.1.1 and Exp.\,XV, Th.\,1.1.2]{deGabber} respectively\footnote{To apply \cite[Exp.\,XV, Th.\,1.1.2]{deGabber}, we use the dimension function constructed in Corollary~\ref{cor:dimension-function-affinoid}. We also note that the dimension function $f^*\delta$ in \cite[Exp.\,XV, Th.\,1.1.2]{deGabber} is exactly the dimension function from Corollary~\ref{cor:dimension-finite-type-over-affinoid}.}) imply that $\rm{R}^\beta h_*\F$ is supported on a closed subscheme  $\Spec D_\beta \subset \Spec C\langle X_1, \dots, X_{d-1}\rangle$ of dimension $\leq d -\beta$. The induction hypothesis implies that
    \[
    \rm{R}\Gamma_\et\left(\Spec C\langle X_1, \dots, X_{d-1}\rangle, \rm{R}^{\beta}h_*\ud{\Z/p\Z}\right) =\rm{R}\Gamma_\et\left(\Spec D_\beta, \rm{R}^{\beta}h_*\ud{\Z/p\Z}\right) \in D^{\leq d - \beta}(\Z).
    \]
    Finally, an easy argument with spectral sequences ensures that
    \[
    \rm{R}\Gamma_\et\Big(\Spec B'\Big[\frac{1}{\varpi}\Big], \Z/p\Z\Big) = \rm{R}\Gamma_\et\Big(\Spec C\langle X_1, \dots, X_{d-1}\rangle, \rm{R}h_* \ud{\Z/p\Z} \Big) \in D^{\leq d}(\Z)
    \]
    finishing the proof. 
\end{proof}

\section{Rigid-analytic Artin--Grothendieck vanishing}

Throughout this section, we fix an algebraically closed non-archimedean field $C$ with valuation $\abs{.}\colon C \to \Gamma_C \cup \{0\}$, ring of integers $\O_C$, and a pseudo-uniformizer $\varpi \in \O_C$.  

The main goal of this section is to prove Conjecture~\ref{conjecture:RAAGV} in some particular situations. In particular, we show Conjecture~\ref{conjecture:RAAGV} for affinoids over a field of characterstic $0$ (in fact, a stronger version of it), and for affinoid curves over a field of arbitrary characteristic. We also provide a counterexample to some expectations from \cite{Bhatt-Hansen}. 

\subsection{Artin--Grothendieck vanishing for algebraic sheaves}

The main goal of this section is to show that Theorem~\ref{thm:main-theorem} implies Conjecture~\ref{conjecture:RAAGV} for a big class of sheaves.  

In order to explicitly specify this class of sheaves, we need some preliminary discussion. We first recall that, for every $C$-affinoid algebra $A$, \cite[Cor.\,1.7.3 and (3.2.8)]{Huber-etale} construct a morphism of topoi $c_A\colon \Spa(A, A^\circ)_\et \to \Spec A_\et$.

\begin{defn}\label{defn:algebraizable} A torsion \'etale sheaf $\F$ on $\Spa(A, A^\circ)$ is {\it algebraic} if there is a torsion \'etale sheaf $\cal{G}$ on $\Spec A$ and an isomorphism $c_A^*\cal{G}\simeq \F$.
\end{defn}

In order to get a good supply of algebraic sheaves, we recall the following definition: 

\begin{defn}[\cite{Hansen-vanishing}]\label{defn:zariski-constructible} An \'etale sheaf of $\Z/n\Z$-modules $\F$ on $\Spa(A, A^\circ)$ is {\it Zariski-constructible} if there is a locally finite stratification $\Spa(A, A^\circ)=\sqcup_{i\in I} X_i$ into Zariski locally closed subspaces\footnote{We recall that an immersion $X \xhookrightarrow{f} Y$ is called {Zariski locally closed} if $f$ can be decomposed as a composition $X \xhookrightarrow{j} Z \xhookrightarrow{i} Y$ where $j$ is a Zariski-open immersion and $i$ is a closed immersion. We refer to \cite[Appendix B.6]{Z-quotients} for the basics on closed immersions in the rigid-analytic context.} of $\Spa(A, A^\circ)$ such that $\F|_{X_i}$ is lisse for every $i\in I$.
\end{defn}

\begin{lemma}\label{lemma:examples-of-algebraizable-sheaves} Let $A$ be a $C$-affinoid algebra, and $n$ an integer. Then
\begin{enumerate}
    \item\label{lemma:examples-of-algebraizable-sheaves-1} any lisse \'etale sheaf of $\Z/n\Z$-modules on $\Spa(A, A^\circ)$ is algebraic;
    \item\label{lemma:examples-of-algebraizable-sheaves-2} if $\charac C=0$, then any Zariski-constructible \'etale sheaf of $\Z/n\Z$-modules on $\Spa(A, A^\circ)$ is algebraic.
\end{enumerate}
\end{lemma}
\begin{proof}
    In order to show (\ref{lemma:examples-of-algebraizable-sheaves-1}), it suffices to show that the analytification functor induces an equivalence $(\Spec A)_{\rm{f\et}} \xrightarrow{\sim} \Spa(A, A^\circ)_{\rm{f\et}}$. This follows directly from the observation that both categories are canonically equivalent to the category $A_{\rm{f\et}}$ of finite \'etale $A$-algebras. (\ref{lemma:examples-of-algebraizable-sheaves-2}) follows directly from \cite[Th.\,1.7]{Hansen-vanishing}. 
\end{proof}

\begin{thm}[Rigid-analytic Artin--Grothendieck vanishing]\label{thm:main-theorem-char-0} Let $A$ be an affinoid $C$-algebra, and let $\F$ be an algebraic torsion \'etale sheaf on $\Spa(A, A^\circ)$ (in the sense of Definition~\ref{defn:algebraizable}). Then 
\begin{equation}\label{eqn:rigid-AGV}
    \rm{R}\Gamma_{\operatorname{\text{\'e}t}}\big(\Spa(A, A^\circ), \F\big) \in D^{\leq \dim A}(\Z).
\end{equation}
In particular, (\ref{eqn:rigid-AGV}) holds in either of the following situations:
\begin{enumerate}
    \item $\F$ is a lisse sheaf of $\Z/n\Z$-modules for some integer $n$;
    \item $\charac C =0$ and $\F$ is a Zariski-constructible sheaf of $\Z/n\Z$-modules for some integer $n$.
\end{enumerate}
\end{thm}
\begin{proof}
    This follows directly from Theorem~\ref{thm:main-theorem}, Lemma~\ref{lemma:examples-of-algebraizable-sheaves}, and \cite[Th.\,1.9]{Hansen-vanishing}.
\end{proof}

We note that Theorem~\ref{thm:main-theorem-char-0} proves a {\it stronger} version of Conjecture~\ref{conjecture:RAAGV} if $\charac C=0$. But, if $\charac C = p>0$, then Theorem~\ref{thm:main-theorem-char-0} does not solve Conjecture~\ref{conjecture:RAAGV} in full generality due to the existence of Zariski-constructible sheaves that are not algebraic (see \cite[p.\,302]{Hansen-vanishing}). 

\subsection{Artin--Grothendieck vanishing in top degree}

In this section, we show vanishing of the top cohomology group of any Zariski-constructible sheaf on an affinoid space (without any assumptions on the characteristic of the ground field). In particular, this will be enough to conclude Conjecture~\ref{conjecture:RAAGV} for affinoid curves over an algebraically closed field of {\it arbitrary} characteristic. We recall that the field $C$ is assumed to be algebraically closed.


\begin{defn} The {\it Gauss point} $\eta\in \bf{D}^d=\Spa\left(C\langle X_1, \dots, X_d\rangle, \O_C\langle X_1, \dots, X_d\rangle\right)$ is the point corresponding to the valuation
\[
v_\eta \colon C\langle X_1, \dots, X_d\rangle \to \Gamma_C \cup \{0\}
\]
\[
v_\eta\left(\sum_{i_1, \dots i_d} a_{i_1, \dots, i_d} X_1^{i_1}\cdot \dots \cdot X_d^{i_d}\right) = \sup_{i_1, \dots, i_d} \Bigl( |a_{i_1, \dots, i_d}|\Bigr).
\]
\end{defn}

\begin{lemma}\label{lemma:specialization-Gauss-point} Let $\bf{D}^d$ be the $d$-dimensional closed unit disc, and let $\bf{D}^{d, c}$ be its universal compactification (in the sense of \cite[Def.\,5.1.1]{Huber-etale}). If $d\geq 1$, then there is a point $x\in \bf{D}^{d, c} \smallsetminus \bf{D}^d$ that generizes to the Gauss point $\eta$. 
\end{lemma}
\begin{proof}
    For brevity, we denote by $A$ the ring $C\langle X_1, \dots, X_d\rangle$ and by $A^{\rm{min}}$ the integral closure of $\O_C[A^{\circ\circ}]$ in $A$. Then \cite[Ex.\,5.10(i)]{Swan} ensures that $\bf{D}^{d, c} = \Spa(A, A^{\rm{min}})$. Now we consider the map
    \[
    v_x \colon C\langle X_1, \dots, X_d \rangle \to \Gamma_C \times \Z \cup \{0\}
    \]
    \[
    v_x\left(\sum_{i_1, \dots i_d} a_{i_1, \dots, i_d} X_1^{i_1}\cdot \dots \cdot X_d^{i_d}\right) = \sup_{i_1, \dots, i_d\ | \ a_{i_1, \dots, i_d}\neq 0} \Bigl( |a_{i_1, \dots, i_d}|, i_1\Bigr), 
    \]
    where supremum over the empty is set to be $0$. One easily checks that it is a valuation if $\Gamma_C \times \Z$ is endowed with the lexicographical order. We also see that $v_x(A^{\rm{min}})\leq (1,0)$ and $v_x$ is continuous due to \cite[L.\,9, Cor.\,9.3.3]{Seminar}, so $v_x$ defines a point $x\in \bf{D}^{d, c}$. Furthermore, we note that $x\in \bf{D}^{d, c} \smallsetminus \bf{D}^d$ because $v_x(X_1) > (1, 0)$. Finally, we note that $\eta$ is a generization of $x$ since, for every $f, g\in A$, $v_x(f)\leq v_x(g)$ implies $v_\eta(f)\leq v_\eta(g)$.
\end{proof}

\begin{lemma}\label{lemma:no-Homs} Let $j\colon \bf{D}^d \hookrightarrow \bf{P}^{d, \an}$ be the standard open immersion for some $d\geq 1$, let $j'\colon U \hookrightarrow \bf{D}^d$ be a Zariski-open immersion, let $n$ be an integer invertible in $\O_C$, and let $\cal{L}$ be a lisse sheaf of $\Z/n\Z$-modules on $U_\et$ and $\F = j'_!\,\cal{L}$. Then any $\Z/n\Z$-linear homomorphism $\varphi\colon j_*\F \to \ud{\Z/n\Z}_{\bf{P}^{d, \an}}$ is the zero morphism.
\end{lemma}
\begin{proof}
    If $U$ is empty, the claim is trivial. Therefore, we may and do assume that $U$ is non-empty throughout the proof. We first note that a map of sheaves is uniquely determined by the induced morphisms on stalks. By \cite[Prop.\,2.6.4]{Huber-etale} and its proof, every stalk of $j_*\cal{F}$ is either $0$ or isomorphic to a stalk of $\cal{F}$ by a specialization map, and thus the restriction morphism
    \begin{align*}
    \rm{Hom}_{\bf{P}^{d, \an}}\left(j_*\F, \ud{\Z/n\Z}_{\bf{P}^{d, \an}}\right) & \to \rm{Hom}_{\bf{D}^d}\left(\F, \ud{\Z/n\Z}_{\bf{D}^d}\right) \\
    & \simeq \rm{Hom}_{\bf{D}^d}\left(j'_!\,\cal{L}, \ud{\Z/n\Z}_{\bf{D}^d}\right) \\
    & \simeq \rm{Hom}_U\left(\cal{L}, \ud{\Z/n\Z}_{U}\right)
    \end{align*}
    is injective. Since $U$ is connected (see \cite[Cor.\,2.7]{Hansen-vanishing}) and $\cal{L}$ and $\ud{\Z/n\Z}$ are {\it lisse} sheaves, it suffices to show that the stalk morphism $\varphi_{\ov{y}}\colon \left(j_*\F\right)_{\ov{y}} \to \left(\ud{\Z/n\Z}_{\bf{P}^{d, \an}}\right)_{\ov{y}}$ vanishes at {\it one} single geometric point $\ov{y} \to U$.  

    Now we note that the Gauss point $\eta\in \bf{D}^d$ is a Shilov point in the sense of \cite[Def.\,2.4]{Bhatt-Hansen}, so \cite[Cor.\,2.10]{Bhatt-Hansen} implies that $\eta\in U$. Therefore, it suffices to show that $\varphi_{\ov{\eta}}=0$ for a geometric point $\ov{\eta}$ above $\eta$.  

    Fot this, we consider the closure $\ov{\bf{D}}^d\subset \bf{P}^{d, \an}$ of $\bf{D}^d$ inside $\bf{P}^{d, \an}$. Then \cite[Lemma 4.2.5 and Prop.\,4.2.11]{rigid-motives} ensure\footnote{Strictly speaking, \cite[Prop.\,4.2.11]{rigid-motives} assumes that $S$ is universally uniform in the sense of \cite[Def.\,4.2.7]{rigid-motives}. This hypothesis is essentially never satisfied due to the observation that $A\langle X\rangle/(X^2)$ is not uniform for a non-zero complete Tate ring $A$. However, the proof \cite[Prop.\,4.2.11]{rigid-motives} does work for any locally strongly noetherian analytic space $S$.} that $\ov{\bf{D}}^d$ is homeomorphic to the universal compactification $\bf{D}^{d, c}$. Therefore, Lemma~\ref{lemma:specialization-Gauss-point} implies that there is a point $x\in \ov{\bf{D}}^{d} \smallsetminus \bf{D}^d$ that generizes to $\eta$. We choose a geometric point $\ov{x}$ above $x$ with a specialization morphism $\ov{\eta} \to \ov{x}$ (see \cite[Lemma 2.5.14 and (2.5.16)]{Huber-etale}).  

    Now we recall that $U\subset \bf{D}^d$ is Zariski-open, so it is closed under all specializations and generizations. Therefore, $\F=j'_!\, \cal{L}$ is overconvergent. Thus \cite[Prop.\,8.2.3(ii)]{Huber-etale} ensures that $j_*\F$ is overconvergent as well. Since $\ud{\Z/n\Z}_{\bf{P}^{d, \an}}$ is clearly overconvergent, we conclude that stalk-morphisms $\varphi_{\ov{\eta}}$ and $\varphi_{\ov{x}}$ are canonically identified. So we reduce the question to showing that $\varphi_{\ov{x}}=0$ for any $\Z/n\Z$-linear morphism $\varphi \colon j_*\F \to \ud{\Z/n\Z}_{\bf{P}^{d, \an}}$.
    
    For this, it suffices to show that, for any \'etale morphism $f_V\colon V \to \bf{P}^{d, \an}$ with a connected affinoid $V$ and $x\in f_V(V)$, the morphism 
    \[
    \varphi(V) \colon \left(j_*\F\right)(V) \to \ud{\Z/n\Z}_{\bf{P}^{d, \an}}(V)=\Z/n\Z
    \]
    is the zero morphism. Since $x\in f_V(V)$, we conclude that $f_V^{-1}\left(\bf{D}^d\right)\neq V$, so there is a classical point $z\in V \smallsetminus f_V^{-1}\left(\bf{D}^d\right)$. We choose a geometric point $\ov{z} \to V$ above $z$. Then we have the following commutative diagram
    \begin{equation}\label{eqn:zero-on-stalks}
    \begin{tikzcd}
        \left(j_*\F\right)(V)\arrow{d} \arrow{r}{\varphi(V)} & \left(\ud{\Z/n\Z}_{\bf{P}^{d, \an}}\right)(V)\simeq \Z/n\Z \arrow{d}{\wr} \\
        \left(j_*\F\right)_{\ov{z}} \arrow{r}{\varphi_{\ov{z}}} & \left(\ud{\Z/n\Z}_{\bf{P}^{d, \an}}\right)_{\ov{z}} \simeq \Z/n\Z.
    \end{tikzcd}
    \end{equation}
    Since $z\in V$ is a classical point, it has no proper generizations. Therefore, \cite[\href{https://stacks.math.columbia.edu/tag/0904}{Tag 0904}]{stacks-project} ensures that there is an open neighborhood $z\in W\subset V$ disjoint from $f_V^{-1}(\bf{D}^d)$. So we conclude that $(j_*\F)_{\ov{z}}=0$. Thus, Diagram~(\ref{eqn:zero-on-stalks}) implies that $\varphi(V)=0$ finishing the proof.  
\end{proof}

\begin{thm}[Rigid-analytic Artin--Grothendieck vanishing in top degree]\label{thm:main-theorem-top-degree} Let $X=\Spa(A, A^\circ)$ be an affinoid rigid-analytic space over $C$ with $\dim A=d\geq 1$, let $n$ be an integer invertible in $\O_C$, and let $\F$ be a Zariski-constructible sheaf of $\Z/n\Z$-modules on $X_\et$. Then 
\[
\rm{H}^{2d}_\et\big(X, \F\big) = 0.
\]
\end{thm}
\begin{proof}
    First, \cite[Prop.\,3.3/1]{B} and \cite[\href{https://stacks.math.columbia.edu/tag/00OK}{Tag 00OK}]{stacks-project} imply that there is a finite morphism $f\colon X \to \bf{D}^d$. Then \cite[Prop.\,2.3]{Hansen-vanishing} ensures that $f_*\F$ is Zariski-constructible, while \cite[Prop.\,2.6.3]{Huber-etale} ensures that $f_*$ is exact. In particular, this implies that $\rm{H}^{2d}_\et\left(X, \F\right) = \rm{H}^{2d}_\et\left(\bf{D}^d, f_*\F\right)$, so we can replace the pair $(X, \F)$ by $(\bf{D}^d, f_*\F)$ to achieve that $X=\bf{D}^d$ is the closed unit disc.  

    Now the definition of Zariski-constructible sheaves implies that there is a non-empty Zariski-open subspace $j'\colon U\hookrightarrow \bf{D}^d$ such that $\cal{L}\coloneqq \F|_U$ is lisse. We denote the closed complement (with the reduced adic space structure) by $i\colon Z \hookrightarrow \bf{D}^d$. Then we have a short exact sequence
    \[
    0 \to j'_!\, \cal{L} \to \F \to i_* \left(\F|_Z \right) \to 0.
    \]
    Now \cite[Lemma 1.8.6(ii)]{Huber-etale} guarantees that $\dim Z=\rm{dim.tr}\,Z \leq d-1$, so \cite[Cor.\,2.8.3]{Huber-etale} implies that $\rm{R}\Gamma_\et\left(X, i_* \left(\F|_Z \right)\right)=\rm{R}\Gamma_\et\left(Z, \F|_Z \right) \in D^{\leq 2d-2}(\Z)$. Therefore, it suffices to prove the claim for $\F = j'_! \cal{L}$.  
    
    In this situation, we consider the compactification $j\colon \bf{D}^d \hookrightarrow  \bf{P}^{d, \rm{an}}$. Then \cite[Prop.\,2.6.4]{Huber-etale} implies that $j_*$ is exact, and so $\rm{H}^{2d}_\et\left(\bf{D}^d, \F\right) = \rm{H}^{2d}_\et\left(\bf{P}^{d, \an}, j_*\F\right)$. In particular, it suffices to show that {\it the dual} $\Z/n\Z$-module $\rm{H}^{2d}_\et\left(\bf{P}^{d, \an}, j_*\F\right)^{\vee}$ vanishes. Now, after choosing a trivialization of the Tate twist $\ud{\Z/n\Z}(d) \cong \ud{\Z/n\Z}$, our assumption that $n\in \O_C^\times$ and Poincar\'e duality (see \cite[Cor.\,7.5.6]{Huber-etale}, \cite[Th.\,7.3.4]{Berkovich}, or \cite[Th.\,1.3.2]{Z-revised}) imply that
    \[
    \rm{H}^{2d}_\et\left(\bf{P}^{d, \an}, j_*\F\right)^{\vee} = \rm{Hom}_{\bf{P}^{d, \an}}\left(j_*\F, \ud{\Z/n\Z}_{\bf{P}^{d, \an}}\right).
    \]
    Thus, the desired vanishing follows directly from Lemma~\ref{lemma:no-Homs}. 
\end{proof}

\begin{rmk} We want to mention that it is possible to prove Theorem~\ref{thm:main-theorem-top-degree} in greater generality. Namely, in the formulation of Theorem~\ref{thm:main-theorem-top-degree}, it suffices to assume that 
\begin{enumerate}
    \item the space $X$ is a quasi-compact separated rigid-analytic $C$-space of dimension $d$ such that none of its irreducible components is a proper rigid-analytic $C$-space of dimension $d$;
    \item $\F$ is a Zariski-constructibe sheaf of $\Lambda$-modules for a $\Z/n\Z$-algebra $\Lambda$ (and an integer $n$ invertible in $\O_C$).
\end{enumerate}
    Since we do not know any interesting application of this extra generality and the argument becomes significantly longer, we prefer not to spell it out in this paper. 
\end{rmk}

\begin{cor}\label{cor:curves} Let $A$ be an $C$-affinoid algebra of dimension $1$, let $n$ be an integer invertible in $C$, and let $\F$ be a Zariski-constructible sheaf of $\Z/n\Z$-modules on $\Spa(A, A^\circ)_\et$. Then 
\[
\rm{R}\Gamma_\et\big(\Spa(A, A^\circ), \F\big) \in D^{\leq 1}(\Z).
\]
\end{cor}
\begin{proof}
    If $\charac C =0$, the result follows from Theorem~\ref{thm:main-theorem-char-0}. If $\charac C=p>0$, then $n$ is invertible in $\O_C$. Therefore, the result follows from Theorem~\ref{thm:main-theorem-top-degree} and \cite[Cor.\,2.8.3]{Huber-etale}.
\end{proof}

\subsection{Counterexample to perverse exactness of nearby cycles}

Throughout this subsection, we assume that $C$ is an algebraically closed non-archimedean field of mixed characteristic $(0, p)$.  

In \cite[Th.\,4.2]{Bhatt-Hansen}, it is claimed that Theorem~\ref{thm:main-theorem-char-0} should imply that, for any admissible formal $\O_C$-scheme $\X$, the nearby cycles functor $\rm{R}\lambda_* \colon D^b_{zc}(\X_\eta; \bf{F}_p) \to D^+(\X_s; \bf{F}_p)$ is exact with respect to the perverse $t$-structure on the source (see \cite[Def.\,4.1]{Bhatt-Hansen}) and the perverse $t$-structure on the target (see \cite{Gabber-perverse} or \cite[\textsection 2, p.7]{Cass}). However, this claim is false: the next lemma shows that this functor is not perverse right $t$-exact even for the (formal) affine line.

\begin{lemma} Let $\X=\wdh{\bf{A}}^1_{\O_C} = \Spf \O_C\langle T\rangle$ be the formal affine line over $\O_C$, let $\zeta\in \X_s=\bf{A}^1_{s}$ be the generic point in the special fiber of $\X$, and let $\rm{R}\lambda_*\colon D^b_{zc}(\X_\eta; \bf{F}_p) \to D^+(\X_s; \bf{F}_p)$ be the nearby cycles functor. Then 
\[
\cal{H}^0\left(\rm{R}\lambda_* \mu_p[1]\right)_{\ov{\zeta}}=\left(\rm{R}^1\lambda_* \mu_p\right)_{\ov{\zeta}} \neq 0.
\]
In particular, $\rm{R}\lambda_*$ is not perverse right $t$-exact. 
\end{lemma}
\begin{proof}
    Since $\X_\eta$ is smooth, we conclude that $\mu_p[1]$ is a perverse sheaf on $\X_\eta$. So, in order to show that $\rR\lambda_*$ is not perverse right $t$-exact, it suffices to show that $\rm{R}\lambda_* \mu_p[1] \notin {}^pD^{\leq 0}(\X_s; \bf{F}_p)$. Using the definition of ${}^pD^{\leq 0}(\X_s; \bf{F}_p)$ (see \cite[p.7]{Cass}), we see that it is enough to show that $\cal{H}^0\left(\rm{R}\lambda_* \mu_p[1]\right)_{\ov{\zeta}}\neq 0$, where $\zeta \in \X_s$ is the generic point and $\ov{\zeta}$ is a geometric point of $\X_s$ above $\zeta$.  

    Now we put $X\coloneqq \bf{A}^1_{\O_C}=\Spec \O_C[T]$ to be the schematic affine line over $\O_C$, $j\colon X_\eta \hookrightarrow  X$ to be the open immersion of the generic fiber $X_\eta$ into $X$, and $i\colon X_s \hookrightarrow  X$ to be the closed immersion of the special fiber $X_s$ into $X$. Since $\abs{\X_s} = \abs{X_s}$, we can consider $\zeta$ as a point of $X$ corresponding to the generic point of the special fiber $X_s$. We denote by $\O_{X, \zeta}^{\rm{sh}}$ the strict henselization of the local ring $\O_{X, \zeta}$ (corresponding to $\ov{\zeta}$). Then combining \cite[Th.\,3.5.13]{Huber-etale} with \cite[\href{https://stacks.math.columbia.edu/tag/03Q9}{Tag 03Q9}]{stacks-project}, we conclude that  
    \[
    \left(\rm{R}^1\lambda_* \mu_p\right)_{\ov{\zeta}} \simeq \left(i^*\rm{R}^1j_* \mu_p\right)_{\ov{\zeta}} \simeq \left(\rm{R}^1j_*\mu_p\right)_{\ov{\zeta}} \simeq \rm{H}^1_\et\Big(\Spec \O_{X, \zeta}^{\rm{sh}}\Big[\frac{1}{p}\Big], \mu_p\Big).
    \]
    For brevity, we denote $\O_{X, \zeta}^{\rm{sh}}[\frac{1}{p}]$ by $K_\zeta$. Then \cite[Prop.\,9.1.32]{GRfoundations} implies that $\O^{\rm{sh}}_{X, \zeta}$ is a rank-$1$ valuation ring with the maximal ideal $\m_C \O^{\rm{sh}}_{X, \zeta}$. Thus $p\in \O^{\rm{sh}}_{X, \zeta}$ is a pseudo-uniformizer, so $K_\zeta$ is the fraction field of $\O^{\rm{sh}}_{X, \zeta}$. Hence, the Kummer exact sequence and Hilbert's Theorem $90$ imply that $K_\zeta^\times/K_\zeta^{\times, p} = \rm{H}^1_\et(\Spec K_\zeta, \mu_p)$. Therefore, we reduce the question to showing that $K_\zeta^\times/K_\zeta^{\times, p} \neq 0$.
    
    Since $K_\zeta$ is a field, we see that $T\in K_\zeta^\times$. Therefore, in order to show that $K_\zeta^\times/K_\zeta^{\times, p} \neq 0$, it suffices to show that there is no element $S \in K_\zeta$ such that $S^p=T$. Suppose that there was such an element $S$, then $S\in \O^{\rm{sh}}_{X, \zeta}$ since $\O^{\rm{sh}}_{X, \zeta}$ is a valuation ring (so it is integrally closed in $K_\zeta$). On the other hand, since $\m_C \O^{\rm{sh}}_{X, \zeta}$ is the maximal ideal of $\O^{\rm{sh}}_{X, \zeta}$, we conclude that
    \[
    \O^{\rm{sh}}_{X, \zeta}/\m_C \O^{\rm{sh}}_{X, \zeta} \simeq k(T)^{\rm{sep}},
    \]
    where $k=\O_C/\m_C$ is the residue field of $\O_C$ and $k(T)^{\rm{sep}}$ is a separable closure of $k(T)$. Therefore, we conclude that the element $T\in k(T)^{\rm{sep}}$ must also admit a $p$-th root, which is clearly false by separability of the extension $k(T)\subset k(T)^{\rm{sep}}$. This implies that no such $S\in K_\zeta$ exists, and thus finishes the proof that $K_\zeta^\times/K_\zeta^{\times, p} \neq 0$.
\end{proof}

\begin{rmk} The nearby cycles functor is expected to be left $t$-exact with respect to the perverse $t$-structures. In the case of an ``algebraizable'' admissible formal $\O_C$-scheme $\X$, this is going to be proven in \cite{BL-RH}.
\end{rmk}

\section{Noetherian rig-smooth algebraization}

In this section, we give a proof of Theorem~\ref{thm:intro-artin-algebraization} for any noetherian ring $A$ and any ideal $I\subset A$. The strategy of the argument is to first establish some ``weak'' uniqueness of an algebraization; this occupies the first three lemmas of this section. Then we use this weak uniqueness and some deformation-theoretic arguments to run induction on the number of generators of the ideal $I$. This occupies the rest of the section. 

\begin{lemma}\label{lemma:isoms-on-completions} Let $(A, I)$ be a noetherian henselian pair, let $B$ be a finite type $A$-algebra, and let $\varphi \colon A \to B^\h_I$ be the structure morphism. If the induced morphism $\varphi^{\wedge}_I \colon A^{\wedge}_I \to B^{\wedge}_I$ is an isomorphism, then so is $\varphi$.
\end{lemma}
\begin{proof}
    Our assumption on $\varphi$ implies that the natural morphism $A/I \to B/IB$ is an isomorphism. Then \cite[Prop.~2.3.2]{Moret-Bailly} ensures that there is a clopen subscheme $\Spec B' \subset \Spec B$ such that $A \to B'$ is a finite morphism and $B/IB \to B'/IB'$ is an isomorphism. Therefore, we conclude that $B^\h_I \xr{\sim} (B')^\h_I$. Thus we can replace $B$ by $B'$ to achieve that $A \to B$ is a finite morphism. In this case, \cite[\href{https://stacks.math.columbia.edu/tag/0DYE}{Tag 0DYE}]{stacks-project} implies that $B=B^\h$, so $A \to B^\h_I$ is a finite morphism. Then we conclude that $B^{\wedge}_I \simeq B^\h_I \otimes_{A} A^{\wedge}_I$, so $\varphi^{\wedge}_I = \varphi\otimes_A A^{\wedge}_I$. We conclude that $\varphi$ is an isomorphism since $\varphi\otimes_A A^{\wedge}_I$ is an isomorphism and $A \to A^{\wedge}_I$ is faithfully flat (see \cite[\href{https://stacks.math.columbia.edu/tag/0AGV}{Tag 0AGV}]{stacks-project}).
\end{proof}

\begin{lemma}\label{lemma:isoms-open} Let $A$ be a noetherian ring, let $I\subset A$ be an ideal, let $B$ and $C$ be finite type $A$-algebras, let $\psi\colon B^\h_I \to C^\h_I$ and $\varphi\colon B^{\wedge}_I \to C^{\wedge}_I$ be  homomorphisms of $A$-algebras. If $\varphi$ is an isomorphism and $\psi \,\rm{ mod }\, I = \varphi \,\rm{ mod }\,I$, then $\psi$ is an isomorphism.
\end{lemma}
\begin{proof}
    We denote by $\psi'\colon B\to C^\h_I$ the composition of $\psi$ with the the morphism $B\to B^\h_I$. Then a standard approximation argument ensures that there is an \'etale morphism $C \to C'$ such that $C/IC \to C'/IC'$ is an isomorphism and $\psi'$ factors as $B \xr{\psi''} C' \to (C')^\h_I \simeq C^\h_I$. Since $C^\h_I \simeq (B^\h_I \otimes_B C')^\h_I$, we conclude that $C^\h_I$ (with the $B^\h_I$-algebra structure coming from $\psi$) is obtained as the $I$-adic henselization of a finite type $B^\h_I$-algebra. Thus Lemma~\ref{lemma:isoms-on-completions} implies that it suffices to show that $\psi^{\wedge}_I \colon B^{\wedge}_I \to C^{\wedge}_I$ is an isomorphism. This follows from \cite[Ch.~III, \textsection 2, n.~8, Cor.~3]{Bourbaki} since $\rm{gr}_I(\psi^{\wedge}_I) = \rm{gr}_I(\varphi)$ by our assumption. 
\end{proof}

For the next definition, we fix a ring $A$ with an ideal $I\subset A$.

\begin{defn} A finitely presented morphism $A \to B$ is {\it smooth outside $\rm{V}(I)$} (or {\it smooth outside $I$}) if $\Spec B \smallsetminus \rm{V}(IB) \to \Spec A \smallsetminus \rm{V}(I)$ is a smooth morphism.
\end{defn}

Now we are ready to show the ``weak'' uniqueness result promised at the beginning of the section. 

\begin{lemma}\label{lemma:approximating-maps} Let $A$ be a noetherian ring, let $I\subset A$ be an ideal, let $B$ be a finite type $A$-algebra that is smooth outside $\rm{V}(I)$, let $C$ be a finite type $A$-algebra, and let $\varphi\colon B^{\wedge}_I \to C^{\wedge}_I$ be a morphism of $A$-algebras. Then, for every integer $n>0$, there is a morphism of $A$-algebras $\psi\colon B^\h_I \to C^\h_I$ such that $\psi \,\rm{ mod }\, I^n = \varphi \,\rm{ mod }\,I^n$. Furthermore, if $\varphi$ is an isomorphism, then any such $\psi$ is an isomorphism as well. 
\end{lemma}
\begin{proof}
    We fix a presentation $B=A[X_1, \dots, X_m]/(f_1, \dots, f_s)$. Then the morphism $\varphi$ is uniquely defined by the set of elements $\wdh{c}_1 \coloneqq \varphi(X_1), \dots, \wdh{c}_m \coloneqq \varphi(X_m)\in C^{\wedge}_I$ such that $f_i(\,\ud{\wdh{c}}\,)=0$ for $i=1, \dots, s$. Now \cite[Th.~2bis on p.560]{Elkik} implies that there are elements $c_1, \dots, c_m\in C^\h_I$ such that $f_i(\,\ud{c}\,)=0$ for $i=1, \dots, s$ and $c_j-\wdh{c}_j\in I^nC^\h_I$ for $j=1, \dots, m$. Then we define $\psi'\colon B\to C^\h_I$ to be the unique $A$-linear morphism that sends $X_j$ to $c_j$ for $j=1, \dots, m$. By construction, we have $\psi' \,\rm{ mod }\, I^n = \varphi \,\rm{ mod }\,I^n$. Now the universal property of henselizations (see \cite[\href{https://stacks.math.columbia.edu/tag/0A02}{Tag 0A02}]{stacks-project}) implies that $\psi'$ uniquely extends to a morphism $\psi\colon B^\h_I \to C^\h_I$. We conclude that $\psi \,\rm{ mod }\, I^n = \varphi \,\rm{ mod }\,I^n$ since the same was true for $\psi'$. Finally, if $\varphi$ is an isomorphism, then Lemma~\ref{lemma:isoms-open} guarantees that $\psi$ must be an isomorphism as well.
\end{proof}

Before we prove the main algebraization result, we need to establish a certain result in deformation theory. It is proven in the two lemmas below. For a ring $A$ and an $A$-module $M$, we denote by $\widetilde{M}$ the quasi-coherent sheaf on $\Spec A$ associated to $M$.

\begin{lemma}\label{lemma:miracle-equiality} Let $A$ be a ring, let $I\subset A$ be a finitely generated ideal, let $U \coloneqq \Spec A \smallsetminus \rm{V}(I)$ be the open complement of $\rm{V}(I)$, let $M$ be an $A$-module, and let $A \to B$ be a flat morphism such that $A/I \to B/IB$ is an isomorphism. We put $M_B \coloneqq M\otimes_A B$ and $U_B \coloneqq U\times_{\Spec A} \Spec B$. Then the natural morphism
\[
\rm{H}^i(U, \widetilde{M}|_U) \to \rm{H}^i(U_B, {\widetilde{M}_B}|_{U_B})
\]
is an isomorphism for any $i\geq 1$.
\end{lemma}
\begin{proof}
    We put $Z\coloneqq \rm{V}(I) \subset \Spec A$ and $Z_B \coloneqq \rm{V}(IB) \subset \Spec B$. Then \cite[\href{https://stacks.math.columbia.edu/tag/0DWR}{Tag 0DWR}]{stacks-project} implies that it suffices to show that the natural morphism $\rm{H}^i_Z(M) \to \rm{H}^i_{Z_B}(M_B)$ is an isomorphism for any $i\geq 2$. We show that this actually holds for any $i\geq 0$. For this, we note that \cite[\href{https://stacks.math.columbia.edu/tag/0ALZ}{Tag 0ALZ}]{stacks-project} implies that it suffices to show that the natural morphism
    \[
    \rm{H}^i_Z(M) \to \rm{H}^i_{Z}(M) \otimes_A B
    \]
    is an isomorphism for any $i\geq 0$. This follows directly from \cite[\href{https://stacks.math.columbia.edu/tag/05E9}{Tag 05E9}]{stacks-project}.
\end{proof}

\begin{lemma}\label{lemma:lift-rig-smooth} Let $A$ be a ring, let $I\subset A$ be a finitely generated ideal, let $J\subset A$ be a square-zero ideal, let $C$ be a finitely presented $A$-algebra that is smooth outside $\rm{V}(I)$. We put $A_0\coloneqq A/J$ and consider a diagram
\begin{equation}\label{eqn:deformation}
\begin{tikzcd}
Z_0 = \Spec C_0 \arrow[r, hook, "i_Z"] \arrow{d}{f_0} & Z = \Spec C \arrow{dd}{h} \\
Y_0 = \Spec B_0 \arrow{d}{g_0} & \\
X_0 = \Spec A_0 \arrow[r, hook, "i_X"] & X = \Spec A
\end{tikzcd}
\end{equation}
such that $i_X$ is the natural closed immersion, the ambient square is cartesian (in particular, this implies that $C_0\simeq C/JC$ as an $A_0$-algebra), $Y_0 \to X_0$ is a finitely presented morphism that is smooth outside $\rm{V}(IA_0)$, and $Z_0 \to Y_0$ is an \'etale morphism that induces an isomorphism $B_0/IB_0 \xr{\sim} C_0/IC_0$. Then we can fill-in Diagram~\ref{eqn:deformation} to a commutative diagram
\begin{equation*}\label{eqn:deformation-2}
\begin{tikzcd}
Z_0 = \Spec C_0 \arrow[r, hook, "i_Z"] \arrow{d}{f_0} & Z = \Spec C \arrow[dd, bend left = 90, "h"] \arrow{d}{f} \\
Y_0 = \Spec B_0 \arrow{d}{g_0} \arrow[r, hook] & Y = \Spec B \arrow{d}{g} \\
X_0 = \Spec A_0 \arrow[r, hook, "i_X"] & X = \Spec A,
\end{tikzcd}
\end{equation*}
such that each square is cartesian, $Y \to X$ is a finitely presented morphism that is smooth outside $\rm{V}(I)$, and $Z \to Y$ is an \'etale morphism that induces an isomorphism $B/IB \xr{\sim} C/IC$. 
\end{lemma}
\begin{proof}
    First, we denote by $U\subset X$, $U_0\subset X_0$, $V_0\subset Y_0$, $W\subset Z$, and $W_0\subset Z_0$ the open complements of the vanishing loci of the ideal $I$. We denote by $\rm{Def}_{X_0 \hookrightarrow X}(W_0)$ (resp. $\rm{Def}_{X_0 \hookrightarrow X}(V_0)$) the set of isomorphism classes of flat $X$-lifts of $W_0$ (resp. of $V_0$); see \cite[(8.5.7)]{FGA-EXP} for more detail. For a smooth morphism $S \to S'$, we denote by $\rm{T}_{S/S'}$ the relative tangent bundle. 
    
    Now the classical deformation theory (see \cite[Th.~8.5.9(b)]{FGA-EXP}) implies that the obstruction to lifting $W_0$ (resp. $V_0$) to a flat $X$-scheme is given by a class in $\rm{H}^2(W_0, \rm{T}_{W_0/X_0} \otimes_{A_0} J)$ (resp. $\rm{H}^2(V_0, \rm{T}_{V_0/X_0} \otimes_{A_0} J)$) and, if this class vanishes, the set of isomorphism classes of flat $X$-lifts is a torsor under the group $\rm{H}^1(W_0, \rm{T}_{W_0/X_0} \otimes_{A_0} J)$ (resp. $\rm{H}^1(V_0, \rm{T}_{V_0/X_0} \otimes_{A_0} J)$). Since $f_0$ is \'etale, we conclude $\rm{T}_{Y_0/X_0} \otimes_{B_0} C_0 \simeq \rm{T}_{Z_0/X_0}$. Therefore, Lemma~\ref{lemma:miracle-equiality} implies that the natural morphism 
    \begin{equation}\label{eqn:iso}
    \rm{H}^i(V_0, \rm{T}_{V_0/X_0} \otimes_{A_0} J) \to \rm{H}^i(W_0, \rm{T}_{W_0/X_0} \otimes_{A_0} J)
    \end{equation}
    is an isomorphism for $i\geq 1$. Thus, (\ref{eqn:iso}) applied to $i=2$ and the functoriality of the obstruction class (see \cite[Rmk.~8.5.10(a)]{FGA-EXP}) imply that the obstruction to the existence of a flat $X$-lift of $V_0$ vanishes (since $W_0$ admits a flat $X$-lift $W$). 
    
    Topological invariance of the small \'etale site (see \cite[\href{https://stacks.math.columbia.edu/tag/04DZ}{Tag 04DZ}]{stacks-project}) implies that any flat $X$-lift of $V_0$ uniquely defines a flat $X$-lift of $W_0$. This induces a map $\alpha\colon \rm{Def}_{X_0\hookrightarrow X}(V_0) \to \rm{Def}_{X_0 \hookrightarrow X}(W_0)$. Now (\ref{eqn:iso}) applied to $i=1$ implies that both of these sets are (compatibly) torsors under the same group. Since both torsors are non-empty, we conclude that $\alpha$ is a bijection. In other words, we can find a flat $X$-lifting $V$ of $V_0$ which fits in a commuting diagram
    \[
    \begin{tikzcd}[column sep = 5em]
        W_0  \arrow[r, hook, "(i_{Z})|_{W_0}"] \arrow[d, swap, "(f_0)|_{W_0}"] & W \arrow{d}{f'} \\
        V_0  \arrow[r, hook] \arrow[d, swap, "(g_0)|_{V_0}"] & V  \arrow{d}{g'}\\
        X_0 \arrow[r, hook, "i_X"] & X 
    \end{tikzcd}
    \]
    such that all squares are cartesian. 

    We wish to lift the whole $X_0$-scheme $Y_0$ to an $X$-scheme $Y$. This is quite subtle because $Y_0$ might not be a flat $X_0$-scheme, so the usual deformation theory does not apply in this case. Instead, we use the theory of algebraic spaces. Namely, we note that $W \to V$ is a separated \'etale morphism due to \cite[\href{https://stacks.math.columbia.edu/tag/06AG}{Tag 06AG}]{stacks-project}. Therefore, \cite[\href{https://stacks.math.columbia.edu/tag/0DVJ}{Tag 0DVJ}]{stacks-project} implies that we can construct an algebraic space $Y$ over $X$ as a pushout of the diagram
    \begin{equation}\label{eqn:pushout}
    \begin{tikzcd}
    W \arrow[r, hook] \arrow{d}{f'} & Z \arrow{d}{f} \\
    V \arrow[r, hook] & Y.
    \end{tikzcd}
    \end{equation}
    Furthermore, {\it loc.cit.} implies that (\ref{eqn:pushout}) is an elementary distinguished square in the sense of \cite[\href{https://stacks.math.columbia.edu/tag/08GM}{Tag 08GM}]{stacks-project}. In particular, the morphism $f$ is \'etale. Then \cite[\href{https://stacks.math.columbia.edu/tag/08GN}{Tag 08GN}]{stacks-project} and \cite[\href{https://stacks.math.columbia.edu/tag/0DVI}{Tag 0DVI}]{stacks-project} imply that $Y\times_X X_0 \simeq Y_0$ as $X_0$-schemes. Since $X_0 \to X$ is a nilpotent thickening, \cite[\href{https://stacks.math.columbia.edu/tag/07VT}{Tag 07VT}]{stacks-project} implies that $Y=\Spec B$ is an affine $X$-scheme. By construction, the affine $X$-scheme $Y$ fits into the following diagram
    \[
    \begin{tikzcd}
        Z_0 = \Spec C_0 \arrow[r, hook, "i_Z"] \arrow{d}{f_0} & Z = \Spec C \arrow{d}{f} \arrow[dd, bend left = 90, "h"] \\
        Y_0 = \Spec B_0 \arrow{d}{g_0} \arrow[r, hook] & Y = \Spec B \arrow{d}{g} \\
        X_0 = \Spec A_0 \arrow[r, hook, "i_X"] & X = \Spec A
    \end{tikzcd}
    \]
    such that each square is cartesian and $f$ is \'etale. Now we note that $h\colon Z \to X$ is finitely presented by assumption, while $g'=g|_V \colon V \to X$ is smooth due to \cite[\href{https://stacks.math.columbia.edu/tag/01V8}{Tag 01V8}]{stacks-project}. Therefore, \cite[\href{https://stacks.math.columbia.edu/tag/036N}{Tag 036N}]{stacks-project} applied to the fppf covering $Z\sqcup V \to Y$ implies $g$ is finitely presented. Since $V = Y \smallsetminus \rm{V}(I)$, we conclude that $g$ is smooth outside of $\rm{V}(I)$. Equivalently, $A \to B$ is smooth outside $\rm{V}(I)$. Finally, \cite[\href{https://stacks.math.columbia.edu/tag/051H}{Tag 051H}]{stacks-project} implies that the natural morphism $B/IB \to C/IC$ is an isomorphism since so is $B/(I+J)B=B_0/IB_0 \to C_0/IC_0 = C/(I+J)C$. This finishes the proof. 
\end{proof}

Before we start the proof of the main result of this section, we need to verify the following basic lemmas.

\begin{lemma}\label{lemma:fully-faithful} Let $A$ be a noetherian ring, let $J=(f_1, \dots, f_s) \subset I = (f_1, \dots, f_s, \dots, f_r)$ be ideals in $A$, and let $B$ be a $J$-adically complete $A$-algebra such that $B/J$ is a finite type $A/J$-algebra. Then the morphism $B \to B^{\wedge}_I \times \prod_{i=s+1}^r B[\frac{1}{f_i}]^{\wedge}_J$ is faithfully flat.
\end{lemma}
\begin{proof}
    First, we note that \cite[\href{https://stacks.math.columbia.edu/tag/0AJQ}{Tag 0AJQ}]{stacks-project} implies that $B$ is noetherian. Therefore, the morphism $B \to B^{\wedge}_I \times \prod_{i=s+1}^r B[\frac{1}{f_i}]^{\wedge}_J$ is flat. Therefore, \cite[\href{https://stacks.math.columbia.edu/tag/00HQ}{Tag 00HQ}]{stacks-project} ensures that it suffices to show that every closed point of $\Spec B$ lies in the image of the map 
    \[
    \Spec \Big( B^{\wedge}_I \times \prod_{i=s+1}^r B[\frac{1}{f_i}]^{\wedge}_J\Big) = \Spec B^{\wedge}_I \bigsqcup \Big(\bigsqcup_{i=s+1}^r \Spec B[\frac{1}{f_i}]^{\wedge}_J\Big) \to \Spec B.
    \]
    First, we note \cite[\href{https://stacks.math.columbia.edu/tag/05GI}{Tag 05GI}]{stacks-project} implies that any maximal ideal $\m \subset B$ contains the ideal $J$. If, furthermore,  $I\subset \m$, then the corresponding closed point $\Spec B$ lies in the image of $\Spec B^{\wedge}_I \to \Spec B$. Otherwise, $J\subset \m$ and there is an index $i\in \{s+1, \dots, r\}$ such that $f_i\notin \m$. In this case, the closed point of $\Spec B$ corresponding to $\m$ lies in the image of the map $\Spec  B[\frac{1}{f_i}]^{\wedge}_J \to \Spec B$. Therefore, $B \to B^{\wedge}_I \times \prod_{i=s+1}^r B[\frac{1}{f_i}]^{\wedge}_J$ is indeed fully faithful. 
\end{proof}

\begin{lemma}\label{lemma:smoothify} Let $A$ be a noetherian ring, let $J\subset I$ be two ideals, and let $B$ be a $J$-adically complete $A$-algebra such that $B/J$ is a finite type $A/J$-algebra. Put $A_n=A/J^n$ and $B_n=B/J^nB$. If $B^{\wedge}_I$ is rig-smooth\footnote{Since $B^{\wedge}_I/IB^{\wedge}_I \simeq B/IB$ is a finite type $A/I$-algebra, the notion of rig-smoothness over $(A, I)$ is well-defined for $B^{\wedge}_I$.} over $(A, I)$ and $B_n$ is smooth outside $\rm{V}(IA_n)$ for every $n\geq 1$, then $B$ is rig-smooth over $(A, J)$. 
\end{lemma}
\begin{proof}
    We refer to \cite[\href{https://stacks.math.columbia.edu/tag/0AJL}{Tag 0AJL}]{stacks-project} for the construction of the naive (completed) cotangent complexes $(\rm{NL}_{B/A})^{\wedge}_{J}$ and $(\rm{NL}_{B^{\wedge}_I/A})^{\wedge}_I$. We choose some compatible generators $J=(f_1, \dots, f_s)$ and $I=(f_1, \dots, f_s, \dots, f_r)$. 
    
    For brevity, we denote by $\rm{N}^{-1}$ and $\rm{N}^0$ the cohomology modules of $(\rm{NL}_{B/A})^{\wedge}_{J}$. Then \cite[\href{https://stacks.math.columbia.edu/tag/0GAJ}{Tag 0GAJ}]{stacks-project} implies that it suffices to show that $\rm{N}^{-1}$ is $J$-power torsion and $\rm{N}^{0}[\frac{1}{f_j}]$ is projective for $j=1, \dots, s$. We will prove a stronger claim that $\rm{N}^{-1}$ is $I$-power torsion and $\rm{N}^{0}[\frac{1}{f_j}]$ is projective for $j=1, \dots, r$. For this, we consider the morphism
    \[
    \alpha \colon B \to B^{\wedge}_I \times \prod_{i=s+1}^r B[\frac{1}{f_i}]^{\wedge}_J.
    \]
    Lemma~\ref{lemma:fully-faithful} implies that $\alpha$ is faithfully flat. Therefore, faithfully flat descent ensures that it suffices to show that $\rm{N}^{-1}\otimes_B B^{\wedge}_I$ and $\rm{N}^{-1}\otimes_B B[\frac{1}{f_i}]^{\wedge}_J$ are $I$-power torsion for any $i=s+1, \dots, r$, and that $(\rm{N}^0 \otimes_B B^{\wedge}_I)[\frac{1}{f_j}]$ and $(\rm{N}^0\otimes_B B[\frac{1}{f_i}]^{\wedge}_J)[\frac{1}{f_j}]$ are projective for any $j=1, \dots, r$ and any $i=s+1, \dots, r$. 
    
    Using the explicit description of $\rm{NL}^{\wedge}$ and the basic properties of completions, we see that $(\rm{NL}_{B/A})^{\wedge}_{J}$ lies in $D_{coh}^{[-1, 0]}(B)$ and there are isomorphisms
    \begin{equation}\label{eqn:base-change-cotangent-complex}
    (\rm{NL}_{B/A})^{\wedge}_{J} \otimes_B B^{\wedge}_I \simeq (\rm{NL}_{B^{\wedge}_I/A})^{\wedge}_I,
    \end{equation}
    \begin{equation}\label{eqn:base-change-cotangent-complex-2}
    (\rm{NL}_{B/A})^{\wedge}_{J} \otimes_B \Big(B\big[\frac{1}{f_i}\big]\Big)^{\wedge}_J\simeq (\rm{NL}_{(B[\frac{1}{f_i}])^{\wedge}_J/A})^{\wedge}_J
    \end{equation}
    for $i=s+1, \dots, r$.

    We first deal with $\rm{N}^{-1}\otimes_B B^{\wedge}_I$ and $(\rm{N}^{0} \otimes_B B^{\wedge}_I)[\frac{1}{f_j}]$. Since $B^{\wedge}_I$ is assumed to be rig-smooth over $(A, I)$, \cite[\href{https://stacks.math.columbia.edu/tag/0GAJ}{Tag 0GAJ}]{stacks-project} and (\ref{eqn:base-change-cotangent-complex}) imply\footnote{More precisely, \cite[\href{https://stacks.math.columbia.edu/tag/0GAJ}{Tag 0GAJ}]{stacks-project} implies that rig-smoothness is equivalent to condition $(5)$ of \cite[\href{https://stacks.math.columbia.edu/tag/0G9K}{Tag 0G9K}]{stacks-project}. Furthermore, the proof of \cite[\href{https://stacks.math.columbia.edu/tag/0G9K}{Tag 0G9K}]{stacks-project} implies that this conditions is equivalent to the condition $(5)$ for any \emph{fixed} set of generators $I=(f_1, \dots, f_s)$.} that $\rm{N}^{-1} \otimes_B B^{\wedge}_I$ is $I$-power torsion and $(\rm{N}^{0} \otimes_B B^{\wedge}_I)[\frac{1}{f_j}]$ is projective for any $j=1, \dots, r$. Then we are only left to deal with $\rm{N}^{-1}\otimes_B B[\frac{1}{f_i}]^{\wedge}_J$ and $(\rm{N}^0\otimes_B B[\frac{1}{f_i}]^{\wedge}_J)[\frac{1}{f_j}]$. We will show a stronger claim that $\rm{N}^{-1}\otimes_B B[\frac{1}{f_i}]^{\wedge}_J=0$ and $\rm{N}^{0}\otimes_B B[\frac{1}{f_i}]^{\wedge}_J$ is projective for any $i=s+1, \dots, r$. First, \cite[\href{https://stacks.math.columbia.edu/tag/0AJS}{Tag 0AJS}]{stacks-project}, and \cite[\href{https://stacks.math.columbia.edu/tag/07BR}{Tag 07BR}]{stacks-project} imply that 
    \[
    (\rm{NL}_{A[\frac{1}{f_i}]^{\wedge}_J/A})^{\wedge}_J \simeq \rR \lim_n \rm{NL}_{A_n[\frac{1}{f_i}]/A_n} \simeq 0
    \]
    for every $i=s+1, \dots, r$. Therefore, \cite[\href{https://stacks.math.columbia.edu/tag/0ALM}{Tag 0ALM}]{stacks-project} implies that $(\rm{NL}_{B[\frac{1}{f_i}]^{\wedge}_J/A})^{\wedge}_J \simeq (\rm{NL}_{B[\frac{1}{f_i}]^{\wedge}_J/A[\frac{1}{f_i}]^{\wedge}_J})^{\wedge}_J$ for any $i=s+1, \dots, r$. Thus, the assumption that each $B_n$ is smooth outside $\rm{V}(IA_n)$, \cite[\href{https://stacks.math.columbia.edu/tag/07BU}{Tag 07BU}]{stacks-project}, and \cite[\href{https://stacks.math.columbia.edu/tag/0AJS}{Tag 0AJS}]{stacks-project} imply that 
    \[
    (\rm{NL}_{B[\frac{1}{f_i}]^{\wedge}_J/A})^{\wedge}_J \simeq (\rm{NL}_{B[\frac{1}{f_i}]^{\wedge}_J/A[\frac{1}{f_i}]^{\wedge}_J})^{\wedge}_J \simeq \rR\lim_n \rm{NL}_{B_n[\frac{1}{f_i}]/A_n[\frac{1}{f_i}]} = \rR\lim_n \Omega^1_{B_n[\frac{1}{f_i}]/A_n[\frac{1}{f_i}]},
    \]
    where each $\Omega^1_{B_n[\frac{1}{f_i}]/A_n[\frac{1}{f_i}]}$ is finite projective. Therefore,  \cite[\href{https://stacks.math.columbia.edu/tag/00RV}{Tag 00RV}]{stacks-project} ensures that $\Omega^1_{B_n[\frac{1}{f_i}]/A_n[\frac{1}{f_i}]} \otimes_{B_n[\frac{1}{f_i}]} B_{n-1}[\frac{1}{f_i}] \simeq \Omega^1_{B_n[\frac{1}{f_i}]/A_n[\frac{1}{f_i}]} \otimes^L_{B_n[\frac{1}{f_i}]} B_{n-1}[\frac{1}{f_i}] \simeq \Omega^1_{B_{n-1}[\frac{1}{f_i}]/A_{n-1}[\frac{1}{f_i}]}$. Hence, the combination of \cite[\href{https://stacks.math.columbia.edu/tag/091D}{Tag 091D}]{stacks-project}, \cite[\href{https://stacks.math.columbia.edu/tag/0912}{Tag 0912}]{stacks-project}, and \cite[\href{https://stacks.math.columbia.edu/tag/0CQF}{Tag 0CQF}]{stacks-project} implies that
    \[
    (\rm{NL}_{B[\frac{1}{f_i}]^{\wedge}_J/A})^{\wedge}_J \simeq \rR\lim_n \Omega^1_{B_n[\frac{1}{f_i}]/A_n[\frac{1}{f_i}]} \simeq \lim_n \Omega^1_{B_n[\frac{1}{f_i}]/A_n[\frac{1}{f_i}]}
    \]
    is finite projective $B[\frac{1}{f_i}]^{\wedge}_J$-module concentrated in degree $0$. In other words, $\rm{H}^{-1}\big((\rm{NL}_{B[\frac{1}{f_i}]^{\wedge}_J/A})^{\wedge}_J\big)=0$ and $\rm{H}^0\big((\rm{NL}_{B[\frac{1}{f_i}]^{\wedge}_J/A})^{\wedge}_J\big)$ is a (finite) projective module for any $i=s+1, \dots, r$. Combining it with (\ref{eqn:base-change-cotangent-complex-2}), we conclude that $\rm{N}^{-1}\otimes_B B[\frac{1}{f_i}]^{\wedge}_J=0$ and $\rm{N}^{0}\otimes_B B[\frac{1}{f_i}]^{\wedge}_J$ is projective for any $i=s+1, \dots, r$. This finishes the proof.
\end{proof}

We also need to discuss the relation between rig-smoothness in the sense of \cite[\href{https://stacks.math.columbia.edu/tag/0GAI}{Tag 0GAI}]{stacks-project} and formal smoothness outside $\rm{V}(I)$ in the sense of \cite[p.~581]{Elkik}. 

\begin{lemma}\label{lemma:rig-smooth-vs-formally-smooth-outside} Let $A$ be a noetherian ring, let $I\subset A$ be an ideal, and let $B$ be an $I$-adically complete $A$-algebra suhc that $B/IB$ is a finite type $A/I$-algebra. Then $B\simeq A^{\wedge}_I\langle T_1, \dots, T_n\rangle/(g_1, \dots, g_m)$ for some $g_1, \dots, g_m \in A^{\wedge}_I\langle T_1, \dots, T_n\rangle$. Furthermore, if $B$ is rig-smooth over $(A, I)$, then the $A^{\wedge}_I$-algebra $B$ is formally smooth outside of $\rm{V}(I)$ in the sense of \cite[p.~581]{Elkik}.
\end{lemma}
\begin{proof}
    The first claim follows immediately from \cite[\href{https://stacks.math.columbia.edu/tag/0AJP}{Tag 0AJP}]{stacks-project}. Therefore, we only need to show that $B$ is formally smooth outside of $\rm{V}(I)$ if $B$ is rig-smooth over $(A, I)$. 
    
    For this, we choose a presentation $B\simeq A^{\wedge}_I\langle T_1, \dots, T_n\rangle/J$ for some ideal $J \subset A^{\wedge}_I\langle T_1, \dots, T_n\rangle$, pick a point $x\in \Spec B \smallsetminus \rm{V}(IB)$ and elements $b_1, \dots, b_s\in B$ as in the equivalent condition $(4)$ of \cite[\href{https://stacks.math.columbia.edu/tag/0GAJ}{Tag 0GAJ}]{stacks-project}. Then there is an index $l\in \{1, \dots, s\}$ such that $x\in \Spec B[\frac{1}{b_l}] \subset \Spec B$. Then \emph{loc.~cit.} constructs elements $f_1, \dots, f_d \in J$ and a subset $T\subset \{1, \dots, n\}$ with $|T|=d$ such that $\det_{i\leq d, j\in T}(\frac{\partial f_i}{\partial T_j})$ divides $b_l$ in $B$ and $b_l J \subset (f_1, \dots, f_d) + J^2$.\footnote{More precisely, \cite[\href{https://stacks.math.columbia.edu/tag/0GAJ}{Tag 0GAJ}]{stacks-project} shows that the $B$-module $J/\big(J^2 + (f_1, \dots, f_d)\big)$ is annihilated by $b_l$.} Using that $x\in \Spec B[\frac{1}{b_l}]$, we conclude that the first condition implies that the minors of order $d$ of the Jacobian matrix $(\frac{\partial f_i}{\partial T_j})_{i=1,\dots, d;\, j=1, \dots, n}$ generate the unit ideal around the point $x\in \Spec B$, while the second condition implies that the elements $(f_1, \dots, f_d)$ generate the ideal $J$ around the point $x\in \Spec A^{\wedge}_I\langle T_1, \dots, T_n\rangle$. Since $x\in \Spec B \smallsetminus \rm{V}(IB)$ was an arbitrary point, we conclude that $B$ is formally smooth outside of $\rm{V}(I)$.
\end{proof}

\begin{thm}[Noetherian rig-smooth algebraization]\label{thm:noetherian-approximation} Let $A$ be a noetherian ring, and let $I\subset A$ be an ideal. Let $B$ be an $I$-adically complete $A$-algebra such that $B/IB$ is a finite type $A/I$-algebra and $B$ is rig-smooth over $(A, I)$ (in the sense of \cite[\href{https://stacks.math.columbia.edu/tag/0GAI}{Tag 0GAI}]{stacks-project}). Then there is a finite type $A$-algebra $C$ such that $C$ is smooth outside $\rm{V}(I)$ and there is an isomorphism $C^{\wedge}_I \simeq B$ of $A$-algebras.
\end{thm}
\begin{proof}
    First, we note that the natural morphism $A \to B$ uniquely factors as $A \to A^\h_I \to B$. Then we note that $B$ is rig-smooth over $(A^\h_I, IA^\h_I)$ (this follows directly from \cite[\href{https://stacks.math.columbia.edu/tag/0GAI}{Tag 0GAI}]{stacks-project}). Now we choose some generators $I=(f_1, \dots, f_r)$ and argue inductively on the number of generators $r$. If $r=1$, Lemma~\ref{lemma:rig-smooth-vs-formally-smooth-outside} implies that $B$ is formally smooth outside $\rm{V}(I)$ in the sense of \cite[p.~581]{Elkik}. Therefore, \cite[Th.~7 on p.~582]{Elkik} implies that there is a finite type $A^\h_I$-algebra $C'$ such that it is smooth outside $\rm{V}(I)$ and there is an isomorphism $(C')^{\wedge}_I\simeq B$ of $A^\h_I$-algebras. Then a standard approximation argument (similar to \cite[\href{https://stacks.math.columbia.edu/tag/0GAS}{Tag 0GAS}]{stacks-project}) implies that there is a finite type $A$-algebra $C$ such that it is smooth outside $\rm{V}(I)$ with an isomorphism $C^{\wedge}_I \simeq B$ of $A$-algebras. 

    Now we do the induction step. We assume that $r>1$ and the result is known for all ideals generated by less than $r$ elements. Then we put $I_2\coloneqq (f_2, \dots, f_r)$, and we also put $A_n\coloneqq A/(f_1^n)$ and $B_n=B/(f_1^n)$ for any integer $n\geq 1$. We note that $B_n$ is $I$-adically complete due to \cite[\href{https://stacks.math.columbia.edu/tag/0AJQ}{Tag 0AJQ}]{stacks-project} and that $B_n$ is rig-smooth over $(A_n, IA_n)$ due to \cite[\href{https://stacks.math.columbia.edu/tag/0GAM}{Tag 0GAM}]{stacks-project}. Since the assertion depends only on $\rm{rad}(I)$ and $\rm{rad}(IA_n) = \rm{rad}(I_2 A_n)$, we can apply the inductive hypothesis to $A_n \to B_n$ to find finite type $A_n$-algebras $C_n$ smooth outside $\rm{V}(IA_n)$ together with isomorphisms $\alpha_n\colon (C_n)^{\wedge}_I \xr{\sim} B_n$ of $A_n$-algebras (equivalently, of $A$-algebras). Since $B_n/(f_1^{n-1}) = B_{n-1}$, we use the isomorphisms $\alpha_n$ to get isomorphisms 
    \[
    \ov{\varphi}_n\colon (C_n)^{\wedge}_I/(f_1^{n-1}) \xr{\sim} (C_{n-1})^{\wedge}_I
    \]
    of $A$-algebras. We denote by $\varphi_n\colon (C_n)^{\wedge}_I \to (C_{n-1})^{\wedge}_I$ the induced morphisms. We also denote by 
    \[
    \pi_n\colon B_n \to B_{n-1}
    \]
    the ``reduction by $f_1^{n-1}$'' morphism. Now we wish to modify $\alpha_n$ and $C_n$ such that all $\varphi_n$ can be ``decompleted'' to morphisms $C_{n} \to C_{n-1}$. 

    {\it Step~$1$. We partially ``decomplete'' $\varphi_n$ to achieve that they come from maps $(C_n)^\h_I \to (C_{n-1})^\h_I$}. To do this, we note that Lemma~\ref{lemma:approximating-maps} ensures that, for each integer $n>1$, we can find an $A$-linear isomorphism 
    \[
    \ov{\psi}_n\colon (C_{n})^\h_I/(f_1^{n-1}) \simeq \big(C_{n}/(f_1^{n-1})\big)^\h_I \xr{\sim} (C_{n-1})^\h_I
    \]
    such that $(\ov{\psi}_n)^{\wedge}_I \,\rm{ mod }\, I^n = \ov{\varphi}_n \,\rm{ mod }\, I^n$. We denote by $\psi_n \colon (C_n)^\h_I \to (C_{n-1})^\h_I$ the induced morphism. For each $n>m\geq 1$, we denote by 
    \[
    \psi_{n, m} \colon (C_n)^\h_I \to (C_m)^\h_I,
    \]
    \[
    \varphi_{n, m} \colon (C_n)^{\wedge}_I \to (C_m)^{\wedge}_I,
    \]
    \[
    \pi_{n, m} \colon B_n \to B_{m}
    \]
    the evident composition of $\psi_i$, $\varphi_i$, and $\pi_i$ for $i=n, \dots, m+1$ respectively. These morphisms induce isomorphisms $\ov{\psi}_{n,m} \colon (C_{n})^\h_I/(f_1^{m})\simeq \big(C_{n}/(f_1^{m})\big)^\h_I \xr{\sim} (C_{m})^\h_I$ and $\ov{\varphi}_{n, m} \colon (C_{n})^{\wedge}_I/(f_1^{m})\simeq\big(C_{n}/(f_1^{m})\big)^{\wedge}_I \xr{\sim} (C_{m})^{\wedge}_I$. Then we define 
    \[
    \theta_{n, m} \coloneqq  \ov{\varphi}_{n, m} \circ \big((\ov{\psi}_{n, m})^{\wedge}_I\big)^{-1} \colon (C_{m})^{\wedge}_I \xr{\sim} (C_{m})^{\wedge}_I
    \]
    for any $n>m\geq 1$. By construction, the diagram
    \begin{equation}\label{eqn:approximate-finite}
    \begin{tikzcd}[column sep = 5em]
    (C_n)^{\wedge}_I  \arrow[r, swap, "\sim"]\arrow{d}{(\psi_{n,m})^{\wedge}_I} \arrow{r}{\alpha_n} & B_n \arrow{d}{\pi_{n,m}} \\
    (C_m)^{\wedge}_I \arrow[r, swap, "\sim"] \arrow{r}{\alpha_{m} \circ \theta_{n, m}} & B_{m}
    \end{tikzcd}
    \end{equation}
    commutes for any $n>m\geq 1$. Now, for any $n\geq 1$ and an element $x\in (C_n)^{\wedge}_I$, the sequence
    \[
    x, \theta_{n+1, n}(x), \theta_{n+2, n}(x), \theta_{n+3, n}(x), \dots
    \]
    is Cauchy with respect to the $I$-adic uniform structure because $(\psi_{n+i})^{\wedge}_I \,\rm{ mod }\, I^{n+i} = \varphi_{n+i} \,\rm{ mod }\, I^{n+i}$ for any integer $i\geq 0$. Since each $(C_n)^{\wedge}_I$ is $I$-adically complete, we can define 
    \[
    \theta_n(x) = \lim_{m \to \infty} \theta_{n+m, n}(x) \in (C_n)^{\wedge}_I
    \]
    for any $x\in (C_n)^{\wedge}_I$. One easily checks that $\theta_n\colon (C_n)^{\wedge}_I \to (C_n)^{\wedge}_I$ is an isomorphism of $A$-algebras. Using (\ref{eqn:approximate-finite}) and a standard limit argument, we conclude that the diagram
    \[
    \begin{tikzcd}[column sep = 5em]
    (C_n)^{\wedge}_I \arrow{d}{(\psi_{n})^{\wedge}_I} \arrow[r, swap, "\sim"] \arrow{r}{\alpha_n\circ \theta_n} & B_n \arrow{d}{\pi_n} \\
    (C_{n-1})^{\wedge}_I \arrow[r, swap, "\sim"]\arrow{r}{\alpha_{n-1} \circ \theta_{n-1}} & B_{n-1}
    \end{tikzcd}
    \]
    commutes for any $n\geq 2$. Therefore, we can replace each $\alpha_n$ by $\alpha_n \circ \theta_n$ to achieve that the diagram
    \[
    \begin{tikzcd}[column sep = 5em]
    (C_n)^{\h}_I \arrow{r} \arrow{d}{\psi_n}& (C_n)^{\wedge}_I \arrow{d}{(\psi_{n})^{\wedge}_I} \arrow[r, swap, "\sim"] \arrow{r}{\alpha_n} & B_n \arrow{d}{\pi_n} \\
    (C_{n-1})^{\h}_I \arrow{r} & (C_{n-1})^{\wedge}_I \arrow[r, swap, "\sim"] \arrow{r}{\alpha_{n-1}} & B_{n-1}
    \end{tikzcd}
    \]
    commutes for any $n\geq 2$. 
    
    {\it Step~$2$. We ``dehenselize'' $\psi_n$ to achieve that they come from maps $C_n \to C_{n-1}$ (possibly after changing $C_n$)}. More precisely, we show that there exist finite type $A_n$-algebras $C'_n$ smooth outside $\rm{V}(IA_n)$ with isomorphisms $\beta_n\colon (C_n)^\h_I \xr{\sim} (C'_n)^\h_I$ of $A_n$-algebras for $n\geq 1$ and morphisms of $A_n$-algebras $\rho_n\colon C'_n \to C'_{n-1}$ for $n\geq 2$ such that $(\rho_n)^\h_I = \psi_n$ for $n\geq 2$ (using the identifications $\beta_n$). 

    We construct $C'_n$, $\beta_n$, and $\rho_n$ inductively on $n$. For $n=1$, we put $C'_1=C_1$ and $\beta_1=\rm{id}$. Now we suppose that $n>1$ and that we have constructed $C'_i$, $\beta_i$, and $\rho_i$ for $i\leq n-1$, and try to construct it for $n$. For the purpose of this construction, we can safely replace $C_i$ by $C'_i$ to achieve that $C_i=C'_i$ and $\beta_i=\rm{id}$ for $i\leq n-1$. 

    First, Lemma~\ref{lemma:isoms-on-completions} ensures that $\psi_n$ induces an isomorphism $\ov{\psi}_n\colon \big(C_n/(f_1^{n-1})\big)^\h_I \xr{\sim} (C_{n-1})^\h_I$. Then a simple approximation argument implies that there is a finite type $A_{n-1}$-algebra $D_{n-1}$ and $A$-linear \'etale morphisms $a_{n-1}\colon C_{n-1} \to D_{n-1}$ and $b_{n-1}\colon C_{n}/(f_1^{n-1}) \to D_{n-1}$ that induce isomorphisms $C_{n-1}/IC_{n-1} \xr{\sim} D_{n-1}/ID_{n-1}$ and $C_n/(f_1^{n-1}, IC_n)\xr{\sim} D_{n-1}/ID_{n-1}$ and such that
    \[
    \big((a_{n-1})^{\h}_I\big)^{-1}\circ (b_{n-1})^{\h}_I = \ov{\psi}_n \colon \big(C_n/(f_1^{n-1})\big)^\h_I \to (C_{n-1})^\h_I.
    \]
    We note that $(a_{n-1})^{\h}_I$ is an isomorphism since it is \'etale and induces an isomorphism on mod-$I$ fibers, so the morphism $\big((a_{n-1})^{\h}_I\big)^{-1}$ in the formula above is well-defined. 
    
    The topological invariance of the small \'etale site (see \cite[\href{https://stacks.math.columbia.edu/tag/039R}{Tag 039R}]{stacks-project}) imply that we can find lift the morphism $b_{n-1}\colon C_n/(f_1^{n-1}) \to D_{n-1}$ to an \'etale $A$-linear morphism $b_n\colon C_n \to D_n$. Then \cite[\href{https://stacks.math.columbia.edu/tag/051H}{Tag 051H}]{stacks-project} ensures that $C_n/IC_n \to D_n/ID_n$ is an isomorphism, thus we have 
    \[
    \big((a_{n-1})^{\h}_I\big)^{-1} \circ (r_n)^\h_I\circ (b_{n})^{\h}_I = \psi_n \colon (C_n)^\h_I \to (C_{n-1})^\h_I,
    \]
    where $r_n\colon D_n \to D_{n-1}$ is the natural reduction morphism. Using Lemma~\ref{lemma:lift-rig-smooth}, we can find a commutative diagram
    \[
    \begin{tikzcd}
    D_n \arrow{r}{r_n} & D_{n-1} \\
    C'_n \arrow{u}{a_n} \arrow{r}{\rho_n}& C_{n-1} \arrow{u}{a_{n-1}}\\
    A_n\arrow{u} \arrow{r} & A_{n-1} \arrow{u}
    \end{tikzcd}
    \]
    such that each square is a push-out square, $C'_n$ is a finite type $A_n$-algebra smooth away from $\rm{V}(IA_n)$, and $a_n$ is an \'etale morphism inducing an isomorphism $C'_n/IC'_n \to D_n/ID_n$. Then we put $\beta_n \coloneqq  \big((a_n)^\h_I\big)^{-1} \circ (b_n)^\h_I\colon (C_n)^\h_I \xr{\sim} (C'_n)^\h_I$; this is evidently an isomorphism of $A$-algebras. Then we have
    \[
    (\rho_n)^\h_I \circ \beta_n = (\rho_n)^\h_I \circ \big((a_n)^\h_I\big)^{-1} \circ (b_n)^\h_I = \big((a_{n-1})^{\h}_I\big)^{-1} \circ (r_n)^\h_I \circ (b_n)^\h_I = \psi_n.
    \]
    Therefore, $C'_n$, $\rho_n$, and $\beta_n$ as defined above do the job. Thus induction finishes the proof of Step~$2$. 

    Now we can replace each $C_n$ by $C'_n$ to achieve that we are in the situation where we have $A$-linear morphisms $\rho_n\colon C_n \to C_{n-1}$ fitting into a commutative diagram
    \[
    \begin{tikzcd}[column sep = 5em]
    C_n \arrow{r} \arrow{d}{\rho_n}& (C_n)^{\wedge}_I \arrow{d}{(g_{n})^{\wedge}_I} \arrow[r, swap, "\sim"] \arrow{r}{\alpha_n} & B_n \arrow{d}{\pi_n} \\
    C_{n-1} \arrow{r} & (C_{n-1})^{\wedge}_I \arrow[r, swap, "\sim"] \arrow{r}{\alpha_{n-1}} & B_{n-1}.
    \end{tikzcd}
    \]
    {\it End of proof.} Now we note that $\widetilde{C} \coloneqq \lim_n C_n$ is an $(f_1)$-adically complete $A$-algebra due to \cite[\href{https://stacks.math.columbia.edu/tag/0AJP}{Tag 0AJP}]{stacks-project}. Also, {\it loc.~cit.} and our assumption on $C_n$ imply that, for any $n\geq 1$, the finite type $A_n$-algebra $\widetilde{C}/(f_1^n)\simeq C_n$ is smooth outside $\rm{V}(IA_n)$. Furthermore, using the following sequence of isomorphisms of $A$-algebras 
    \[
    \widetilde{C}^{\wedge}_I \simeq \lim_n \widetilde{C}/I^n\widetilde{C} \simeq \lim_n \widetilde{C}/(f_1^n, I^n\widetilde{C}) \simeq \lim_{n, m} \widetilde{C}/(f_1^n, I^m\widetilde{C}) \simeq \lim_{n,m} C_n/I^mC_n \simeq \lim_n (\lim_m C_n/I^mC_n) \simeq \lim_n B_n = B,
    \]
    we conclude that $\widetilde{C}^{\wedge}_I$ is rig-smooth over $(A, I)$. Therefore, Lemma~\ref{lemma:smoothify} ensures that $\widetilde{C}$ is rig-smooth over $(A, (f_1))$. Thus, we can apply the base of induction to find a finite type $A$-algebra $C$ such that it is smooth outside $\rm{V}(f_1A)$ and there is an isomorphism $C^{\wedge}_{(f_1)} \simeq \widetilde{C}$ of $A$-algebras. Since $C/(f_1^n) \simeq \widetilde{C}/(f_1^n)$ is smooth outside $\rm{V}(IA_n)$, \cite[\href{https://stacks.math.columbia.edu/tag/0523}{Tag 0523}]{stacks-project} and \cite[\href{https://stacks.math.columbia.edu/tag/00TF}{Tag 00TF}]{stacks-project} imply that $\Spec C \to \Spec A$ is smooth at all points of $\rm{V}(f_1C) \smallsetminus \rm{V}(IC)$. Since it is also assumed to be smooth outside $\rm{V}(f_1A)$, we conclude that $C$ is smooth outside $\rm{V}(IA)$. Finally, we note that
    \[
    C^{\wedge}_I \simeq \big(C^{\wedge}_{(f_1)}\big)^{\wedge}_I \simeq (\widetilde{C})^{\wedge}_I \simeq B.
    \]
    Thus, $C$ does the job. 
\end{proof}

\appendix

\section{Dimension function on affinoids}\label{appendix:berkovich}

In this appendix, we fix a non-archimedean field $K$ and a $K$-affinoid algebra $A$. The main goal of this section is to construct a canonical dimension function on $\Spec A$. The existence of this dimension function is used in the proof of Theorem~\ref{thm:main-theorem} to invoke results from \cite{deGabber}. 

\begin{defn} Let $X$ be a spectral space. A point $y\in X$ is a {\it specialization of} a point $x$ if $y\in \ov{\{x\}}$. We denote the specialization relation by the symbol $x\rightsquigarrow y$.  

A point $y\in X$ is an {\it immediate specialization of} a point $x$ if $x\rightsquigarrow y$, and there is no $z\in X\setminus \{x, y\}$ such that $x\rightsquigarrow z$ and $z \rightsquigarrow y$.
\end{defn}

\begin{defn}\label{defn:dimension-function}\cite[Exp.\,XIV, D\'ef.\,2.1.8, Prop.\,2.1.4]{deGabber} Let $X$ be a noetherian, universally catenary scheme. A {\it dimension function} on $X$ is a map
\[
\delta\colon \abs{X} \to \Z
\]
such that $\delta(y)=\delta(x)-1$ for any immediate specialization $x\rightsquigarrow y$.
\end{defn}

\begin{lemma}\label{lemma:maximal-points-detect-dimension} \cite[Lemma 2.1.5]{Conrad99} Let $A$ be a $K$-affinoid domain, and $\m\subset A$ a(ny) maximal ideal. Then $\dim A = \dim A_{\m}$. 
\end{lemma}

\begin{cor}\label{cor:dimension-function-affinoid} Let $A$ be a $K$-affinoid algebra. Then the function
\[
\delta_A \colon \abs{\Spec A}\to \Z
\]
defined by $\delta_A(x)=\dim A/\mathfrak{p}_x$ is a dimension function on $\Spec A$.
\end{cor}
We note that $A$ is a quotient of a regular algebra due to \cite[\textsection 2.2, Prop.\,17]{B}. Therefore, \cite[\href{https://stacks.math.columbia.edu/tag/00NQ}{Tag 00NQ} and \href{https://stacks.math.columbia.edu/tag/00NM}{Tag 00NM}]{stacks-project}  imply that $A$ is universally catenary. In particular, Definition~\ref{defn:dimension-function} applies to $\Spec A$.  
\begin{proof}
    The only thing we have to check is that, for any prime ideals $\mathfrak p' \subset \mathfrak p$ with no proper prime ideals between them, we have 
    \[
    \dim A/\mathfrak p =\dim A/\mathfrak p' -1.
    \]
    For this, we can assume that $\mathfrak p'=(0)$ (so $A$ is a domain) and so $\mathfrak p\subset A$ is an ideal of height-$1$. In this case, we pick a maximal ideal $\m\subset A$. Then Lemma~\ref{lemma:maximal-points-detect-dimension} reduces the question to showing that
    \[
    \dim A_{\m}/\mathfrak p =\dim A_{\m} - 1=\dim A_{\m} - \rm{ht}(\mathfrak p).
    \]
    This follows from \cite[\href{https://stacks.math.columbia.edu/tag/0ECF}{Tag 0ECF}]{stacks-project}. 
\end{proof}

\begin{cor}\label{cor:dimension-finite-type-over-affinoid} \cite[Exp.\,XIV, Cor.\,2.5.2]{deGabber} Let $A$ be a $K$-affinoid algebra, and let $f\colon X \to \Spec A$ be a finite type morphism. Then 
\[
\delta_X(x)=\delta_A(f(x)) + \rm{trdeg}\left(k(x)/k(f(x))\right)
\]
is a dimension function of $X$.
\end{cor}

\begin{rmk} In the notation of Corollary~\ref{cor:dimension-finite-type-over-affinoid}, we have $\delta_X(x)\leq \dim A + d'$, where  $d'$ is the relative dimension of $f$.
\end{rmk}

\section{Non-algebraizable affine formal schemes}\label{appendix:counterexamples}

In this appendix, we provide examples of non-algebraizable admissible affine formal schemes. We expect that these examples are known to the experts. However, we were not able to find any examples of non-algebraizable admissible formal schemes recorded in the existing literature. Therefore, we decided to include these examples in this paper. 

\subsection{Preliminary results}

In this subsection, we discuss some preliminary results that will be important to justify our examples of non-algebraizable affine admissible formal schemes. 

For the rest of this appendix, we fix a non-archimedean field $K$ with ring of integers $\O_K$ and residue field $k$. We denote by $\m\subset \O_K$ its maximal ideal and fix a choice of a pseudo-uniformizer $\varpi \in \m$. 

\begin{convention} For a ring $R$ and $d=0$, the $R$-algebra $R[T_1, \dots, T_d]$ is equal to $R$. Similarly, for a topologically finite type $\O_K$-algebra $R$ and $d=0$, the $R$-algebra $R\langle T_1, \dots, T_d\rangle$ is equal to $R$.
\end{convention}

Most results of this subsection are concerned about the relationship between $(\varpi)$-adic henselization and $(\varpi)$-adic completion of $\O_K$-algebras which are finite type over a topologically finite type $\O_K$-algbebra. Before we start this discussion, we prove the following basic lemma. 

\begin{lemma}\label{lemma:reduce-to-flat} Let $R_0$ be a topologically finite type $\O_K$-algebra, let $R$ be a finite type $R_0$-algebra, and set $J_R\coloneqq \{f\in R \ | \ \text{ there is } \pi \in \m\smallsetminus \{0\} \text{ such that } \pi f=0\} \subset R$. Then $R^\h_{(\varpi)}[\frac{1}{\varpi}] \to (R/J_R)^\h_{(\varpi)}[\frac{1}{\varpi}]$ and $R^\wedge_{(\varpi)}[\frac{1}{\varpi}] \to (R/J_R)^\wedge_{(\varpi)}[\frac{1}{\varpi}]$ are isomorphisms. 
\end{lemma}
\begin{proof}
    First, we note that flatness of the morphism $R \to R^\h_{(\varpi)}$ and \cite[\href{https://stacks.math.columbia.edu/tag/0DYE}{Tag 0DYE}]{stacks-project} imply that $(R/J_R)^\h_{(\varpi)} \simeq R^\h_{(\varpi)}/ (J_R \otimes_R R^\h_{(\varpi)})$. Therefore, the first part of the claim follows from the observation that $(J_R \otimes_R R^\h_{(\varpi)})[\frac{1}{\varpi}] \simeq J_R[\frac{1}{\varpi}] \otimes_R R^\h_{(\varpi)} \simeq  0$. 

    Now we notice that $R$ is $(\varpi)$-adically adhesive due to \cite[Th.\,7.3.2]{FGK} and \cite[Prop.~0.8.5.19]{FujKato}. Therefore, \cite[Prop.~4.3.4]{FGK} ensures that $(R/J_R)^\wedge_{(\varpi)} \simeq R^\wedge_{(\varpi)}/ (J_R \otimes_R R^\wedge_{(\varpi)})$. Therefore, the second part of the claim follows from the observation that $(J_R \otimes_R R^\wedge_{(\varpi)})[\frac{1}{\varpi}]\simeq J_R[\frac{1}{\varpi}] \otimes_R R^\wedge_{(\varpi)} \simeq  0$. 
\end{proof}

Now we discuss our first main tool in relating the properties of the $(\varpi)$-adic henselizations and $(\varpi)$-adic completions.

\begin{lemma}\label{lemma:completion-faithfully-flat-appendix} Let $R_0$ be a topologically finite type $\O_K$-algebra, and let $R$ be a finite type $R_0$-algebra. Then the natural map $R^\wedge_{(\varpi)} \to \big(R^\h_{(\varpi)}\big)^{\wedge}_{(\varpi)}$ is an isomorphism, the natural map $\varphi\colon R^{\h}_{(\varpi)} \to R^{\wedge}_{(\varpi)}$ is faithfully flat, and $R^{\h}_{(\varpi)}$ is $(\varpi)$-adically topologically universally adhesive (in the sense of \cite[Def.\,0.8.5.17]{FujKato}). 
\end{lemma}
\begin{proof}
    We first show that the natural map $R^{\wedge}_{(\varpi)} \to \big(R^\h_{(\varpi)}\big)^{\wedge}_{(\varpi)}$ is an isomorphism. For this, it suffices to show that the map $R/(\varpi^n) \to R^\h_{(\varpi)}/(\varpi^n) $ is an isomorphism for any $n\geq 1$. This follows from the combination of \cite[\href{https://stacks.math.columbia.edu/tag/0AGU}{Tag 0AGU}]{stacks-project}.  

    Now we show that the map $\varphi \colon R^{\h}_{(\varpi)} \to R^{\wedge}_{(\varpi)}$ is faithfully flat. First, \cite[\href{https://stacks.math.columbia.edu/tag/00HQ}{Tag 00HQ}]{stacks-project} ensures that it suffices to show that $\varphi$ is flat and the image of $\Spec \varphi \colon \Spec R^{\wedge}_{(\varpi)} \to \Spec R^{\h}_{(\varpi)}$ contains all closed points of $\Spec R^{\h}_{(\varpi)}$. Since $\varpi\in \rm{rad}(R^\h_{(\varpi)})$ and the natural map $R^\h_{(\varpi)}/(\varpi) \xr{\sim} R^{\wedge}_{(\varpi)}/(\varpi)$ is an isomorphism, we conclude that every closed point of $\Spec R^{\h}_{(\varpi)}$ lies in the image of $\Spec \varphi$. In other words, for the purpose of proving faithful flatness of $\varphi$, it is enough to show that it is flat. 
    
    For this, we write $R^{\h}_{(\varpi)} = \colim \,R_i$, where the colimit is taken over all \'etale maps $R \to R_i$ such that $R/(\varpi) \to R_i/(\varpi)$ is an isomorphism. For any such $R_i$, the natural map $R^\h_{(\varpi)} \to (R_i)^\h_{(\varpi)}$ is an isomorphism. Using the first paragraph of this proof, we conclude that $R^{\wedge}_{(\varpi)} \to (R_i)^{\wedge}_{(\varpi)}$ is an isomorphism as well. Thus, \cite[\href{https://stacks.math.columbia.edu/tag/05UU}{Tag 05UU}]{stacks-project} and the above observation reduce the question to showing that the natural map $R \to R^{\wedge}_{(\varpi)}$ is flat. Now \cite[Th.\,7.3.2]{FGK} and \cite[Prop.~0.8.5.19]{FujKato} imply that $R$ is $(\varpi)$-adically adhesive in the sense of \cite[Def.\,7.1.3]{FGK}. Therefore, the map $R \to R^{\wedge}_{(\varpi)}$ is flat by virtue of \cite[Prop.~4.3.4]{FGK}.  

    The only thing we are left to show is that $R^\h_{(\varpi)}$ is $(\varpi)$-adically topologically universally adhesive. For this, we need to show that $R^\h_{(\varpi)}$ is $(\varpi)$-adically universally adhesive and $\big(R^\h_{(\varpi)}\big)^{\wedge}_{(\varpi)}\langle T_1, \dots, T_n\rangle \simeq R^{\wedge}_{(\varpi)}\langle T_1, \dots, T_n\rangle$ is $(\varpi)$-adically universally adhesive for any integer $n\geq 0$. The latter condition follows automatically from \cite[Th.\,7.3.2]{FGK} and \cite[Prop.~0.8.5.19]{FujKato} because $R^{\wedge}_{(\varpi)}\langle T_1, \dots, T_n\rangle$ is topologically finite type over $\O_K$. Therefore, it only suffices to show that $R^\h_{(\varpi)}$ is $(\varpi)$-adically universally adhesive. Using \cite[Prop.~0.8.5.19 and Prop.~0.8.5.10]{FujKato}, we see that it is enough to show that $R^\h_{(\varpi)}[\frac{1}{\varpi}]$ is noetherian. Now \cite[\href{https://stacks.math.columbia.edu/tag/033E}{Tag 033E}]{stacks-project} and the faithful flatness of $R^\h_{(\varpi)} \to R^{\wedge}_{(\varpi)}$ imply that it suffices to show that $R^{\wedge}_{(\varpi)}[\frac{1}{\varpi}]$ is noetherian. This follows from \cite[Prop.~3.1/3(i)]{B}.   
\end{proof}

Now we collect some easy consequences of Lemma~\ref{lemma:completion-faithfully-flat-appendix}. 

\begin{cor}\label{cor:preimage-dense} Let $R_0$ be a topologically finite type $\O_K$-algebra, let $R$ be a finite type $R_0$-algebra, and let $\pi \colon \Spec \big(R^{\wedge}_{(\varpi)}[\frac{1}{\varpi}]\big) \to \Spec \big(R^{\h}_{(\varpi)}[\frac{1}{\varpi}]\big)$ be the natural morphism, and let $U\subset \Spec \big(R^{\h}_{(\varpi)}[\frac{1}{\varpi}]\big)$ be a dense open subset. Then $\pi^{-1}(U) \subset \Spec \big(R^{\wedge}_{(\varpi)}[\frac{1}{\varpi}]\big)$ is also a dense open subset.
\end{cor}
\begin{proof}
    Lemma~\ref{lemma:completion-faithfully-flat-appendix} implies that $R^\h_{(\varpi)}[\frac{1}{\varpi}]$ is noetherian. Therefore, $\Spec R^\h_{(\varpi)}[\frac{1}{\varpi}]$ has finitely many irreducible $Z_1, \dots, Z_n$. Let us denote by $\eta_i\in Z_i$ the corresponding generic point. Since $U\subset \Spec \big(R^{\h}_{(\varpi)}[\frac{1}{\varpi}]\big)$ is dense, we conclude that $\eta_i\in U$ for each $i=1, \dots, n$.   

    Now we need to show that $V\cap \pi^{-1}(U)\neq \varnothing$ for any non-empty open subset $V \subset \Spec \big(R^{\wedge}_{(\varpi)}[\frac{1}{\varpi}]\big)$. Pick any $v\in V$. Then there is  an index $i\in\{1, \dots, n\}$ such that $\eta_i$ generizes to $\pi(v)$.  Therefore, Lemma~\ref{lemma:completion-faithfully-flat-appendix} and the going-down lemma (see \cite[\href{https://stacks.math.columbia.edu/tag/00HS}{Tag 00HS}]{stacks-project}) imply that there is a point $w\in \Spec \big(R^{\wedge}_{(\varpi)}[\frac{1}{\varpi}]\big)$ such that $w$ generizes to $v$ and $\pi(w)=\eta_i$. Since $V$ is closed under generizations and $\eta_i\in U$, we see that $w\in V\cap \pi^{-1}(U)$. 
\end{proof}

\begin{cor}\label{cor:approximate} Let $R_0$ be a topologically finite type $\O_K$-algebra and let $R$ be a finite type $R_0$-algebra, and let $R^\h_{(\varpi)} \to \widetilde{S}$ be a finite morphism. Then there is a finite type morphism $R \to R'$, which induces isomorphisms $R^\h_{(\varpi)}[\frac{1}{\varpi}] \xr{\sim} R'^\h_{(\varpi)}[\frac{1}{\varpi}]$ and $R^\wedge_{(\varpi)}[\frac{1}{\varpi}] \xr{\sim} R'^\wedge_{(\varpi)}[\frac{1}{\varpi}]$, and a finite, finitely presented morphism $R' \to S$ such that $S^\h_{(\varpi)}[\frac{1}{\varpi}] \simeq \widetilde{S}$ as an $R^\h_{(\varpi)}[\frac{1}{\varpi}] \simeq R'^\h_{(\varpi)}[\frac{1}{\varpi}]$-algebra. 
\end{cor}
\begin{proof}
    Lemma~\ref{lemma:reduce-to-flat} allows us to reduce to the case when $R$ is $\O_K$-flat. In particular, $R^\h_{(\varpi)}$ is also $\O_K$-flat and, thus, the natural morphism $R^\h_{(\varpi)} \to R^\h_{(\varpi)}[\frac{1}{\varpi}]$ is injective. Let $x_1, \dots, x_n\in \widetilde{S}$ be $R^\h_{(\varpi)}[\frac{1}{\varpi}]$-module generators of $\widetilde{S}$. After replacing each $x_i$ with $\varpi^N\cdot x_i$ for $N\gg 0$, we can assume that the $R^\h_{(\varpi)}$-module submodule of $\widetilde{S}$ generated by $x_i$ is closed under multiplication. In particular, it defines an $\O_K$-flat finite $R^\h_{(\varpi)}$-algebra $\widetilde{S}^+$ such that $\widetilde{S}^+[\frac{1}{\varpi}] \simeq \widetilde{S}$ as $R^\h_{(\varpi)}[\frac{1}{\varpi}]$-algebras.  

    Choose any surjection $f\colon R^\h_{(\varpi)}[X_1, \dots, X_n] \twoheadrightarrow \widetilde{S}^+$ and set $M \coloneqq \ker(f)$. Note that $M$ is $(\varpi)$-saturated (because $\widetilde{S}^+$ is $\O_K$-flat) and $R^\h_{(\varpi)}$ is $(\varpi)$-adically universally adhesive (see Lemma~\ref{lemma:completion-faithfully-flat-appendix}), so \cite[Prop.~0.8.5.3(c)]{FujKato} implies that $M$ is a finite $R^\h_{(\varpi)}[X_1, \dots, X_n]$-module. In other words, $\widetilde{S}^+$ is a finitely presented $R^\h_{(\varpi)}$-algebra. Thus, a standard approximation argument implies that we can find an \'etale map $R \to R'$ which is an isomorphism modulo $\varpi$ (in particular, it induces equivalences $R^\h_{(\varpi)} \xr{\sim} R'^\h_{(\varpi)}$ and $R^\wedge_{(\varpi)} \xr{\sim} R'^\wedge_{(\varpi)}$) and a finite, finitely presented $R' \to S$ such that 
    $S\otimes_{R'} R^\h_{(\varpi)} \simeq \widetilde{S}^+$. Then \cite[\href{https://stacks.math.columbia.edu/tag/0DYE}{Tag 0DYE}]{stacks-project} implies that $S^\h_{(\varpi)} \simeq S\otimes_{R'} R^\h_{(\varpi)} \simeq \widetilde{S}^+$. In particular, we have an natural isomorphism $\widetilde{S} \simeq S^\h_{(\varpi)}[\frac{1}{\varpi}]$ of $R^\h_{(\varpi)}[\frac{1}{\varpi}]$-algebras. 
\end{proof}

Now we are ready to establish our second main tool of this subsection. The following lemma is a much stronger version of Lemma~\ref{lemma:completion-faithfully-flat-appendix}. 

\begin{lemma}\label{lemma:completion-regular-generic-fiber} Let $R_0$ be a topologically finite type $\O_K$-algebra, and let $R$ be a finite type $R_0$-algebra. Then the natural maps $R \to R^{\wedge}_{(\varpi)}$ and $R^{\h}_{(\varpi)} \to R^{\wedge}_{(\varpi)}$ are regular. 
\end{lemma}
\begin{proof}
    First, Lemma~\ref{lemma:completion-faithfully-flat-appendix} ensures that $R \to R^{\wedge}_{(\varpi)}$ and $R^{\h}_{(\varpi)} \to R^{\wedge}_{(\varpi)}$ are flat. Therefore, it suffices to show that these morphisms have geometrically regular fibers. Since $R/(\varpi) \to R^{\wedge}_{(\varpi)}/(\varpi)$ and $R^\h_{(\varpi)}/(\varpi) \to R^{\wedge}_{(\varpi)}/(\varpi)$ are isomorphisms (see \cite[\href{https://stacks.math.columbia.edu/tag/0AGU}{Tag 0AGU}]{stacks-project} and \cite[\href{https://stacks.math.columbia.edu/tag/05GG}{Tag 05GG}]{stacks-project}), it suffices to show that $R[\frac{1}{\varpi}] \to R^{\wedge}_{(\varpi)}[\frac{1}{\varpi}]$ and $R^\h_{(\varpi)}[\frac{1}{\varpi}] \to R^{\wedge}_{(\varpi)}[\frac{1}{\varpi}]$ have geometrically regular fibers.  

    {\it Step~$1$. We reduce the question to showing that the morphism $R[\frac{1}{\varpi}] \to R^{\wedge}_{(\varpi)}[\frac{1}{\varpi}]$ is regular.} The morphism $R[\frac{1}{\varpi}] \to R^{\h}_{(\varpi)}[\frac{1}{\varpi}]$ is ind-\'etale by its very construction. Furthermore, Lemma~\ref{lemma:completion-faithfully-flat-appendix} implies that $R^{\h}_{(\varpi)}[\frac{1}{\varpi}]$ is noetherian. In particular, for any prime ideal $\p\subset R[\frac{1}{\varpi}]$ with residue field $k(\p)$, the $k(\p)$-algebra $R^\h(\p)\coloneqq R^\h_{(\varpi)}[\frac{1}{\varpi}] \otimes_{R[\frac{1}{\varpi}]} k(\p)$ is ind-\'etale and noetherian. Now \cite[\href{https://stacks.math.columbia.edu/tag/092E}{Tag 092E}]{stacks-project} implies that $R^\h(\p)$ has weak dimension $\leq 0$. Therefore, \cite[\href{https://stacks.math.columbia.edu/tag/092F}{Tag 092F}]{stacks-project} implies that $R^\h(\p)$ is reduced and is of Krull dimension $0$. Since $R^\h(\p)$ is noetherian, it has finitely many minimal prime ideals. Therefore, we conclude that $R^\h(\p)$ is a finite product of fields. More precisely, $R^\h(\p) = \prod_{i=1}^m k(\q_i)$ where $\{\q_1, \dots, \q_m\}$ is the set of all prime ideals of $R^\h_{(\varpi)}$ lying over $\p$ and $k(\q_i)$ are the associated residue fields. Since $k(\p) \to R^\h(\p)$ is ind-\'etale, we conclude that each extension $k(\p) \subset k(\q_i)$ is an algebraic separable extension. Therefore, \cite[\href{https://stacks.math.columbia.edu/tag/07QH}{Tag 07QH}]{stacks-project} implies that $R^{\wedge}_{(\varpi)}[\frac{1}{\varpi}] \otimes_{R[\frac{1}{\varpi}]} k(\p)$ is geometrically regular if and only if $R^{\wedge}_{(\varpi)}[\frac{1}{\varpi}] \otimes_{R^{\h}_{(\varpi)}[\frac{1}{\varpi}]} k(\q_i)$ is geometrically regular for each $i=1, \dots, m$. In other words, it suffices to show that $R[\frac{1}{\varpi}] \to R^{\wedge}_{(\varpi)}[\frac{1}{\varpi}]$ is regular.  

    {\it Step~$2$. We reduce to the case $R=\O_K\langle X_1, \dots, X_d\rangle[T_1, \dots, T_n]$ for integers $d, n\geq 0$.} Our assumption on $R$ implies that we can find a surjection $\O_K\langle X_1, \dots, X_d\rangle[T_1, \dots, T_n] \twoheadrightarrow R$. Since $\O_K\langle X_1, \dots, X_d\rangle[T_1, \dots, T_n]$ is $(\varpi)$-adically adhesive, we see that \cite[Prop.~4.3.4]{FGK} guarantees that 
    \[
    R^{\wedge}_{(\varpi)} \simeq R \otimes_{\O_K\langle X_1, \dots, X_d\rangle[T_1, \dots, T_n]} \O_K\langle X_1, \dots, X_d, T_1, \dots, T_n\rangle.
    \]
    Since regular morphisms are preserved under finite type base change (see \cite[\href{https://stacks.math.columbia.edu/tag/07C1}{Tag 07C1}]{stacks-project}), we conclude that it suffices to prove the claim under the additional assumption that $R=\O_K\langle X_1, \dots, X_d\rangle[T_1, \dots, T_n]$.  

    {\it Step~$3$. We show that $R[\frac{1}{\varpi}] \to R^{\wedge}_{(\varpi)}[\frac{1}{\varpi}]$ is regular for $R=\O_K\langle X_1, \dots, X_d\rangle[T_1, \dots, T_n]$.} In this case, the question boils down to showing that $K\langle X_1, \dots, X_d\rangle[T_1, \dots, T_n] \to K\langle X_1, \dots, X_d, T_1, \dots, T_n\rangle$ is regular. This follows from \cite[Lem.~7.1.4]{LRZ24} applied to $A = K \langle X_1, \dots, X_d\rangle$, $B = K\langle X_1, \dots, X_d\rangle[T_1, \dots, T_n]$, and $R = K\langle X_1, \dots, X_d, T_1, \dots, T_n\rangle$.    
\end{proof}

\begin{rmk} The rings $R[\frac{1}{\varpi}]$, $R^\h_{(\varpi)}[\frac{1}{\varpi}]$, and $R^{\wedge}_{(\varpi)}[\frac{1}{\varpi}]$ are noetherian, so a combination of Lemma~\ref{lemma:completion-regular-generic-fiber}, Popescu's smoothening theorem (see \cite[Th.~1.8]{Popescu} or \cite[\href{https://stacks.math.columbia.edu/tag/07GC}{Tag 07GC}]{stacks-project}), and \cite[Th.~1.2]{Longke} imply that $R \to R^{\wedge}_{(\varpi)}$ and $R^\h_{(\varpi)} \to R^{\wedge}_{(\varpi)}$ are ind-smooth maps. 
\end{rmk}

Now we are ready to collect some consequences of Lemma~\ref{lemma:completion-regular-generic-fiber}. In the rest of this appendix, we denote by $\Min(R)$ the set of minimal prime ideals of a ring $R$. 

\begin{cor}\label{cor:properties-after-completion} Let $R_0$ be a topologically finite type $\O_K$-algebra, and let $R$ be a finite type $R_0$-algebra.
\begin{enumerate}
    \item\label{cor:properties-after-completion-1} The natural map $\pi_0\big(\Spec R^{\wedge}_{(\varpi)}[\frac{1}{\varpi}] \big) \to \pi_0\big(\Spec R^{\h}_{(\varpi)}[\frac{1}{\varpi}] \big)$ is a bijection;
    \item\label{cor:properties-after-completion-3} The ring $R^\h_{(\varpi)}[\frac{1}{\varpi}]$ is reduced (resp.~normal, resp.~a domain) if and only if $R^{\wedge}_{(\varpi)}[\frac{1}{\varpi}]$ is reduced (resp.~normal, resp.~a domain);
    \item\label{cor:properties-after-completion-4} For any prime ideal $\p\subset R^\h_{(\varpi)}[\frac{1}{\varpi}]$, the ideal $\p \big(R^\wedge_{\varpi}[\frac{1}{\varpi}] \big)$ is also a prime ideal;
    \item\label{cor:properties-after-completion-5} The map $\Spec \big(R^\wedge_{(\varpi)}[\frac{1}{\varpi}]\big) \to \Spec \big(R^\h_{(\varpi)}[\frac{1}{\varpi}]\big)$ restricts to a bijection  $\Min\big(R^\wedge_{(\varpi)}[\frac{1}{\varpi}]\big) \xr{\sim} \Min\big(R^\h_{(\varpi)}[\frac{1}{\varpi}]\big)$. Its inverse is given by the map $\Min\big(R^\h_{(\varpi)}[\frac{1}{\varpi}]\big) \ni \p \mapsto \p \big(R^\wedge_{(\varpi)}[\frac{1}{\varpi}]\big) \in \Min\big(R^\wedge_{(\varpi)}[\frac{1}{\varpi}]\big)$.  
\end{enumerate}
\end{cor}
\begin{proof}
    (\ref{cor:properties-after-completion-1}): Let $\Idem \colon \Rings \to \Sets$ be the functor that sends a ring $A$ to the set of idempotents in $A$. It suffices to show that the map $\Idem\big(R^{\h}_{(\varpi)}[\frac{1}{\varpi}]\big) \to \Idem\big(R^{\wedge}_{(\varpi)}[\frac{1}{\varpi}]\big)$ is a bijection. Since the functor $\Idem$ commutes with filtered colimits and $\Idem(B) \xr{\sim} \Idem(B/I)$ is a bijection for any henselian pair $(B,I)$ (see \cite[\href{https://stacks.math.columbia.edu/tag/09XI}{Tag 09XI}]{stacks-project}),  the result follows from \cite[Th.~2.1.15(b)]{Bouthier-Cesnavicius}.  


    (\ref{cor:properties-after-completion-3}): First, we note that $R^\h_{(\varpi)}[\frac{1}{\varpi}]$ and $R^{\wedge}_{(\varpi)}[\frac{1}{\varpi}]$ are noetherian due to Lemma~\ref{lemma:completion-regular-generic-fiber} and \cite[Prop.~3.1/3(i)]{B}. Therefore, Lemma~\ref{lemma:completion-faithfully-flat-appendix} and Lemma~\ref{lemma:completion-regular-generic-fiber} guarantee that $R^\h_{(\varpi)}[\frac{1}{\varpi}] \to R^{\wedge}_{(\varpi)}[\frac{1}{\varpi}]$ is a faithfully flat regular morphism of noetherian rings. Thus, \cite[\href{https://stacks.math.columbia.edu/tag/033F}{Tag 033F}]{stacks-project} and \cite[\href{https://stacks.math.columbia.edu/tag/07QK}{Tag 07QK}]{stacks-project} (resp. \cite[\href{https://stacks.math.columbia.edu/tag/033G}{Tag 033G}]{stacks-project} and \cite[\href{https://stacks.math.columbia.edu/tag/0BFK}{Tag 0BFK}]{stacks-project}) imply that $R^{\h}_{(\varpi)}[\frac{1}{\varpi}]$ is reduced (resp. normal) if and only if $R^\wedge_{(\varpi)}[\frac{1}{\varpi}]$ is reduced (resp.~normal).  

    Since the morphism $R^\h_{(\varpi)}[\frac{1}{\varpi}] \to R^{\wedge}_{(\varpi)}[\frac{1}{\varpi}]$ is faithfully flat (in particular, injective), we conclude that $R^{\h}_{(\varpi)}[\frac{1}{\varpi}]$ is a domain if $R^{\wedge}_{(\varpi)}[\frac{1}{\varpi}]$ is a domain. Therefore, we are only left to show that $R^{\wedge}_{(\varpi)}[\frac{1}{\varpi}]$ is a domain if so is $R^\h_{(\varpi)}[\frac{1}{\varpi}]$. Furthermore, we already know that $R^{\wedge}_{(\varpi)}[\frac{1}{\varpi}]$ is reduced in this case, so it suffices to show that $\Spec \big(R^{\wedge}_{(\varpi)}[\frac{1}{\varpi}]\big)$ is irreducible when $R^\h_{(\varpi)}[\frac{1}{\varpi}]$ is a domain.   
    
    Now \cite[Th.~5.3(iv)]{Greco} implies that $R^\h_{(\varpi)}[\frac{1}{\varpi}]$ is excellent. Then \cite[\href{https://stacks.math.columbia.edu/tag/07QV}{Tag 07QV}]{stacks-project} ensures that the normalization $R^\h_{(\varpi)}[\frac{1}{\varpi}] \hookrightarrow \widetilde{S}$ is finite. Therefore, Corollary~\ref{cor:approximate}  allows us to reduce to the situation where there is a finite, finitely presented morphism $R \to S$ such that $S^{\h}_{(\varpi)}[\frac{1}{\varpi}] \simeq \widetilde{S}$. Hence,  (\ref{cor:properties-after-completion-1}) and the case of normal rings imply that $S^{\wedge}_{(\varpi)}[\frac{1}{\varpi}]$ is a normal domain. Therefore, in order to show that $\Spec \big( R^{\wedge}_{(\varpi)}[\frac{1}{\varpi}]\big)$ is irreducible, it suffices to show that the natural morphism $\Spec \big(S^{\wedge}_{(\varpi)}[\frac{1}{\varpi}]\big) \to \Spec \big( R^{\wedge}_{(\varpi)}[\frac{1}{\varpi}]\big)$ is an isomorphism over some dense open subset $U\subset \Spec R^{\wedge}_{(\varpi)}[\frac{1}{\varpi}]$. We see that \cite[Prop.~4.3.4]{FGK} implies that $S^{\wedge}_{(\varpi)}[\frac{1}{\varpi}] \simeq S\otimes_{R} R^{\wedge}_{(\varpi)}[\frac{1}{\varpi}] \simeq \widetilde{S}\otimes_{R^{\h}_{(\varpi)}[\frac{1}{\varpi}]} R^{\wedge}_{(\varpi)}[\frac{1}{\varpi}]$. Therefore, the result follows from Corollary~\ref{cor:preimage-dense} and the fact that the normalization $\Spec \widetilde{S} \to \Spec \big(R^\h_{(\varpi)}[\frac{1}{\varpi}]\big)$ is an isomorphism over a dense open subset.  

    (\ref{cor:properties-after-completion-4}): It follows immediately from Corollary~\ref{cor:approximate} and (\ref{cor:properties-after-completion-3}).  

    (\ref{cor:properties-after-completion-5}): Since $R^{\h}_{(\varpi)}[\frac{1}{\varpi}] \to R^{\wedge}_{(\varpi)}[\frac{1}{\varpi}]$ is faithfully flat, \cite[\href{https://stacks.math.columbia.edu/tag/00FK}{Tag 00FK}]{stacks-project} and \cite[\href{https://stacks.math.columbia.edu/tag/00HS}{Tag 00HS}]{stacks-project} imply that the map $\Spec \big(R^\wedge_{(\varpi)}[\frac{1}{\varpi}]\big) \to \Spec \big(R^\h_{(\varpi)}[\frac{1}{\varpi}]\big)$ restricts to a surjection  $\Min\big(R^\wedge_{(\varpi)}[\frac{1}{\varpi}]\big) \twoheadrightarrow \Min\big(R^\h_{(\varpi)}[\frac{1}{\varpi}]\big)$.  

    Now we show that this map is also injective. Suppose we have $\mathfrak{P}_1, \mathfrak{P}_2\in \Min\big(R^\wedge_{(\varpi)}[\frac{1}{\varpi}]\big)$ such that 
    \[
    \mathfrak{P}_1 \cap R^\h_{(\varpi)}\big[\frac{1}{\varpi}\big] =  \mathfrak{P}_2 \cap R^\h_{(\varpi)}\big[\frac{1}{\varpi}\big]  \eqqcolon \p.
    \]
    Then (\ref{cor:properties-after-completion-4}) implies that $\mathfrak{P}\coloneqq \p \cdot (R^{\wedge}_{(\varpi)}[\frac{1}{\varpi}])$ is a prime ideal which is contained in both $\mathfrak{P}_1$ and $\mathfrak{P}_2$. Since both $\mathfrak{P}_1$ and $\mathfrak{P}_2$ are minimal, we conclude that $\mathfrak{P}_1=\mathfrak{P}=\mathfrak{P}_2$. This proves that $\Min\big(R^\wedge_{(\varpi)}[\frac{1}{\varpi}]\big) \to \Min\big(R^\h_{(\varpi)}[\frac{1}{\varpi}]\big)$ is bijective.  

    To understand the inverse to this map, it suffices to show that, for every minimal prime ideal $\p\subset R^\h_{(\varpi)}[\frac{1}{\varpi}]$, the ideal $\mathfrak{P}\coloneqq \p \cdot (R^{\wedge}_{(\varpi)}[\frac{1}{\varpi}])$ is a minimal prime ideal of $R^{\wedge}_{(\varpi)}[\frac{1}{\varpi}]$ and that $\mathfrak{P} \cap R^{\h}_{(\varpi)}[\frac{1}{\varpi}] =\p$. The latter follows immediately from \cite[Th.~7.5]{Matsumura} and faithful flatness of $R^{\h}_{(\varpi)}[\frac{1}{\varpi}] \to R^{\wedge}_{(\varpi)}[\frac{1}{\varpi}]$. Therefore, we only need to show that $\mathfrak{P}$ is a minimal prime ideal. It is a prime ideal by virtue of (\ref{cor:properties-after-completion-4}). So we only need to show that it is minimal among prime ideals of $R^{\wedge}_{(\varpi)}[\frac{1}{\varpi}]$.  
    
    If it is not a minimal prime ideal, then there is a prime ideal $\mathfrak{P}'\nsubseteq \mathfrak{P}$. Then $\mathfrak{P}' \cap R^{\h}_{(\varpi)}[\frac{1}{\varpi}] \subset \mathfrak{P} \cap R^{\h}_{(\varpi)}[\frac{1}{\varpi}] = \p$. Since $\p$ is a minimal prime ideal, we conclude that $\mathfrak{P'} \cap R^{\h}_{(\varpi)}[\frac{1}{\varpi}] = \p$. Hence, $\mathfrak{P} \subset \p \cdot R^{\wedge}_{(\varpi)}[\frac{1}{\varpi}] \subset \mathfrak{P'}$. This implies $\mathfrak{P} = \mathfrak{P'}$ leading to the contradiction. 
\end{proof}

\begin{cor}\label{cor:base-change-nilradical} Let $R_0$ be a topologically finite type $\O_K$-algebra, and let $R$ be a finite type $R_0$-algebra. Assume, in addition, that $R$ is flat over $\O_K$. Then the natural map $\rm{nil}(R) \otimes_{R} R^{\wedge}_{(\varpi)} \to \rm{nil}(R^{\wedge}_{(\varpi)})$ is an isomorphism. 
\end{cor}
\begin{proof}
    Lemma~\ref{lemma:completion-faithfully-flat-appendix} implies that the map $R \to R^{\wedge}_{(\varpi)}$ is flat. This ensures that the natural map $\rm{nil}(R) \otimes_{R} R^{\wedge}_{(\varpi)} \to \rm{nil}\big(R^{\wedge}_{(\varpi)}\big)$ is injective. Thus, in order to show that it is an isomorphism, it suffices to show that 
    \[
    R^{\wedge}/(\rm{nil}(R) \otimes_R R^{\wedge}_{(\varpi)}) \simeq R/\rm{nil}(R) \otimes_R R^{\wedge}_{(\varpi)}) \simeq \big(R/\rm{nil}(R) \big)^{\wedge}_{(\varpi)}
    \]
    is reduced. For brevity, we denote by $S\coloneqq R/\rm{nil}(R)$ and wish to show that $S^{\wedge}_{(\varpi)}$ is reduced. By construction, $S$ is still $\O_K$-flat and finite type over $R_0$. Therefore, we conclude that $S^{\wedge}_{(\varpi)}$ is $\O_K$-flat. Thus, it suffices to show that $S^{\wedge}_{(\varpi)}[\frac{1}{\varpi}]$ is reduced.  Corollary~\ref{cor:properties-after-completion}(\ref{cor:properties-after-completion-3}) guarantees that it is enough to show that $S^\h_{(\varpi)}[\frac{1}{\varpi}]$ is reduced. This follows directly from basic properties of \'etale morphisms and the fact that $S$ is reduced.  
\end{proof}

Our next goal is to relate $\dim R^\h_{(\varpi)}[\frac{1}{\varpi}]$ to $\dim R^{\wedge}_{(\varpi)}[\frac{1}{\varpi}]$. Our main tool to relate these two dimensions will be a henselian version of Noether normalization. For simplicity, we will restrict our attention to the case when $R$ is a finite type $\O_K$-algebra which is sufficient for all our applications. 

\begin{notation} We set $\O_K\{T_1, \dots, T_d\} \coloneqq \O_K[T_1, \dots, T_d]^{\h}_{(\varpi)}$ and $K\{T_1, \dots, T_d\} \coloneqq \O_K\{T_1, \dots, T_d\}[\frac{1}{\varpi}]$.
\end{notation}

\begin{lemma}[Noether normalization]\label{lemma:henselian-noether-normalization} Let $R$ be a finite type flat $\O_K$-algebra. Then there is an integer $d$ and a finite injective morphism $h\colon \O_K\{T_1, \dots, T_d\} \hookrightarrow R^{\h}_{(\varpi)}$ such that $h \text{ \emph{mod} }\m\colon k[T_1, \dots, T_d] \hookrightarrow R/\m R$ is injective as well.
\end{lemma}
\begin{proof}
    The usual Noether normalization lemma \cite[\href{https://stacks.math.columbia.edu/tag/07NA}{Tag 07NA}]{stacks-project} implies that we can find an finite injective morphism $\ov{h}\colon k[T_1, \dots, T_d] \hookrightarrow R/\m R$. Since $R$ is of finite type over $\O_K$, we can lift $h$ to a finite type morphism $h'\colon \O_K[T_1, \dots, T_d] \to R$.  
    
    Now we base change $h'$ along the map $\O_K[T_1, \dots, T_d] \to \O_K\{T_1, \dots, T_d\}$ to get a finite type morphism 
    \[
    h''\colon \O_K\{T_1, \dots, T_d\} \to R \otimes_{\O_K[T_1, \dots, T_d]} \O_K\{T_1, \dots, T_d\}
    \]
    such that $h'' \text{ mod }\m = h \text{ mod }\m$ is finite. Note that \cite[\href{https://stacks.math.columbia.edu/tag/09XJ}{Tag 09XJ}]{stacks-project} implies that $\O_K\{T_1, \dots, T_d\}$ is henselian with respect to the ideal $\m \cdot \O_K\{T_1, \dots, T_d\}$. Therefore, \cite[Prop.~2.3.2]{Moret-Bailly} implies that there is a clopen subscheme $\Spec R' \subset \Spec \big(R \otimes_{\O_K[T_1, \dots, T_d]} \O_K\{T_1, \dots, T_d\}\big)$ such that $R'$ is a finite $\O_K\{T_1, \dots, T_d\}$-algebra and  $\Spec R'$ contains the special fiber of $\Spec \big(R \otimes_{\O_K[T_1, \dots, T_d]} \O_K\{T_1, \dots, T_d\}\big)$. In particular, the natural morphism $R'^\h_{(\varpi)} \to \big(R \otimes_{\O_K[T_1, \dots, T_d]} \O_K\{T_1, \dots, T_d\}\big)^{\h}_{(\varpi)}$ is an isomorphism.  

    The universal property of henselizations (see \cite[\href{https://stacks.math.columbia.edu/tag/0A02}{Tag 0A02}]{stacks-project}) and \cite[\href{https://stacks.math.columbia.edu/tag/09XK}{Tag 09XK}]{stacks-project} imply that 
    \[
    R^\h_{(\varpi)} \simeq \big(R \otimes_{\O_K[T_1, \dots, T_d]} \O_K\{T_1, \dots, T_d\}\big)^{\h}_{(\varpi)} \simeq R'^\h_{(\varpi)} \simeq R'.
    \]
    Thus, we conclude that the natural morphism $h\coloneqq (h')^\h_{(\varpi)} \colon \O_K\{T_1, \dots, T_d\} \to R^\h_{(\varpi)}$ is finite. By construction, the induced morphism $h \text{ mod }\m = \ov{h} \colon k[T_1, \dots, T_d] \to R/\m R$ is injective. Therefore, we are only left to show that $h$ is injective as well.   
    
    To verify injectivity of $h$, we set $M\coloneqq \rm{Im}(h)\subset R$ and $N\coloneqq \rm{Ker}(h)$. Since $R$ is flat over $\O_K$, we conclude that $M \subset R$ is flat over $\O_K$ as well. Thus, the assumption that $h\text{ mod }\m$ is injective implies that $N/\m N \simeq 0$. Since $\O_K\{T_1, \dots, T_d\}$ is $(\varpi)$-adically adhesive (see Lemma~\ref{lemma:completion-faithfully-flat-appendix}) and $N$ is $(\varpi)$-saturated, e.g., $\O_K\{T_1, \dots, T_d\}/N \simeq M$ is $\O_K$-flat, \cite[Prop.~0.8.5.3(c)]{FujKato} implies that $N$ is a finitely generated $\O_K\{T_1, \dots, T_d\}$-module. Therefore, Nakayama's lemma (see \cite[\href{https://stacks.math.columbia.edu/tag/00DV}{Tag 00DV}]{stacks-project}) ensures that $N/\m N=0$ implies that $N=0$. In other words, the morphism $h$ is indeed injective. 
\end{proof}

Now we relate Noether normalization from Lemma~\ref{lemma:henselian-noether-normalization} to the usual Noether normalization for flat, topologically finite type $\O_K$-algebras. 

\begin{lemma}\label{lemma:henselian-noether-normalization-2} Let $R$ be a finite type flat $\O_K$-algebra and let $h\colon \O_K\{T_1, \dots, T_d\} \to R^{\h}_{(\varpi)}$ be an injective finite morphism such that $h \text{ \emph{mod} }\m\colon k[T_1, \dots, T_d] \to R/\m R$ is still injective. Then its completion $h^{\wedge}_{(\varpi)} \colon \O_K\langle T_1, \dots, T_d\rangle \to R^{\wedge}_{(\varpi)}$ is a finite injective morphism such that $h^{\wedge}_{(\varpi)} \text{ \emph{mod} }\m\colon k[T_1, \dots, T_d] \to R/\m R$ is still injective.
\end{lemma}
\begin{proof}
    First, we note that $\O_K\{T_1, \dots, T_d\}$ is $(\varpi)$-adically adhesive due to Lemma~\ref{lemma:completion-faithfully-flat-appendix}. Therefore, \cite[Prop.~4.3.4]{FGK} implies that $R^{\wedge}_{(\varpi)} \simeq R^{\h}_{(\varpi)} \otimes_{\O_K\{T_1, \dots, T_d\}} \O_K\langle T_1, \dots, T_d\rangle$ and $h^{\wedge}_{(\varpi)} = h \otimes_{\O_K\{T_1, \dots, T_d\}} \O_K\langle T_1, \dots, T_d\rangle$. In particular, $h^{\wedge}_{(\varpi)}$ is finite and $h^{\wedge}_{(\varpi)} \text{ mod }\m = h \text{ mod }\m$ is injective. Therefore, we are only left to show that $h^{\wedge}_{(\varpi)}$ is injective. For this, we could repeat the argument in the last paragraph of the proof of Lemma~\ref{lemma:henselian-noether-normalization}. The only difference is that we need to use \cite[Th.\,7.3.2]{FGK} and \cite[Prop.~0.8.5.19]{FujKato} to justify that $\O_K\langle T_1, \dots, T_d\rangle$ is $(\varpi)$-adically adhesive. 
\end{proof}

Now we use Lemma~\ref{lemma:henselian-noether-normalization} to understand the dimension of $K\{T_1, \dots, T_d\}$ and its maximal ideals. We later combine it with Lemma~\ref{lemma:henselian-noether-normalization-2} to show that $\dim R^\h_{(\varpi)}[\frac{1}{\varpi}] = \dim R^{\wedge}_{(\varpi)}[\frac{1}{\varpi}]$ for any finite type $\O_K$-algebra. 

\begin{cor}\label{cor:maximal-ideal-henselian} Let $d\geq 0$ be an integer and let $\p \subset K\{T_1, \dots, T_d\}$ be a prime ideal. Then $\p$ is maximal if and only if $K \to K\{T_1, \dots, T_d\}/\p$ is a finite morphism. 
\end{cor}
\begin{proof}
    Corollary~\ref{cor:approximate} and Lemma~\ref{lemma:henselian-noether-normalization} imply that we can find a finite injective morphism $K\{T_1, \dots, T_n\} \hookrightarrow K\{T_1, \dots, T_d\}/\p$ for some integer $n\geq 0$. First, we assume that $K\{T_1, \dots, T_d\}/\p$ is finite over $K$. Then $K\{T_1, \dots, T_n\}$ is integral over $K$ because it is a $K$-subalgebra of $K\{T_1, \dots, T_d\}/\p$. Hence, $K\{T_1, \dots, T_n\}$ is a field by virtue of \cite[\textsection 9, Lemma 1]{Matsumura}. This implies $n=0$, so another application of \cite[\textsection 9, Lemma 1]{Matsumura} ensures that $K\{T_1, \dots, T_d\}/\p$ is a field, i.e., $\p$ is a maximal ideal.   

    Now we assume that $\p$ is a maximal ideal. Then \cite[\textsection 9, Lemma 1]{Matsumura} guarantees that $K\{T_1, \dots, T_n\}$ is a field. Hence $n=0$, so $K\{T_1, \dots, T_d\}/\p$ is finite over $K$.
\end{proof}

\begin{lemma}\label{lemma:dimension-henselian} Let $d\geq 0$ be an integer, and let $\n \subset K\{T_1, \dots, T_d\}$ be a maximal ideal. Then $\rm{ht}\, \n = d$. In particular, $\dim K\{T_1, \dots, T_d\}=d$ for any integer $d\geq 0$.
\end{lemma}
\begin{proof}
    Lemma~\ref{lemma:completion-faithfully-flat-appendix} implies that the natural morphism $K\{T_1, \dots, T_d\} \to K\langle T_1, \dots, T_d\rangle$ is faithfully flat. Therefore, \cite[\href{https://stacks.math.columbia.edu/tag/00OH}{Tag 00OH}]{stacks-project} and \cite[Prop.~2.2/17]{B} imply that $\dim K\{T_1, \dots, T_d\} \leq \dim K\langle T_1, \dots, T_d\rangle = d$. In particular, $\mathrm{ht}\, \n \leq d$.  

    Now denote by $f\colon K[T_1, \dots, T_d] \to K\{T_1, \dots, T_d\}$ the natural morphism and set $\widetilde{\n}\coloneqq f^{-1}(\n)$. Corollary~\ref{cor:maximal-ideal-henselian} implies that the morphism $K \to K[T_1, \dots, T_d]/\widetilde{\n}$ is a finite type, integral morphism. In particular, it is a finite morphism. Therefore, the polynomial analogue of Corollary~\ref{cor:maximal-ideal-henselian} implies that $\widetilde{\n}$ is a maximal ideal. Furthermore, $K[T_1, \dots, T_d]_{\widetilde{\n}} \to K\{T_1, \dots, T_d\}_{\n}$ is faithfully flat morphism. Therefore, \cite[\href{https://stacks.math.columbia.edu/tag/00OH}{Tag 00OH}]{stacks-project} and \cite[\href{https://stacks.math.columbia.edu/tag/00OP}{Tag 00OP}]{stacks-project} ensure that $\rm{ht}(\n) \geq \rm{ht}(\widetilde{\n})=d$. This finishes the proof. 
\end{proof}

\begin{cor}\label{cor:dimension-henselization-completion} Let $R$ be a finite type $\O_K$-algebra. 
\begin{enumerate}
    \item\label{cor:dimension-henselization-completion-1} $\dim R^\h_{(\varpi)}[\frac{1}{\varpi}] = \dim R^{\wedge}_{(\varpi)}[\frac{1}{\varpi}]$;
    \item\label{cor:dimension-henselization-completion-2} If $R^\h_{(\varpi)}[\frac{1}{\varpi}]$ is a domain, then $\dim \big(R^\h_{(\varpi)}[\frac{1}{\varpi}]\big) = \rm{trdeg}\big(\rm{Frac}(R^\h_{(\varpi)}[\frac{1}{\varpi}])/K\big)$. 
\end{enumerate}
\end{cor}
\begin{proof}
    First, Lemma~\ref{lemma:reduce-to-flat} implies that we can assume that $R$ is $\O_K$-flat. Then Lemma~\ref{lemma:henselian-noether-normalization} ensures that there is a finite injective morphism $\O_K\{T_1, \dots, T_d\} \hookrightarrow  R^\h_{(\varpi)}$ which remains injective modulo the maximal ideal $\m$. 
    
    To see (\ref{cor:dimension-henselization-completion-1}), we note that \cite[Th.~20]{Matsumura-algebra}, Lemma~\ref{lemma:dimension-henselian}, and \cite[Prop.~2.2/17]{B} imply that 
    \[
    \dim R^\h_{(\varpi)}\big[\frac{1}{\varpi}\big] = \dim K\{T_1, \dots, T_d\} = d = \dim K\langle T_1, \dots, T_d\rangle = \dim R^\wedge_{(\varpi)}\big[\frac{1}{\varpi}\big]. 
    \]

    To see $(\ref{cor:dimension-henselization-completion-2})$, we note that $\rm{trdeg}\big(\rm{Frac}(R^\h_{(\varpi)}[\frac{1}{\varpi}])/K\big) = \rm{trdeg}\big(\rm{Frac}(K\{T_1, \dots, T_d\})/K\big)$ and $\dim R^\h_{(\varpi)}[\frac{1}{\varpi}] = \dim K\{T_1, \dots, T_d\}=d$. Therefore, it suffices to show that $\rm{trdeg}\big(\rm{Frac}(K\{T_1, \dots, T_d\})/K\big) = d$. Since $K\{T_1, \dots, T_d\}$ is ind-\'etale over $K[T_1, \dots, T_d]$, we conclude that the field extension $K(T_1, \dots, T_d) \subset \rm{Frac}(K\{T_1, \dots, T_d\})$ is algebraic. Therefore, $\rm{trdeg}\big(\rm{Frac}(K\{T_1, \dots, T_d\})/K\big) = \rm{trdeg}\big(K(T_1, \dots, T_d)/K\big) = d$. 
\end{proof}

\subsection{Counterexamples}\label{appendix:counter-examples}

In this subsection, we provide three examples of non-algebraizable admissible affine formal schemes. In two of these examples, the affine schemes will be of pure (relative) dimension $1$ showing that it is essential to assume geometric reducedness in Corollary~\ref{cor:intro-partial-algebraization-2}. The other example will be geometrically reduced showing that some stronger smoothness assumption is necessary in Theorem~\ref{thm:intro-elkik-algebraization}.  

We start with the first example which works only over (some) non-perfect non-archimedean fields of equi-characteristic $p>0$. 

\begin{proposition}[{\cite[\textsection 3.3, p.~509]{Conrad99}}]\label{prop:counterexample-1-appendex} Let $K$ be a non-archimedean field of equi-characteristic $p>0$ such that $\dim_K K^{1/p}=\infty$. Let $\varpi \in \O_K$ be a pseudo-uniformizer and let $a_0, \dots, a_n, \dots {} \in \O_K$ be a countable sequence of elements whose $p$th roots generate an extension of $K$ of infinite degree, and such that $\lim_{n\to \infty} a_i=0$. Put $f\coloneqq \sum_{i\geq 0} a_i X^{ip}\in \O_K\langle X\rangle $ and $A_0\coloneqq \O_K\langle X, Y\rangle/(Y^p - f)$. Then there is no finite type $\O_K$-algebra $B$ with an $\O_K$-linear isomorphism $B^{\wedge}_{(\varpi)} \simeq A_0$. 
\end{proposition}

This example was communicated to us by Brian Conrad. 

\begin{proof}
    First, we show that $A_0 \wdh{\otimes}_{\O_K} \O_{K'}$ is reduced for any \emph{finite} field extension $K\subset K'$. For this, we notice that 
    \[
    A_0 \wdh{\otimes}_{\O_K} \O_{K'} \simeq \O_{K'}\langle X, Y\rangle/\big(Y^p -f (X)\big) \simeq \O_{K'}\langle X\rangle[Y]/\big(Y^p -f(X)\big)
    \]
    and put $Q\coloneqq \Frac\big(\O_{K'}\langle X\rangle \big)$. Then, for the purpose of proving reducedness of $A_0 \wdh{\otimes}_{\O_K} \O_{K'}$, it suffices to show that $Q[Y]/\big(Y^p-f(X)\big)$ is reduced. For this, it is enough to show that $f(X)$ is not a $p$th power in the field $Q$. Since $\O_{K'}\langle X\rangle$ is a normal ring, it suffices to show that $f(X)$ is not a $p$th power in $\O_{K'}\langle X\rangle$. This, in turn, follows from our choice of $f(X)$. \smallskip
    

    Now suppose there is a finite type $\O_K$-algebra $B$ such that $B^{\wedge}_{(\varpi)} \simeq A_0$. Then \cite[\href{https://stacks.math.columbia.edu/tag/030V}{Tag 030V}]{stacks-project} implies that there is a finite extension $K\subset K'$ such that, for any further field extension $K'\subset F$, the nilradical $\rm{nil}(B_{F})$ is isomorphic to $\rm{nil}(B_{K'}) \otimes_{K'} F$. Since $A_0\wdh{\otimes}_{\O_K} \O_{K'}$ is reduced, we can replace $K$ by $K'$, the $\O_K$-algebra $A_0$ by the $\O_{K'}$-algebra $A_0 \wdh{\otimes}_{\O_K} \O_{K'}$, and $B$ by $(B\otimes_{\O_K} \O_{K'})/\rm{nil}(B\otimes_{\O_K} \O_{K'})$ to assume that $B_K$ is geometrically reduced in the algebraic sense, i.e., $B_F=B\otimes_K F$ is reduced for any field extension $K\subset F$. \smallskip

    Lemma~\ref{lemma:completion-regular-generic-fiber} applied to a finite type $\O_{K^{1/p}}$-algebra $B\otimes_{\O_K} \O_{K^{1/p}}$ ensures that the morphism 
    \[
    \big(B\otimes_{\O_K} {\O_{K^{1/p}}}\big)\big[\frac{1}{\varpi}\big] \to \big(B\wdh{\otimes}_{\O_K} {\O_{K^{1/p}}}\big)\big[\frac{1}{\varpi}\big] \simeq \big(A_0\wdh{\otimes}_{\O_K} {\O_{K^{1/p}}}\big)\big[\frac{1}{\varpi}\big]
    \]
    is regular. Since $B_K$ is geometrically reduced, we conclude that $\big(B\otimes_{\O_K} {\O_{K^{1/p}}}\big)\big[\frac{1}{\varpi}\big] \simeq B_{K^{1/p}}$ is reduced. Therefore, \cite[\href{https://stacks.math.columbia.edu/tag/07QK}{Tag 07QK}]{stacks-project} guarantees that $\big(A_0\wdh{\otimes}_{\O_K} {\O_{K^{1/p}}}\big)\big[\frac{1}{\varpi}\big]$ is reduced as well. However, the ring $\big(A_0\wdh{\otimes}_{\O_K} {\O_{K^{1/p}}}\big)\big[\frac{1}{\varpi}\big]$ is not reduced by its very construction. Therefore, no such $B$ exists. 
\end{proof}

\begin{rmk} We note that Proposition~\ref{prop:counterexample-1-appendex} gives an example of a reduced flat, topologically finitely presented $\O_K$-algebra $A_0$ such that $A\coloneqq A_0\otimes_{\O_K} K$ is reduced and is of Krull dimension $1$, but the affine formal scheme $\Spf A_0$ is not algebraizable. In particular, the assumption that $\Spa(A, A^\circ)$ is \emph{geometrically} reduced cannot be dropped in Corollary~\ref{cor:intro-partial-algebraization-2}. We also note that, in this example, the $K$-algebra $A$ is geometrically reduced in the usual algebraic sense, i.e., $A \otimes_K F$ is reduced for any field extension $K\subset F$.   
\end{rmk}

\begin{rmk} In fact, for a non-archimedean field $K$ of equicharacteristic $p>0$, there is a non-algebraizable reduced topologically finite type flat affine relative curve $\Spf A_0$ over $\O_K$ if and only if $\dim_{K^p} K>p$. We do not justify this claim as we never need this result in this paper.  
\end{rmk}

To discuss the other two examples, we will need to use the notion of cross-ratio. Let $F$ be a field, and let $p_1, p_2, p_3, p_4\in \PP^1(F)$ be an ordered quadruple of distict points of the projective line with homogeneous coordinates $p_1=[X_1: X_2]$, $p_2=[Y_1: Y_2]$, $p_3=[Z_1:Z_2]$, and $p_4=[W_1:W_2]$. Then the \emph{cross-ratio} $(p_1, p_2; p_3, p_4)$ is defined to be 
\[
(p_1, p_2; p_3, p_4) = \frac{\det{  \begin{bmatrix}
    X_1 & W_1\\
    X_2 & W_2
  \end{bmatrix}
}}{\det{  \begin{bmatrix}
    Z_1 & W_1\\
    Z_2 & W_2
  \end{bmatrix}
}} \colon \frac{\det{  \begin{bmatrix}
    X_1 & Y_1\\
    X_2 & Y_2
  \end{bmatrix}
}}{\det{  \begin{bmatrix}
    Z_1 & Y_1\\
    Z_2 & Y_2
  \end{bmatrix}
}}.
\]
One checks that this expression does not depend on the choice of coordinates. 

\begin{proposition}\label{prop:counterexample-2-appendex} Let $K$ be a non-archimedean field, let $\varpi\in \O_K$ be a pseudo-uniformizer, and let $f\in K\langle T\rangle$ be an element which is transcendental over $K(T)$ (e.g., $f(T) = \sum_{n\geq 0} \varpi^n T^{n!}$). Set 
\[
A_0\coloneqq \O_K\langle T, X, Y\rangle/\big(X\cdot Y\cdot(X-Y)\cdot (X-TY)\cdot(X-f(T)Y)\big).
\]
Then there is no finite type $\O_K$-algebra $B$ with an $\O_K$-linear isomorphism $B^{\wedge}_{(\varpi)} \simeq A_0$.
\end{proposition}
\begin{proof}
    For brevity, we set $g_1(T, X, Y)\coloneqq X$, $g_2(T, X, Y)\coloneqq Y$, $g_3(T, X, Y) \coloneqq X-Y$, $g_4(T, X, Y) \coloneqq X-TY$, and $g_5(T, X, Y)\coloneqq X-f(T)Y$. We also set $A\coloneqq A_0[\frac{1}{\varpi}]$ and $Z\coloneqq \Spec A$. Then $Z$ has five regular irreducible components $Z_i \coloneqq  \rm{V}(g_i)\subset Z$ of pure dimension $2$. Their scheme-theoretic intersection $W\coloneqq \cap_{i=1}^5 Z_i = \rm{V}(X, Y) = \Spec K\langle T\rangle \subset Z$ is a regular irreducible scheme of pure dimension $1$. We denote by $\zeta\in W$ its generic point and by $F_\zeta\coloneqq \rm{Frac}\big(\O(W)\big)$ the field of rational functions on $W$.  

    We denote by $J \coloneqq \O(Z)=A \twoheadrightarrow \O(W)$ the ideal of $W$. Then one checks directly that the $\O(W)$-module $J/J^2 \cong \ov{X}\cdot A/J \oplus \ov{Y}\cdot A/J$ is a free vector bundle of rank $2$ with a basis given by the residue classes $\ov{X}$ and $\ov{Y}$. We also denote by $J_{i}\coloneqq \O(Z)=A \to \O(Z_i)$ the ideals of $Z_i$ for $i=1, \dots, 5$.  
    
    Then $\ell_i\coloneqq \rm{Im}(J_{i} \to J/J^2)_{\zeta}$ is a line in the two-dimensional $F_\zeta$-vector space $V_\zeta\coloneqq (J/J^2)_{\zeta} \cong \ov{X}\cdot F_\zeta \oplus \ov{Y}\cdot F_\zeta$. In particular, each line $\ell_i$ defines a point $[\ell_i]\in \PP(V_\zeta) = \PP^1(F_\zeta)$. Now let $K\subset L\subset F_\zeta$ be the smallest subfield of $F_\zeta$ which contains $K$ and all possible cross-ratios of ordered quadruples of disticts points among $\{[\ell_1], [\ell_2], [\ell_3], [\ell_4], [\ell_5]\}$. Then, using the homogenenuous coordinates on $\PP(V_\zeta)$ coming from the trivialization $V_\zeta = \ov{X}\cdot F_\zeta \oplus \ov{Y}\cdot F_\zeta$, one easily checks that the extension $K\subset L$ is generated by $T$ and $f(T)$. In particular, we conclude that 
    \begin{equation}\label{eqn:transcendence-degree}
    \rm{trdeg}\big(L/K\big)=2.
    \end{equation}

    Now suppose there is a finite type $\O_K$-algebra $B$ such that $B^{\wedge}_{(\varpi)} \simeq A_0$. We set $\widetilde{Z} \coloneqq \Spec \big(B^{\h}_{(\varpi)}[\frac{1}{\varpi}]\big)$. Then Lemma~\ref{lemma:completion-faithfully-flat-appendix}, Corollary~\ref{cor:properties-after-completion}, and Corollary~\ref{cor:dimension-henselization-completion} imply that $\widetilde{Z}$ has five regular irreducible components $\widetilde{Z}_i$ of pure dimension $2$ and their scheme-theoretic intersection $\widetilde{W}$ is a regular irreducible scheme of pure dimension $1$. We denote by $\widetilde{\zeta} \in \widetilde{W}$ its generic point and by $F_{\widetilde{\zeta}}\coloneqq \rm{Frac}\big(\O(\widetilde{W})\big)$ the field of rational functions on $\widetilde{W}$.  

    We also denote by $\widetilde{J}\coloneqq \O(\widetilde{Z}) \twoheadrightarrow \O(\widetilde{W})$ and $\widetilde{J}_i\coloneqq \O(\widetilde{Z}) \twoheadrightarrow \O(\widetilde{Z}_i)$ the corresponding ideals of $\widetilde{W}$ and $\widetilde{Z}_i$ for $i=1, \dots, 5$. Lemma~\ref{lemma:completion-faithfully-flat-appendix} implies that $\O(\widetilde{W}) \to \O(W)$ and $\O(\widetilde{Z}_i) \to \O(Z_i)$ are faithfully flat and that $\widetilde{J}\otimes_{\O(\widetilde{Z})} \O(Z) \simeq J$ and $\widetilde{J}_i \otimes_{\O(\widetilde{Z}_i)} \O(Z_i) \simeq J_i$ for each $i=1, \dots, 5$. In particular, we conclude that $\widetilde{J}/\widetilde{J}^2$ is a vector bundle of rank $2$, the image $\widetilde{\ell}_i\coloneqq \rm{Im}(\widetilde{J}_{i} \to \widetilde{J}/\widetilde{J}^2)_{\widetilde{\zeta}}$ is a line in the two-dimensional $F_{\widetilde{\zeta}}$-vector space $V_{\widetilde{\zeta}}\coloneqq (\widetilde{J}/\widetilde{J}^2)_{\widetilde{\zeta}}$, and the natural map $\PP\big(V_{\widetilde{\zeta}}\big) \to \PP(V_\zeta)$ sends $[\widetilde{\ell}_i]$  to $[\ell_i]$ for each $i=1, \dots, 5$. This implies that the field $L$ defined above is equal to the the smallest subfield of $F_{\widetilde{\zeta}}$ which contains $K$ and all possible cross-ratios of ordered quadruples of disticts points among $\{[\widetilde{\ell}_1], [\widetilde{\ell}_2], [\widetilde{\ell}_3], [\widetilde{\ell}_4], [\widetilde{\ell}_5]\}$. In particular, we see that $L$ is a subfield of $F_{\widetilde{\zeta}}$. Thus, Corollary~\ref{cor:dimension-henselization-completion} implies that 
    \[
    \rm{trdeg}\big(L/K\big) \leq \rm{trdeg}\big(F_{\widetilde{\zeta}}/K\big) = 1. 
    \]
    This contradicts (\ref{eqn:transcendence-degree}). Hence, there is no finite type $\O_K$-algebra $B$ such that $B^{\wedge}_{(\varpi)} \simeq A_0$. 
\end{proof}

Using Proposition~\ref{prop:counterexample-2-appendex}, we also construct an example of a non-algebraizable non-reduced admissible affine formal scheme of relative dimension $1$.   
\begin{proposition}\label{prop:counterexample-3-appendex} Let $K$ be a non-archimedean field, let $\varpi\in \O_K$ be a pseudo-uniformizer, and let $f\in K\langle T\rangle$ be an element which is transcendental over $K(T)$ (e.g., $f(T) = \sum_{n\geq 0} \varpi^n T^{n!}$). Set 
\[
A_0\coloneqq \O_K\langle T, X, Y\rangle/\Big(X Y(X-Y) (X-TY)(X-f(T)Y), X^6, X^5Y, X^4Y^2, X^3Y^3, X^2Y^4, XY^5, Y^6\Big).
\]
Then there is no finite type $\O_K$-algebra $B$ with an $\O_K$-linear isomorphism $B^{\wedge}_{(\varpi)} \simeq A_0$.
\end{proposition}
\begin{proof}
    We denote by $N \coloneqq \rm{nil}(A_0) \subset A_0$ the nilradical of $A_0$ and by $\O_K\langle T\rangle \simeq A_{0, \red} \coloneqq A_0/N$ the reduced quotient of $A_0$. Then we see that $N/N^2$ is a free $A_{0, \red}$-module of rank two with a basis given by the residue classes $\ov{X}$ and $\ov{Y}$, i.e., $N/N^2\simeq \ov{X}\cdot A_{0, \red}\oplus \ov{Y}\cdot A_{0, \red}$. Now we consider the kernel $I\coloneqq \ker\big(\Sym^5_{A_{0, \red}}(N/N^2) \to N^5/N^6\big)$ and see that $I$ is an $A_{0, \red}$-module generated by the element $\ov{X}\ov{Y}(\ov{X}-\ov{Y})(\ov{X}-T\ov{Y})(\ov{X}-f(T)\ov{Y})$. Therefore, the $\O_K$-scheme 
    \[
    Z\coloneqq \rm{V}\Big(I\cdot \big(\bigoplus_{n\geq 0}\Sym^n_{A_{0, \red}} N/N^2\big)\Big) \subset \Spec \big(\bigoplus_{n\geq 0}\Sym^n_{A_{0, \red}} N/N^2\big) \simeq \Spec A_{0, \red}[\ov{X}, \ov{Y}] \simeq \Spec \O_K\langle T\rangle[\ov{X}, \ov{Y}]
    \]
    is isomorphic to $\Spec \O_K\langle T\rangle[\ov{X},\ov{Y}]/\big(\ov{X}\cdot\ov{Y}\cdot(\ov{X}-\ov{Y})\cdot(\ov{X}-T\ov{Y})\cdot (\ov{X}-f(T)\ov{Y})\big)$. We put $A_1\coloneqq \O(Z)$, then 
    \[
    (A_1)^{\wedge}_{(\varpi)} \simeq \O_K\langle T, \ov{X}, \ov{Y}\rangle/\big(\ov{X}\cdot\ov{Y}\cdot(\ov{X}-\ov{Y})\cdot(\ov{X}-T\ov{Y})\cdot (\ov{X}-f(T)\ov{Y})\big).
    \]

    Now suppose that there is a finite type $\O_K$-algebra $B$ such that $B^{\wedge}_{(\varpi)} \simeq A_0$. Lemma~\ref{lemma:reduce-to-flat} ensures that we can assume that $B$ is $\O_K$-flat. Let $\widetilde{N}\coloneqq \nil(B)$ be the nilradical of $B$ and put $\widetilde{I}\coloneqq \ker\big(\Sym^5_{B_{\red}}(\widetilde{N}/\widetilde{N}^2) \to \widetilde{N}^5/\widetilde{N}^6\big)$. We note that $\widetilde{N}$ is a finite $B$-module by the combination of \cite[Prop.~0.8.5.3, Prop.~0.8.5.8]{FujKato}, and \cite[Th.\,7.3.2]{FGK}. We set
    \[
    B_1\coloneqq \Big(\bigoplus_{n\geq 0} \Sym^n_{B_{\red}} \widetilde{N}/\widetilde{N}^2 \Big)/\Big(\widetilde{I}\cdot\big(\bigoplus_{n\geq 0} \Sym^n_{B_{\red}} \widetilde{N}/\widetilde{N}^2\big)\Big).
    \]
    By construction, the $\O_K$-algebra $B_1$ is of finite type.  
    
    Now Lemma~\ref{lemma:completion-faithfully-flat-appendix} and Corollary~\ref{cor:base-change-nilradical} imply that $\widetilde{N} \otimes_B A_0 \simeq N$ and that $\widetilde{I} \otimes_B A_0 \simeq I$. Therefore, we conclude that $(B_1)^{\wedge}_{(\varpi)} \simeq (A_1)^{\wedge}_{(\varpi)} \simeq \O_K\langle T, \ov{X}, \ov{Y}\rangle/\big(\ov{X}\cdot\ov{Y}\cdot (\ov{X}-\ov{Y})\cdot(\ov{X}-T\ov{Y})\cdot (\ov{X}-f(T)\ov{Y})\big)$. This contradicts Proposition~\ref{prop:counterexample-2-appendex}. 
\end{proof}

\begin{rmk} We note that Proposition~\ref{prop:counterexample-3-appendex} gives an example of a flat, topologically finitely presented $\O_K$-algebra $A_0$ such that $A\coloneqq A_0\otimes_{\O_K} K$ is of Krull dimension $1$, but the affine formal scheme $\Spf A_0$ is not algebraizable. In particular, the assumption that $\Spa(A, A^\circ)$ is reduced cannot be dropped in Corollary~\ref{cor:intro-partial-algebraization-2}. 
\end{rmk}

\bibliography{main}
\end{document}